\documentclass[11pt,a4paper]{article}
\usepackage[utf8]{inputenc}
\usepackage[english]{babel}
\usepackage[T1]{fontenc}
\usepackage{amsmath,amsfonts,amssymb,amsthm}
\usepackage{lmodern}
\usepackage{graphicx}
\usepackage{setspace}
\usepackage{hyperref}
\usepackage{tikz}
\usepackage{multirow}
\usepackage{color}
\usepackage[caption=false,lofdepth,lotdepth]{subfig}
\usepackage{bbding}
\usepackage{pifont}
\usepackage{algorithm}
\usepackage{algorithmicx}
\usepackage{algpseudocode}
\algnewcommand\algorithmicinput{\textbf{Input:}}
\algnewcommand\Input{\item[\algorithmicinput]}
\algnewcommand\algorithmicoutput{\textbf{Output:}}
\algnewcommand\Output{\item[\algorithmicoutput]}

\usepackage[left=2cm,right=2cm,top=2cm,bottom=2cm]{geometry}
\author{Guillaume Olikier\footnotemark[2] \and Kyle A. Gallivan\footnotemark[3] \and P.-A. Absil\footnotemark[2]}
\title{First-order optimization on stratified sets\footnotemark[1]}

\newcommand{\N}{\mathbb{N}}

\newcommand{\R}{\mathbb{R}}

\DeclareMathOperator{\dom}{dom}
\DeclareMathOperator{\im}{im}
\newcommand{\dist}{d}
\newcommand{\ball}{\mathrm{B}}
\newcommand{\norm}[1]{\Vert #1 \Vert}
\newcommand{\ip}[2]{\langle #1, #2 \rangle}
\DeclareMathOperator*{\argmin}{argmin}
\DeclareMathOperator*{\lip}{Lip}
\newcommand{\dd}{\mathrm{d}}
\DeclareMathOperator{\s}{s}

\DeclareMathOperator*{\inlim}{\underline{Lim}}
\DeclareMathOperator*{\outlim}{\overline{Lim}}
\DeclareMathOperator*{\setlim}{Lim}
\newcommand{\setmapsto}{\multimap}

\newcommand{\proj}[2]{P_{#1}(#2)}

\makeatletter
\def\widebreve{\mathpalette\wide@breve}
\def\wide@breve#1#2{\sbox\z@{$#1#2$}%
     \mathop{\vbox{\m@th\ialign{##\crcr
\kern0.08em\brevefill#1{0.8\wd\z@}\crcr\noalign{\nointerlineskip}%
                    $\hss#1#2\hss$\crcr}}}\limits}
\def\brevefill#1#2{$\m@th\sbox\tw@{$#1($}%
  \hss\resizebox{#2}{\wd\tw@}{\rotatebox[origin=c]{90}{\upshape(}}\hss$}
\makeatletter

\newcommand{\gencone}[4]{{#1}_{#2}^{#4}(#3)} 
\newcommand{\tancone}[2]{\gencone{T}{#1}{#2}{}} 
\newcommand{\sectancone}[3]{T_{#1}^2(#2|#3)} 
\newcommand{\restancone}[2]{\gencone{\widebreve{T}}{#1}{#2}{}} 
\newcommand{\regnorcone}[2]{\gencone{\widehat{N}}{#1}{#2}{}} 
\newcommand{\norcone}[2]{\gencone{N}{#1}{#2}{}} 
\newcommand{\connorcone}[2]{\gencone{\overline{N}}{#1}{#2}{}} 

\newcommand{\oshort}[1]{\mkern 0.9mu\overline{\mkern-0.9mu#1\mkern-0.9mu}\mkern 0.9mu}
\newcommand{\ushort}[1]{\mkern 0.9mu\underline{\mkern-0.9mu#1\mkern-0.9mu}\mkern 0.9mu}

\newcommand{\sparse}[2]{\R_{\le #2}^{#1}}
\newcommand{\FixedSparsity}[2]{\R_{#2}^{#1}}
\newcommand{\StrictSparsity}[2]{\R_{< #2}^{#1}}
\DeclareMathOperator{\supp}{supp}

\newcommand{\tp}{\top}
\DeclareMathOperator{\rank}{rank}
\DeclareMathOperator{\tr}{tr}
\newcommand{\st}{\mathrm{St}}
\DeclareMathOperator{\diag}{diag}

\newcommand{\pgd}{\mathrm{PGD}}
\newcommand{\ppgd}{\mathrm{P}^2\mathrm{GD}}
\newcommand{\ppgdr}{\mathrm{P}^2\mathrm{GDR}}
\newcommand{\rfd}{\mathrm{RFD}}
\newcommand{\rfdr}{\mathrm{RFDR}}

\newtheorem{theorem}{Theorem}[section]
\newtheorem{proposition}[theorem]{Proposition}
\newtheorem{lemma}[theorem]{Lemma}
\newtheorem{corollary}[theorem]{Corollary}
\theoremstyle{definition}

\newtheorem{assumption}[theorem]{Assumption}

\usepackage{chngcntr}
\counterwithin{algorithm}{section}
\counterwithin{table}{section}
\counterwithin{figure}{section}

\begin{document}
\renewcommand{\thefootnote}{\fnsymbol{footnote}}
\footnotetext[1]{This work was supported by the Fonds de la Recherche Scientifique -- FNRS and the Fonds Wetenschappelijk Onderzoek -- Vlaanderen under EOS Project no 30468160. K. A. Gallivan is partially supported by the U.S. National Science Foundation under grant CIBR 1934157.}
\footnotetext[2]{ICTEAM Institute, UCLouvain, Avenue Georges Lema\^{\i}tre 4, 1348 Louvain-la-Neuve, Belgium (\href{mailto:guillaume.olikier@uclouvain.be}{\nolinkurl{guillaume.olikier@uclouvain.be}}, \href{mailto:pa.absil@uclouvain.be}{\nolinkurl{pa.absil@uclouvain.be}}).}
\footnotetext[3]{Department of Mathematics, Florida State University, 
1017 Academic Way, Tallahassee, FL 32306-4510, USA (\href{mailto:kgallivan@fsu.edu}{\nolinkurl{kgallivan@fsu.edu}}).}
\renewcommand{\thefootnote}{\arabic{footnote}}

\maketitle

\begin{abstract}
We consider the problem of minimizing a differentiable function with locally Lipschitz continuous gradient on a stratified set and present a first-order algorithm designed to find a stationary point of that problem. Our assumptions on the stratified set are satisfied notably by the determinantal variety (i.e., matrices of bounded rank), its intersection with the cone of positive-semidefinite matrices, and the set of nonnegative sparse vectors. The iteration map of the proposed algorithm applies a step of projected-projected gradient descent with backtracking line search, as proposed by Schneider and Uschmajew (2015), to its input but also to a projection of the input onto each of the lower strata to which it is considered close, and outputs a point among those thereby produced that maximally reduces the cost function. Under our assumptions on the stratified set, we prove that this algorithm produces a sequence whose accumulation points are stationary, and therefore does not follow the so-called apocalypses described by Levin, Kileel, and Boumal (2022). We illustrate the apocalypse-free property of our method through a numerical experiment on the determinantal variety.
\medskip

\medskip

\end{abstract}

\section{Introduction}
\label{sec:Introduction}
Given a Euclidean vector space $\mathcal{E}$ with inner product and induced norm respectively denoted by $\ip{\cdot}{\cdot}$ and $\norm{\cdot}$, a differentiable function $f : \mathcal{E} \to \R$ with locally Lipschitz continuous gradient, and a nonempty closed subset $C$ of $\mathcal{E}$, we consider the problem
\begin{equation}
\label{eq:OptiProblem}
\min_{x \in C} f(x)
\end{equation}
of minimizing $f$ on $C$.
In general, problem~\eqref{eq:OptiProblem} is intractable and one is thus content with finding a stationary point of that problem, i.e., a point satisfying a first-order necessary condition to be a local minimizer of $f|_C$. Every definition of stationarity is based on a tangent or normal cone. Classic notions of tangent or normal cone include the tangent cone, the regular normal cone, the normal cone, and the Clarke normal cone; they are reviewed in Section~\ref{subsec:TangentNormalConesGeometricDerivability} based on \cite[Chapter~6]{RockafellarWets}. Each of these notions of normal cone yields a definition of stationarity. We review them briefly here and refer to \cite{LiSoMa2020, HosseiniLukeUschmajew2019} for more details.

A point $x \in C$ is said to be \emph{stationary} for~\eqref{eq:OptiProblem} if it satisfies one of the following equivalent conditions:
\begin{enumerate}
\item $\ip{\nabla f(x)}{v} \ge 0$ for all $v \in \tancone{C}{x}$, where $\tancone{C}{x}$ denotes the tangent cone to $C$ at $x$;
\item $-\nabla f(x) \in \regnorcone{C}{x}$, where $\regnorcone{C}{x}$ denotes the regular normal cone to $C$ at $x$;
\item $\s(x; f, C) = 0$, where the function
\begin{equation}
\label{eq:StationarityMeasure}
\s(\cdot; f, C) :
C \to \R :
x \mapsto \norm{\proj{\tancone{C}{x}}{-\nabla f(x)}},
\end{equation}
called the \emph{stationarity measure} of~\eqref{eq:OptiProblem}, returns the norm of any projection of $-\nabla f(x)$ onto $\tancone{C}{x}$.
\end{enumerate}
The point $x$ is said to be \emph{Mordukhovich stationary} for~\eqref{eq:OptiProblem} if $-\nabla f(x) \in \norcone{C}{x}$, where $\norcone{C}{x}$ denotes the normal cone to $C$ at $x$. The inclusion $\regnorcone{C}{x} \subseteq \norcone{C}{x}$ always holds, and $C$ is said to be \emph{Clarke regular} at $x$ if $\regnorcone{C}{x} = \norcone{C}{x}$. Thus, the stationarity of $x$ implies the Mordukhovich stationarity of $x$, and the two conditions are equivalent if and only if $C$ is Clarke regular at $x$.
The point $x$ is said to be \emph{Clarke stationary} for~\eqref{eq:OptiProblem} if $-\nabla f(x) \in \connorcone{C}{x}$, where $\connorcone{C}{x}$ denotes the Clarke normal cone to $C$ at $x$ defined as the closure of the convex hull of $\norcone{C}{x}$.
If $x \in C$ is a local minimizer of $f|_C$, then $x$ is stationary for~\eqref{eq:OptiProblem}, hence Mordukhovich stationary for~\eqref{eq:OptiProblem}, and hence Clarke stationary for~\eqref{eq:OptiProblem}. The stationarity of a point depends only on $f|_C$ since, by \cite[Lemmas~A.7 and A.8]{LevinKileelBoumal2022}, the correspondence
\begin{equation*}
C \setmapsto \mathcal{E} : x \mapsto \proj{\tancone{C}{x}}{-\nabla f(x)}
\end{equation*}
depends on $f$ only through $f|_C$. In contrast, the Mordukhovich stationarity depends on the values taken by $f$ outside $C$.

To the best of our knowledge, without further assumptions, the algorithm in the literature with the strongest convergence guarantee for problem~\eqref{eq:OptiProblem} is the projected gradient descent proposed in \cite[Algorithm~3.1]{JiaEtAl2022} and dubbed $\pgd$ in \cite[\S 1]{LevinKileelBoumal2022}. Given $x \in C$ as input, the iteration map of $\pgd$ performs a projected line search along the direction of $-\nabla f(x)$, i.e., computes a point in $P_C(x-\alpha\nabla f(x))$ for decreasing values of $\alpha \in (0, \infty)$ until an Armijo condition is satisfied. By \cite[Theorem~3.4]{JiaEtAl2022}, $\pgd$ produces a sequence whose accumulation points are Mordukhovich stationary; it is an open question whether these accumulation points can fail to be stationary. Furthermore, as pointed out in \cite[\S 1]{LevinKileelBoumal2022}, a sequence produced by $\pgd$ generally depends on the values taken by $f$ on $\mathcal{E} \setminus C$ which is, at least conceptually, unsatisfying.

A frequently encountered obstacle against guaranteeing convergence to stationary points of~\eqref{eq:OptiProblem} is the possible presence in $C$ of so-called apocalyptic points. By \cite[Definition~2.7]{LevinKileelBoumal2022}, a point $x \in C$ is said to be \emph{apocalyptic} if there exist a sequence $(x_i)_{i \in \N}$ in $C$ converging to $x$ and a continuously differentiable function $\phi : \mathcal{E} \to \R$ such that $\lim_{i \to \infty} \s(x_i; \phi, C) = 0$ whereas $\s(x; \phi, C) > 0$. Such a triplet $(x, (x_i)_{i \in \N}, \phi)$ is called an \emph{apocalypse}. By \cite[Corollary~2.15]{LevinKileelBoumal2022}, if $x \in C$ is apocalyptic, then $C$ is not Clarke regular at $x$. \emph{Apocalyptic sets}, i.e., sets that have at least one apocalyptic point, include:
\begin{enumerate}
\item the \emph{determinantal variety} \cite[Lecture~9]{Harris}
\begin{equation}
\label{eq:RealDeterminantalVariety}
\R_{\le r}^{m \times n} := \{X \in \R^{m \times n} \mid \rank X \le r\},
\end{equation}
$m$, $n$, and $r$ being positive integers such that $r < \min\{m,n\}$;

\item the closed cone
\begin{equation}
\label{eq:PSDconeBoundedRank}
\mathrm{S}_{\le r}^+(n) := \{X \in \R_{\le r}^{n \times n} \mid X^\tp = X,\, X \succeq 0\},
\end{equation}
$n$ and $r$ being positive integers such that $r < n$, of order-$n$ real symmetric positive-semidefinite matrices of rank at most $r$ ;

\item the closed cone of nonnegative sparse vectors, specifically $\sparse{n}{s} \cap \R_+^n$, $n$ and $s$ being positive integers such that $s < n$, where $\sparse{n}{s}$ is the set of $s$-sparse vectors of $\R^n$, i.e., those having at most $s$ nonzero components, and $\R_+^n$ is the nonnegative orthant of $\R^n$.
\end{enumerate}
Problem~\eqref{eq:OptiProblem} with $C$ one of these three sets appears in numerous applications; see Section~\ref{sec:ExamplesStratifiedSetsSatisfyingMainAssumption}.

In this paper, we propose a first-order optimization algorithm (Algorithm~\ref{algo:P2GDR}), called $\ppgdr$, that produces a sequence whose accumulation points are stationary for~\eqref{eq:OptiProblem} (see Theorem~\ref{thm:P2GDRPolakConvergence}) under assumptions on $C$ (Assumption~\ref{assumption:Stratification}) that apply to the three apocalyptic sets listed (see Section~\ref{sec:ExamplesStratifiedSetsSatisfyingMainAssumption}). For a high-level description of the algorithmic strategy underlying the $\ppgdr$ algorithm, see Section~\ref{subsec:OverviewProposedAlgorithm}.

When $C = \R_{\le r}^{m \times n}$, the proposed $\ppgdr$ algorithm competes against two algorithms known to accumulate at stationary points of~\eqref{eq:OptiProblem}: the second-order method given in \cite[Algorithm~1]{LevinKileelBoumal2022} and the first-order method given in \cite[Algorithm~3]{OlikierAbsil2022RFDR} and dubbed $\rfdr$, which are reviewed in Sections~\ref{subsec:OptimizationThroughSmoothLift} and \ref{subsec:RFDRreview}, respectively. These three algorithms and others are compared based on the computational cost per iteration and the convergence guarantees in Section~\ref{subsubsec:ComparisonSixOptimizationAlgorithmsRealDeterminantalVariety}, and numerically on the instance of~\eqref{eq:OptiProblem} from \cite[\S 2.2]{LevinKileelBoumal2022} in Section~\ref{subsubsec:LKB22instance}. A concise overview of the algorithms on $C = \R_{\le r}^{m \times n}$ and their properties can be found in Tables~\ref{tab:IndexAlgorithmsRealDeterminantalVariety}--\ref{tab:ConvergenceAlgorithmsRealDeterminantalVariety}. Numerical experiments indicate that $\ppgdr$ converges faster than \cite[Algorithm~1]{LevinKileelBoumal2022} (see Section~\ref{subsubsec:LKB22instance}) and $\rfdr$ (see Section~\ref{subsec:RFDRreview}).

When $C = \mathrm{S}_{\le r}^+(n)$, \cite[Algorithm~1]{LevinKileelBoumal2022} is the only algorithm known to accumulate at stationary points of~\eqref{eq:OptiProblem} (provided that one can find a suitable hook which, to our knowledge, has not been done yet explicitly in the literature). Indeed, $\rfdr$ does not seem to easily extend to $\mathrm{S}_{\le r}^+(n)$, hence $\ppgdr$ (and his variant using Algorithm~\ref{algo:P2GDRmapPSDconeBoundedRank} instead of Algorithm~\ref{algo:P2GDRmap} in line~\ref{algo:P2GDR:P2GDRmap}) has no competitor in the realm of first-order optimization methods on $\mathrm{S}_{\le r}^+(n)$ that accumulate at stationary points.

When $C = \sparse{n}{s} \cap \R_+^n$ and $f(x) = \norm{Ax-b}^2$ with $A \in \R^{q \times n}$ and $b \in \R^q$, problem~\eqref{eq:OptiProblem} is known as the $s$-sparse nonnegative least squares problem and can be solved exactly; see \cite{NadisicCohenVandaeleGillis} and the references therein. However, we are not aware of algorithms designed to address other cost functions on that set.

This paper gathers, expands, and generalizes results of the technical reports \cite{OlikierGallivanAbsil2022} and \cite{OlikierGallivanAbsil2022Comparison}. $\ppgdr$ on $\R_{\le r}^{m \times n}$ (Algorithm~\ref{algo:P2GDR} using Algorithm~\ref{algo:P2GDRmapRealDeterminantalVariety} in line~\ref{algo:P2GDR:P2GDRmap}) was proposed in \cite{OlikierGallivanAbsil2022} in response to a question raised in \cite[\S 4]{LevinKileelBoumal2021}: ``Is there an algorithm running directly on $\R_{\le r}^{m \times n}$ that only uses first-order information about the cost function and which is guaranteed to converge to a stationary point?''
The proposed $\ppgdr$ algorithm (Algorithm~\ref{algo:P2GDR}) answers positively an open question raised in \cite[\S 4]{LevinKileelBoumal2022}: ``Is there an algorithm running directly on a general class of nonsmooth sets including $\R_{\le r}^{m \times n}$ that only uses first-order information about the cost function, and which is guaranteed to converge to a stationary point?''

This paper is organized as follows. In Section~\ref{sec:AssumptionsFeasibleSetOverviewProposedAlgorithm}, we specify the assumptions on $C$ (Assumption~\ref{assumption:Stratification}) and give an overview of the proposed $\ppgdr$ algorithm. In Section~\ref{sec:PriorWork}, we review prior work on problem~\eqref{eq:OptiProblem}. In Section~\ref{sec:Preliminaries}, we introduce the background material needed to analyze the behavior of the algorithms. In Section~\ref{sec:ProposedAlgorithmConvergenceAnalysis}, we define $\ppgdr$ and analyze its convergence properties under Assumption~\ref{assumption:Stratification}. In Section~\ref{sec:ExamplesStratifiedSetsSatisfyingMainAssumption}, we study the three apocalyptic sets listed and prove that they satisfy Assumption~\ref{assumption:Stratification} (see Theorem~\ref{thm:ExamplesStratifiedSetsSatisfyingMainAssumption}). In Section~\ref{sec:ComplementaryResults}, we propose complementary results that are not needed for the proofs of the main theorems (Theorems~\ref{thm:P2GDRPolakConvergence} and \ref{thm:ExamplesStratifiedSetsSatisfyingMainAssumption}) but are of interest in the context of this work, notably because some of them relate Assumption~\ref{assumption:Stratification} to known concepts of variational analysis or stratification theory.
Section~\ref{sec:Conclusion} contains concluding remarks, and Appendix~\ref{sec:GradientHessianRealValuedFunctionOnHilbertSpace} basic material on the gradient and the Hessian of a real-valued function on a Hilbert space.

\section{Assumptions on the feasible set $C$ and overview of the proposed algorithm}
\label{sec:AssumptionsFeasibleSetOverviewProposedAlgorithm}
In this section, we introduce Assumption~\ref{assumption:Stratification} and give an overview of the $\ppgdr$ algorithm (Algorithm~\ref{algo:P2GDR}). Based on Assumption~\ref{assumption:Stratification}, we prove that $\ppgdr$ produces a sequence whose accumulation points are stationary for~\eqref{eq:OptiProblem} (see Theorem~\ref{thm:P2GDRPolakConvergence}).

\subsection{Assumptions on the feasible set $C$}
\label{subsec:AssumptionsFeasibleSet}
The closure and the boundary of a subset $S$ of $\mathcal{E}$ are respectively denoted by $\overline{S}$ and $\partial S$. The distance from $x \in \mathcal{E}$ to a nonempty subset $S$ of $\mathcal{E}$ is $\dist(x, S) := \inf_{y \in S} \norm{x-y}$. For every $x \in \mathcal{E}$ and every $\rho \in (0,\infty)$, $\ball(x,\rho) := \{y \in \mathcal{E} \mid \norm{x-y} < \rho\}$ and $\ball[x,\rho] := \{y \in \mathcal{E} \mid \norm{x-y} \le \rho\}$ are respectively the open and closed balls of center $x$ and radius $\rho$ in $\mathcal{E}$.

\begin{assumption}
\label{assumption:GlobalSecondOrderUpperBoundDistanceFromTangentLine}
For all $x \in C$,
\begin{equation*}
u(x) := \sup_{v \in \tancone{C}{x} \setminus \{0\}} \frac{\dist(x+v, C)}{\norm{v}^2} < \infty.
\end{equation*}
\end{assumption}

We relate Assumption~\ref{assumption:GlobalSecondOrderUpperBoundDistanceFromTangentLine} to the concept of parabolic derivability in Proposition~\ref{prop:LocalSecondOrderUpperBoundDistanceFromTangentLineParabolicDerivability}.
The concept of continuity for a correspondence which appears in Assumption~\ref{assumption:Stratification} is reviewed in Section~\ref{subsec:InnerOuterLimitsContinuityCorrespondences}.

\begin{assumption}
\label{assumption:Stratification}
The set $C$ satisfies the following conditions:
\begin{enumerate}
\item there exist a positive integer $p$ and nonempty smooth submanifolds $S_0, \dots, S_p$ of $\mathcal{E}$ contained in $C$ such that:
\begin{enumerate}
\item for all $i, j \in \{0, \dots, p\}$, $i \ne j$ implies $S_i \cap S_j =\emptyset$;
\item $\overline{S_p} = C$ and, for all $i \in \{0, \dots, p\}$, $\overline{S_i} = \bigcup_{j=0}^i S_j$;
\item if $p \ge 2$, then, for all $i \in \{2, \dots, p\}$, all $x \in S_i$, and all $j \in \{1, \dots, i-1\}$, $\dist(x, S_j) < \dist(x, S_{j-1})$;
\end{enumerate}
\item Assumption~\ref{assumption:GlobalSecondOrderUpperBoundDistanceFromTangentLine} holds and, for every $i \in \{0, \dots, p\}$, $u|_{S_i}$ is locally bounded;
\item for every $i \in \{0, \dots, p\}$, $\tancone{C}{\cdot}$ is continuous on $S_i$ relative to $S_i$ (see the definition in Section~\ref{subsec:InnerOuterLimitsContinuityCorrespondences}).
\end{enumerate}
\end{assumption}

Assumption~\ref{assumption:Stratification} is related to several important observations.
First, every real algebraic variety in $\mathcal{E}$ can be partitioned into finitely many smooth submanifolds of $\mathcal{E}$ \cite{Whitney1957} and therefore satisfies condition~1(a).
Second, we require $p \ge 1$ because, if $p = 0$, then $C$ is a closed smooth submanifold of $\mathcal{E}$ and is thus Clarke regular by \cite[Example~6.8]{RockafellarWets}, and $\ppgdr$ reduces to the Riemannian gradient descent (a particular case of \cite[Algorithm~1]{AbsilMahonySepulchre}) which accumulates at stationary points of~\eqref{eq:OptiProblem} as proven in \cite[\S 4.3.3]{AbsilMahonySepulchre}.
Third, condition~1(b) implies that, for every $x \in S_p$, $d(x, S_{p-1}) > 0$. Therefore, $C \cap \ball(x, \dist(x, S_{p-1})) = S_p \cap \ball(x, \dist(x, S_{p-1}))$. Thus, $C$ is locally a smooth submanifold of $\mathcal{E}$ around $x \in S_p$. Therefore, by \cite[Example~6.8]{RockafellarWets}, the tangent cone $\tancone{C}{x}$ equals the tangent space $\tancone{S_p}{x}$, and the normal cones $\regnorcone{C}{x}$, $\norcone{C}{x}$, and $\connorcone{C}{x}$ equal the normal space $\norcone{S_p}{x}$. In particular, $C$ is Clarke regular at $x \in S_p$ and hence, by \cite[Corollary~2.15]{LevinKileelBoumal2022}, $x$ is not apocalyptic.
Fourth, by Proposition~\ref{prop:FiniteStratificationsConditionFrontierVsAssumption}, conditions~1(a) and 1(b) imply that $\{S_0, \dots, S_p\}$ is a \emph{stratification} of $C$ satisfying the \emph{condition of the frontier} \cite[\S 5]{Mather}; therefore, $S_0, \dots, S_p$ are called the \emph{strata} of $\{S_0, \dots, S_p\}$, and $C$ is called a \emph{stratified set}.
Fifth, condition~1(c) is added to condition~1(b) only to ensure that, for every $i \in \{0, \dots, p-1\}$, every point in $C \setminus \overline{S_i}$ has a projection onto $S_i$ (Proposition~\ref{prop:Conditions1(b)(c)MainAssumption}).
Sixth, by Proposition~\ref{prop:ConvergenceStationarityMeasureZeroUpperStratumMordukhovichStationary}, if $C$ satisfies Assumption~\ref{assumption:Stratification} and $(x_i)_{i \in \N}$ is a sequence in $S_p$ such that $\lim_{i \to \infty} \s(x_i; f, C) = 0$, then every accumulation point of $(x_i)_{i \in \N}$ is Mordukhovich stationary for~\eqref{eq:OptiProblem}.
Seventh, by Corollary~\ref{coro:ContinuityTangentConeImpliesContinuityStationarityMeasure}, condition~3 implies that, for every $i \in \{0, \dots, p\}$, $\s(\cdot; f, C)|_{S_i}$ is continuous.
Eighth, by Theorem~\ref{thm:ExamplesStratifiedSetsSatisfyingMainAssumption}, if $C$ is one of the three apocalyptic sets listed in Section~\ref{sec:Introduction}, then:
\begin{enumerate}
\item Assumption~\ref{assumption:Stratification} is satisfied;
\item there exists $a \in (0, 1)$ such that, for all $x \in C$,
\begin{equation*}
u(x) = \sup_{v \in \tancone{C}{x} \setminus \{0\}} \frac{\dist(x+v, C)}{\norm{v}^2} \in [a\tilde{u}(x), \tilde{u}(x)],
\end{equation*}
where $\tilde{u}(x) := 0$ if $x \in S_0$ and $\tilde{u}(x) := \frac{1}{\dist(x, S_{i-1})}$ if $x \in S_i$ with $i \in \{1, \dots, p\}$;
\item $u$ is not locally bounded at any point of $C \setminus S_p$;
\item the set of apocalyptic points of $C$ is $C \setminus S_p$;
\item for all $i \in \{0, \dots, p\}$ and all $x \in S_i$, $\connorcone{C}{x} = \norcone{S_i}{x}$.
\end{enumerate}
The fourth and fifth statements respectively imply that, if $C$ is one of the three sets, then $\s(\cdot; f, C)$ is not necessarily lower semicontinuous at a point of $C \setminus S_p$, and a point $x \in S_i$ with $i \in \{0, \dots, p\}$ is Clarke stationary for~\eqref{eq:OptiProblem} if and only if $x$ is stationary for the problem of minimizing $f$ on $\bigcup_{j=0}^i S_j$.

\subsection{Overview of the proposed algorithm}
\label{subsec:OverviewProposedAlgorithm}
In this section, we give an overview of the $\ppgdr$ algorithm (Algorithm~\ref{algo:P2GDR}) which we introduce in Section~\ref{sec:ProposedAlgorithmConvergenceAnalysis}. The iteration map of $\ppgdr$ (Algorithm~\ref{algo:P2GDRmap}), called the $\ppgdr$ map, uses the $\ppgd$ map (Algorithm~\ref{algo:P2GDmap}) as a subroutine. The $\ppgd$ map essentially corresponds to the iteration map of \cite[Algorithm~3]{SchneiderUschmajew2015}, dubbed $\ppgd$ in \cite[\S 1]{LevinKileelBoumal2022}, except that it is defined on any set $C$ satisfying Assumption~\ref{assumption:GlobalSecondOrderUpperBoundDistanceFromTangentLine}, and not only on $\R_{\le r}^{m \times n}$. The name $\ppgd$ comes from the fact that each iteration involves two projections: given $x \in C$ as input, the $\ppgd$ map performs a projected line search along a direction $g \in \proj{\tancone{C}{x}}{-\nabla f(x)}$, i.e., computes a point in $P_C(x+\alpha g)$ for decreasing values of $\alpha \in (0, \infty)$ until an Armijo condition is satisfied. $\ppgd$ has at least three desirable properties that $\pgd$ does not have.
First, each sequence produced by $\ppgd$ depends on $f$ only through $f|_C$ by \cite[Lemmas~A.7 and A.8]{LevinKileelBoumal2022}. Second, as shown in Section~\ref{subsec:P2GDRmap}, if Assumption~\ref{assumption:Stratification} holds, then $\ppgd$ reduces to the Riemannian gradient descent on the upper stratum $S_p$ of $C$. Third, for certain sets $C$, the fact that the search direction is in the tangent cone to $C$ makes the projection onto $C$ easier to compute; this is the case if $C = \R_{\le r}^{m \times n}$ as shown in \cite[\S 3]{Vandereycken2013} and recalled in Section~\ref{subsubsec:TangentConeRealDeterminantalVariety}.

Unfortunately, $\ppgd$ can converge to a point that is Mordukhovich stationary for~\eqref{eq:OptiProblem} but not stationary. Indeed, for some instances of problem~\eqref{eq:OptiProblem}, $\ppgd$ follows an apocalypse, i.e., produces a sequence $(x_i)_{i \in \N}$ in $C$ converging to a point $x \in C$ such that $(x, (x_i)_{i \in \N}, f)$ is an apocalypse; an example is given in \cite[\S 2.2]{LevinKileelBoumal2022} for the case where $C = \R_{\le r}^{m \times n}$. If $C$ satisfies Assumption~\ref{assumption:Stratification}, then, by condition~3 and Corollary~\ref{coro:ContinuityTangentConeImpliesContinuityStationarityMeasure}, an apocalypse can occur only if the sequence $(x_i)_{i \in \N}$ has finitely many elements in the stratum containing its limit.
The $\ppgdr$ map is designed based on that fact. Given $x \in C$ as input, it applies the $\ppgd$ map to $x$ but also to a projection of $x$ onto each of the lower strata to which it is considered close, and outputs a point among those thereby produced that maximally decreases $f$.
The R in $\ppgdr$ comes from the fact that, on $\R_{\le r}^{m \times n}$, a projection onto a lower stratum is a rank reduction.

\section{Prior work}
\label{sec:PriorWork}
In this section, we review related work. In Section~\ref{subsec:StratifiedSetsOptimization}, we list several papers using the concept of stratified set in optimization. Then, in Section~\ref{subsec:OptimizationThroughSmoothLift}, we review \cite[Algorithm~1]{LevinKileelBoumal2022} and, more generally, the main results of \cite{LevinKileelBoumal2022SmoothLifts} which concern optimization through a smooth lift. Finally, in Section~\ref{subsec:RFDRreview}, we review the $\rfdr$ algorithm.

\subsection{Stratified sets in optimization}
\label{subsec:StratifiedSetsOptimization}
The concept of stratification has been used in optimization but, to the best of our knowledge, only in nonsmooth optimization with the goal of finding a Clarke stationary point.

First, several works including \cite{BolteEtAl, Ioffe, DavisEtAl, BianchiHachemSchechtman} concern the problem of minimizing a function whose graph admits a Whitney stratification.
They consider Clarke stationarity: see \cite[Definition~2]{BolteEtAl} for the unconstrained case and \cite[(6.2)]{DavisEtAl} for the constrained case.
For example, \cite[Theorem~6.2]{DavisEtAl} ensures, under suitable assumptions, that, almost surely, the proximal stochastic subgradient method produces a sequence whose accumulation points are Clarke stationary. It is an open question whether those points may fail to be stationary.

Second, the authors of \cite{HosseiniUschmajew2019} consider the problem of minimizing a locally Lipschitz continuous function on a real algebraic variety for which they propose a gradient sampling method. Being an algebraic variety, the feasible set admits a Whitney stratification \cite{Whitney1965}. By \cite[Theorem~3.3]{HosseiniUschmajew2019}, \cite[Algorithm~1]{HosseiniUschmajew2019} accumulates at Clarke stationary points. Again, it is not known whether those points can fail to be stationary.

\subsection{Optimization through a smooth lift}
\label{subsec:OptimizationThroughSmoothLift}
In \cite{LevinKileelBoumal2022SmoothLifts}, the authors study problem~\eqref{eq:OptiProblem} under the assumption that there exist a smooth manifold $\mathcal{M}$ and a smooth map $\varphi : \mathcal{M} \to \mathcal{E}$ such that $\varphi(\mathcal{M}) = C$; they call $\varphi$ a \emph{smooth lift} of $C$. Specifically, they investigate how desirable points of the problem
\begin{equation}
\label{eq:OptiSmoothLift}
\min_{y \in \mathcal{M}} (f \circ \varphi)(y)
\end{equation}
map to desirable points of~\eqref{eq:OptiProblem}:
\begin{itemize}
\item \cite[Theorem~2.8]{LevinKileelBoumal2022SmoothLifts} gives a necessary and sufficient condition on $\varphi$ for the property ``for all continuous $f$, if $y \in \mathcal{M}$ is a local minimum of~\eqref{eq:OptiSmoothLift}, then $\varphi(y)$ is a local minimum of~\eqref{eq:OptiProblem}'' to hold;
\item \cite[Theorem~2.10]{LevinKileelBoumal2022SmoothLifts} implies that, for every $y \in \mathcal{M}$, if $\tancone{C}{\varphi(y)}$ is not a linear subspace of $\mathcal{E}$, then there exists a differentiable $f$ such that $y$ is a stationary point of~\eqref{eq:OptiSmoothLift} and $\varphi(y)$ is not a stationary point of~\eqref{eq:OptiProblem};
\item \cite[Theorem~2.12]{LevinKileelBoumal2022SmoothLifts} gives two sufficient conditions and one necessary condition on $\varphi$ for the property ``for all twice differentiable $f$, if $y \in \mathcal{M}$ is a second-order stationary point of~\eqref{eq:OptiSmoothLift}, then $\varphi(y)$ is a stationary point of~\eqref{eq:OptiProblem}'' to hold.
\end{itemize}
As can be seen in \cite[Table~1]{LevinKileelBoumal2022SmoothLifts}, for many feasible sets $C$ of interest, there exist smooth lifts $\varphi$ mapping each second-order stationary point of~\eqref{eq:OptiSmoothLift} to a stationary point of~\eqref{eq:OptiProblem}. For example, \cite[Table~1]{LevinKileelBoumal2022SmoothLifts} gives such lifts for two of the three apocalyptic sets listed in Section~\ref{sec:Introduction}:
\begin{itemize}
\item the map
\begin{equation*}
\varphi : \R^{m \times r} \times \R^{n \times r} \to \R^{m \times n} : (L, R) \mapsto LR^\tp
\end{equation*}
is a smooth lift of $\R_{\le r}^{m \times n}$ called the rank factorization lift;
\item the map
\begin{equation*}
\varphi : \R^{n \times r} \to \R^{n \times n} : Y \mapsto YY^\tp
\end{equation*}
is a smooth lift of $\mathrm{S}_{\le r}^+(n)$ called the Burer--Monteiro lift.
\end{itemize}
For such feasible sets, one can find a stationary point of~\eqref{eq:OptiProblem} by running on~\eqref{eq:OptiSmoothLift} an algorithm guaranteed to accumulate at second-order stationary points; the trust-region method given in \cite[Algorithm~1]{LevinKileelBoumal2022} is an example of such an algorithm. This approach was successfully implemented in \cite{LevinKileelBoumal2022} for $\R_{\le r}^{m \times n}$.

\subsection{The $\rfdr$ algorithm}
\label{subsec:RFDRreview}
In this section, we review the $\rfdr$ algorithm \cite[Algorithm~3]{OlikierAbsil2022RFDR}. $\rfdr$ is defined on $\R_{\le r}^{m \times n}$ in \cite[\S 5]{OlikierAbsil2022RFDR} and its convergence properties are analyzed in \cite[\S 6]{OlikierAbsil2022RFDR}. More generally, it can be defined on the set $C$ while preserving the convergence properties under Assumption~\ref{assumption:StratificationRFDR}, as proven in Section~\ref{subsec:ConvergenceAnalysisRFDR}.

\begin{assumption}
\label{assumption:StratificationRFDR}
The set $C$ satisfies the following conditions:
\begin{enumerate}
\item conditions~1(a) and 1(b) of Assumption~\ref{assumption:Stratification} hold and, if $p \ge 2$, then, for all $x \in S_p$, $\dist(x, S_{p-1}) < \dist(x, S_{p-2})$;
\item $\inf_{x \in C \setminus S_p, z \in \mathcal{E} \setminus \{0\}} \frac{\norm{\proj{\tancone{C}{x}}{z}}}{\norm{z}} > 0$;
\item $C$ admits a \emph{restricted tangent cone}, i.e., a correspondence $C \setmapsto \mathcal{E} : x \mapsto \restancone{C}{x}$ such that:
\begin{enumerate}
\item for every $x \in C$, $\restancone{C}{x}$ is a closed cone contained in $\tancone{C}{x}$;
\item for all $x \in C$ and all $z \in \restancone{C}{x}$, $x+z \in C$;
\item there exists $\mu \in (0, 1]$ such that, for all $x \in C$ and all $z \in \mathcal{E}$, $\norm{\proj{\restancone{C}{x}}{z}} \ge \mu \norm{\proj{\tancone{C}{x}}{z}}$.
\end{enumerate}
\end{enumerate}
\end{assumption}

Observe that condition~1 of Assumption~\ref{assumption:StratificationRFDR} is weaker than condition~1 of Assumption~\ref{assumption:Stratification}.
The paper \cite{OlikierAbsil2022RFDR} is based on the fact that $\R_{\le r}^{m \times n}$ satisfies Assumption~\ref{assumption:StratificationRFDR}:
\begin{itemize}
\item we prove in Section~\ref{subsubsec:StratificationRealDeterminantalVariety} that condition~1 of Assumption~\ref{assumption:StratificationRFDR} holds;
\item by \eqref{eq:NormProjectionTangentConeDeterminantalVariety}, condition~2 of Assumption~\ref{assumption:StratificationRFDR} holds and the infimum equals $(\min\{m,n\}-r+1)^{-\frac{1}{2}}$;
\item by \cite[Proposition~3.2]{OlikierAbsil2022RFDR}, which is based on \cite[\S 3]{SchneiderUschmajew2015}, condition~3 of Assumption~\ref{assumption:StratificationRFDR} holds with $\mu := 2^{-\frac{1}{2}}$ for the restricted tangent cone from \cite[Definition~3.1]{OlikierAbsil2022RFDR}.
\end{itemize}
We prove in Section~\ref{subsec:SparseVectors} that $\sparse{n}{s}$ satisfies Assumption~\ref{assumption:StratificationRFDR}, its tangent cone being itself a restricted tangent cone.
However, to the best of our knowledge, $\R_{\le r}^{m \times n}$ and $\sparse{n}{s}$ are the only known examples of a set $C$ that satisfies Assumption~\ref{assumption:StratificationRFDR}.
In the rest of this section, we discuss the $\rfdr$ algorithm on a set $C$ satisfying Assumption~\ref{assumption:StratificationRFDR}.

The iteration map of $\rfdr$ \cite[Algorithm~2]{OlikierAbsil2022RFDR}, called the $\rfdr$ map, uses the $\rfd$ map \cite[Algorithm~1]{OlikierAbsil2022RFDR} as a subroutine. The $\rfd$ map essentially corresponds to the iteration map of \cite[Algorithm~4]{SchneiderUschmajew2015}, dubbed $\rfd$ in \cite[\S 1]{OlikierAbsil2022RFDR}. The name $\rfd$ comes from the fact that it is a retraction-free descent method, i.e., it performs each update along a straight line: given $x \in C$ as input, the $\rfd$ map performs a line search along a direction selected in $\proj{\restancone{C}{x}}{-\nabla f(x)}$, which does not involve any projection onto $C$. This retraction-free property has two advantages. First, it is fundamental to define and analyze $\rfdr$. Second, it saves the cost of computing a retraction, which, as pointed out in \cite[\S 3.4]{SchneiderUschmajew2015}, does not confer to $\rfd$ a significant advantage over $\ppgd$ if $C = \R_{\le r}^{m \times n}$ and $r \ll \min\{m,n\}$ since every point produced by the $\ppgd$ map is in $\R_{\le 2r}^{m \times n}$ (see Section~\ref{subsubsec:TangentConeRealDeterminantalVariety}) and reducing the rank from $2r$ to $r$ is typically much less expensive than evaluating the cost function or its gradient.

By \cite[Theorem~3.10]{SchneiderUschmajew2015}, if $C = \R_{\le r}^{m \times n}$ and $f$ is real-analytic and bounded from below, then $\rfd$ either produces a convergent sequence along which the stationarity measure $\s(\cdot; f, \R_{\le r}^{m \times n})$ goes to zero or produces a sequence diverging to infinity. We do not have such a guarantee for $\ppgd$; see \cite[Theorem~3.9]{SchneiderUschmajew2015}. However, as $\ppgd$, $\rfd$ can converge to a point that is Mordukhovich stationary for~\eqref{eq:OptiProblem} but not stationary. For example, by \cite[(19)]{OlikierAbsil2022RFDR}, $\rfd$ produces the same sequence as $\ppgd$ on the instance of~\eqref{eq:OptiProblem} from \cite[\S 2.2]{LevinKileelBoumal2022}.

Given $x \in C$ as input, the $\rfdr$ map applies the $\rfd$ map to $x$ but also, if $x \in S_p$ and $x$ is considered close to $S_{p-1}$, to a projection of $x$ onto $S_{p-1}$, and outputs a point among those thereby produced that maximally decreases $f$. Based on Assumption~\ref{assumption:StratificationRFDR}, \cite[Theorem~6.2]{OlikierAbsil2022RFDR}  (see also Theorem~\ref{thm:RFDRPolakConvergence}) states that $\rfdr$ produces a sequence whose accumulation points are stationary for~\eqref{eq:OptiProblem}.

Both $\ppgdr$ and $\rfdr$ provably accumulate at stationary points of~\eqref{eq:OptiProblem}. The main advantage of $\rfdr$ over $\ppgdr$ is that it does so by projecting its input onto at most one lower stratum while $\ppgdr$ can project its input onto each of the $p$ lower strata in the worst case. On the other hand, $\ppgdr$ has two advantages over $\rfdr$. First, as Assumption~\ref{assumption:Stratification} is less restrictive than Assumption~\ref{assumption:StratificationRFDR}, $\ppgdr$ can be defined on a broader class of feasible sets. For example, no restricted tangent cone to $\sparse{n}{s} \cap \R_+^n$ or $\mathrm{S}_{\le r}^+(n)$ is known to us. Second, the directions in $\proj{\tancone{C}{x}}{-\nabla f(x)}$ are more closely related to $-\nabla f(x)$ than those in $\proj{\restancone{C}{x}}{-\nabla f(x)}$; while this does not imply that $\ppgd$ converges faster than $\rfd$, such an observation was made experimentally in \cite[\S 3.4]{SchneiderUschmajew2015} for $C = \R_{\le r}^{m \times n}$.

\section{Preliminaries}
\label{sec:Preliminaries}
This section, mostly based on \cite{RockafellarWets} and \cite{LevinKileelBoumal2022}, introduces the background material needed in Sections~\ref{sec:ProposedAlgorithmConvergenceAnalysis}, \ref{sec:ExamplesStratifiedSetsSatisfyingMainAssumption}, and \ref{sec:ComplementaryResults}. In Section~\ref{subsec:ProjectionOntoClosedCones}, we recall basic properties of the projection onto closed cones. In Section~\ref{subsec:InnerOuterLimitsContinuityCorrespondences}, we review the concepts of inner and outer limits and continuity of correspondences, and prove Proposition~\ref{prop:DistanceToSetJointlyContinuous} on which Corollary~\ref{coro:ContinuityTangentConeImpliesContinuityStationarityMeasure} is based. In Section~\ref{subsec:TangentNormalConesGeometricDerivability}, we review the concepts of tangent and normal cones mentioned in Section~\ref{sec:Introduction} and related notions such as geometric derivability. In Section~\ref{subsec:StationarityMeasure}, we review basic properties of the stationarity measure $\s(\cdot;f, C)$ defined in~\eqref{eq:StationarityMeasure}, and prove that it is continuous if the tangent cone $\tancone{C}{\cdot}$ is continuous (Corollary~\ref{coro:ContinuityTangentConeImpliesContinuityStationarityMeasure}), a result which we use in Section~\ref{sec:ExamplesStratifiedSetsSatisfyingMainAssumption} to prove that $\R_{\le s}^n \cap \R_+^n$, $\R_{\le r}^{m \times n}$, and $\mathrm{S}_{\le r}^+(n)$ satisfy condition~3 of Assumption~\ref{assumption:Stratification}.

\subsection{Projection onto closed cones}
\label{subsec:ProjectionOntoClosedCones}
A set $S \subseteq \mathcal{E}$ is said to be \emph{locally closed} at $x \in \mathcal{E}$ if there exists $\delta \in (0, \infty)$ such that $S \cap \ball[x, \delta]$ is closed; $S$ is closed if and only if it is locally closed at every $x \in \mathcal{E}$.
For every nonempty subset $S$ of $\mathcal{E}$ and every $x \in \mathcal{E}$, $\proj{S}{x} := \argmin_{y \in S} \norm{x-y}$ is the projection of $x$ onto $S$. The set $\proj{S}{x}$ can be empty in general but not if $S$ is closed, as formulated in Proposition~\ref{prop:ProjectionOntoClosedSet}. If $\proj{S}{x}$ is a singleton, we also use $\proj{S}{x}$ to denote the element of the singleton.

\begin{proposition}[{\cite[Example~1.20]{RockafellarWets}}]
\label{prop:ProjectionOntoClosedSet}
For every nonempty closed subset $S$ of $\mathcal{E}$ and every $x \in \mathcal{E}$, $\proj{S}{x}$ is nonempty and compact.
\end{proposition}

A nonempty subset $S$ of $\mathcal{E}$ is said to be a \emph{cone} if, for all $x \in S$ and all $\lambda \in [0,\infty)$, $\lambda x \in S$. In this paper, we mostly project onto closed cones and, in that case, Proposition~\ref{prop:ProjectionOntoClosedSet} can be completed as follows.

\begin{proposition}[{\cite[Proposition~A.6]{LevinKileelBoumal2022}}]
\label{prop:NormProjectionOntoClosedCone}
Let $S \subseteq \mathcal{E}$ be a closed cone. For all $x \in \mathcal{E}$ and all $y \in \proj{S}{x}$,
\begin{equation*}
\ip{x}{y} = \norm{y}^2
\end{equation*}
and, in particular,
\begin{equation*}
\norm{y}^2 = \norm{x}^2 - \dist(x, S)^2.
\end{equation*}
\end{proposition}

For every nonempty subset $S$ of $\mathcal{E}$,
\begin{equation*}
S^* := \{y \in \mathcal{E} \mid \ip{y}{x} \le 0 \; \forall x \in S\}
\end{equation*}
is a closed convex cone called the \emph{(negative) polar} of $S$.
If $S$ is a linear subspace of $\mathcal{E}$, then $S^*$ equals the orthogonal complement $S^\perp$ of $S$. If $\emptyset \ne S_1 \subseteq S_2 \subseteq \mathcal{E}$, then $S_1^* \supseteq S_2^*$.
Moreover, we have Proposition~\ref{prop:PolarAmbientSpace} which we use to prove Propositions~\ref{prop:TangentNormalConesAmbientSpace} and \ref{prop:PolarPSDconeBoundedRank}.

\begin{proposition}
\label{prop:PolarAmbientSpace}
If $S$ is contained in a linear subspace $V$ of $\mathcal{E}$, then
\begin{equation*}
S^* = (S^* \cap V) + V^\perp.
\end{equation*}
\end{proposition}

\begin{proof}
Observe that
\begin{align*}
S^*
&= \{y \in \mathcal{E} \mid \ip{y}{x} \le 0 \; \forall x \in S\}\\
&= \{y_\parallel+y_\perp \mid y_\parallel \in V,\, y_\perp \in V^\perp,\, \ip{y_\parallel+y_\perp}{x} \le 0 \; \forall x \in S\}\\
&= \{y_\parallel+y_\perp \mid y_\parallel \in V,\, y_\perp \in V^\perp,\, \ip{y_\parallel}{x} \le 0 \; \forall x \in S\}\\
&= \{y_\parallel \in V \mid \ip{y_\parallel}{x} \le 0 \; \forall x \in S\} + V^\perp\\
&= (S^* \cap V) + V^\perp.
\qedhere
\end{align*}
\end{proof}

\subsection{Inner and outer limits, continuity of correspondences}
\label{subsec:InnerOuterLimitsContinuityCorrespondences}
This section is based on \cite[Chapters 4 and 5]{RockafellarWets}.
For every sequence $(S_i)_{i \in \N}$ of sets in a metric space $(\mathcal{M}, \dist_\mathcal{M})$, the two sets
\begin{align*}
\inlim_{i \to \infty} S_i := \big\{x \in \mathcal{M} \mid \lim_{i \to \infty} \dist_\mathcal{M}(x, S_i) = 0\big\},&&
\outlim_{i \to \infty} S_i := \big\{x \in \mathcal{M} \mid \liminf_{i \to \infty} \dist_\mathcal{M}(x, S_i) = 0\big\}
\end{align*}
are closed and respectively called the \emph{inner} and \emph{outer limits} of $(S_i)_{i \in \N}$ \cite[Definition~4.1, Exercise~4.2(a), and Proposition~4.4]{RockafellarWets}. If $S_i \ne \emptyset$ for all $i \in \N$, then $\inlim_{i \to \infty} S_i$ and $\outlim_{i \to \infty} S_i$ are respectively the sets of all possible limits and of all possible accumulation points of sequences $(x_i)_{i \in \N}$ such that $x_i \in S_i$ for all $i \in \N$.
It is always true that $\inlim_{i \to \infty} S_i \subseteq \outlim_{i \to \infty} S_i$; if the inclusion is an equality, then $(S_i)_{i \in \N}$ is said to \emph{converge in the sense of Painlev\'{e}} and $\setlim_{i \to \infty} S_i := \inlim_{i \to \infty} S_i = \outlim_{i \to \infty} S_i$ is called the \emph{limit} of $(S_i)_{i \in \N}$.

A \emph{correspondence}, or a \emph{set-valued mapping}, is a triplet $F := (A, B, G)$ where $A$ and $B$ are sets respectively called the \emph{set of departure} and the \emph{set of destination} of $F$, and $G$ is a subset of $A \times B$ called the \emph{graph} of $F$.
If $F := (A, B, G)$ is a correspondence, written $F : A \setmapsto B$, then the \emph{image} of $x \in A$ by $F$ is $F(x) := \{y \in B \mid (x,y) \in G\}$ and the \emph{domain} of $F$ is $\dom F := \{x \in A \mid F(x) \ne \emptyset\}$.

We now review a notion of continuity for correspondences $F : \mathcal{M}_1 \setmapsto \mathcal{M}_2$ where $(\mathcal{M}_1, \dist_{\mathcal{M}_1})$ and $(\mathcal{M}_2, \dist_{\mathcal{M}_2})$ are two metric spaces. Let $S$ be a nonempty subset of $\dom F$ and $x$ be in $S$. The two sets
\begin{align}
\label{eq:InnerLimitCorrespondence}
\inlim_{S \ni z \to x} F(z)
&:= \bigcap_{S \ni x_i \to x} \inlim_{i \to \infty} F(x_i)
= \big\{y \in \mathcal{M}_2 \mid \lim_{S \ni z \to x} \dist_{\mathcal{M}_2}(y, F(z)) = 0\big\},\\
\label{eq:OuterLimitCorrespondence}
\outlim_{S \ni z \to x} F(z)
&:= \bigcup_{S \ni x_i \to x} \outlim_{i \to \infty} F(x_i)
= \big\{y \in \mathcal{M}_2 \mid \liminf_{S \ni z \to x} \dist_{\mathcal{M}_2}(y, F(z)) = 0\big\}
\end{align}
are closed and respectively called the inner and outer limits of $F$ at $x$ relative to $S$ \cite[5(1)]{RockafellarWets}. Clearly, $\inlim_{S \ni z \to x} F(z) \subseteq \outlim_{S \ni z \to x} F(z)$; if the inclusion is an equality, then $\setlim_{S \ni z \to x} F(z) := \inlim_{S \ni z \to x} F(z) = \outlim_{S \ni z \to x} F(z)$ is called the limit of $F$ at $x$ relative to $S$.
By \cite[Definition~5.4]{RockafellarWets}, $F$ is said to be \emph{inner semicontinuous} at $x$ relative to $S$ if $\inlim_{S \ni z \to x} F(z) \supseteq F(x)$, \emph{outer semicontinuous} at $x$ relative to $S$ if $\outlim_{S \ni z \to x} F(z) \subseteq F(x)$, and \emph{continuous} at $x$ relative to $S$ if $F$ is both inner and outer semicontinuous at $x$ relative to $S$, i.e., $\setlim_{S \ni z \to x} F(z) = F(x)$. Thus, $F$ is continuous at $x$ relative to $S$ if and only if, for every sequence $(x_i)_{i \in \N}$ in $S$ converging to $x$, it holds that $\setlim_{i \to \infty} F(x_i) = F(x)$, i.e., $\outlim_{i \to \infty} F(x_i) \subseteq F(x) \subseteq \inlim_{i \to \infty} F(x_i)$.

We close this section by proving Proposition~\ref{prop:DistanceToSetJointlyContinuous} on which Corollary~\ref{coro:ContinuityTangentConeImpliesContinuityStationarityMeasure} is based.

\begin{proposition}
\label{prop:DistanceToSetJointlyContinuous}
Let $g : \mathcal{E} \to \mathcal{E}$ be continuous, $F : \mathcal{E} \setmapsto \mathcal{E}$ be closed-valued, and $S$ be a nonempty subset of $\dom F$. If $F$ is continuous at $x \in S$ relative to $S$, then the function
\begin{equation*}
\dom F \to \R : y \mapsto \dist(g(y), F(y))
\end{equation*}
is continuous at $x$ relative to $S$.
\end{proposition}

\begin{proof}
Let $x \in S$. For all $y \in S$,
\begin{align*}
|\dist(g(x), F(x))-\dist(g(y), F(y))|
\le\;& |\dist(g(x), F(x))-\dist(g(x), F(y))|\\
&+ |\dist(g(x), F(y))-\dist(g(y), F(y))|
\end{align*}
and, by \cite[Proposition~1.3.17]{Willem},
\begin{equation*}
|\dist(g(x), F(y))-\dist(g(y), F(y))| \le \norm{g(x)-g(y)}.
\end{equation*}
Let $\varepsilon \in (0, \infty)$.
First, by \cite[Proposition~5.11(c)]{RockafellarWets}, the function
\begin{equation*}
\dom F \to \R : y \mapsto \dist(g(x), F(y))
\end{equation*}
is continuous at $x$ relative to $S$. Thus, there exists $\delta_1 \in (0, \infty)$ such that, for all $y \in \ball[x, \delta_1] \cap S$,
\begin{equation*}
|\dist(g(x), F(x))-\dist(g(x), F(y))| \le \frac{\varepsilon}{2}.
\end{equation*}
Second, since $g$ is continuous at $x$, there exists $\delta_2 \in (0, \infty)$ such that, for all $y \in \ball[x, \delta_2]$,
\begin{equation*}
\norm{g(x)-g(y)} \le \frac{\varepsilon}{2}.
\end{equation*}
Therefore, if $\delta := \min\{\delta_1, \delta_2\}$, then, for all $y \in \ball[x, \delta] \cap S$,
\begin{equation*}
|\dist(g(x), F(x))-\dist(g(y), F(y))| \le \varepsilon,
\end{equation*}
which completes the proof.
\end{proof}

\subsection{Tangent and normal cones and geometric derivability}
\label{subsec:TangentNormalConesGeometricDerivability}
In this section, based on \cite[Chapters~6 and 13]{RockafellarWets}, we review the concepts of tangent cone, geometric derivability, regular normal cone, normal cone, Clarke normal cone, second-order tangent set, and parabolic derivability. Tangent and normal cones play a fundamental role in constrained optimization to describe admissible search directions and, in particular, to formulate optimality conditions. They notably appear in Section~\ref{subsec:StationarityMeasure}, in the description of the stationarity measure and in the characterization of apocalyptic points (Proposition~\ref{prop:CharacterizationApocalypticSerendipitousPoint}). The concept of parabolic derivability is related to Assumption~\ref{assumption:GlobalSecondOrderUpperBoundDistanceFromTangentLine}, as shown in Section~\ref{subsec:GeometricParabolicDerivability}.

In the rest of this section, $x$ is a point in a subset $S$ of $\mathcal{E}$. The tangent cone to $S$, the regular normal cone to $S$, the normal cone to $S$, and the Clarke normal cone to $S$ are correspondences with sets of departure and of destination both equal to $\mathcal{E}$, and domain equal to $S$.

The set
\begin{align}
\label{eq:TangentConeOuterLimit}
\tancone{S}{x}
:=&\; \outlim_{t \searrow 0} \frac{S-x}{t}\\
\label{eq:TangentConeDistance}
=&\; \left\{v \in \mathcal{E} \mid \liminf_{t \searrow 0} \frac{\dist(x+tv, S)}{t} = 0\right\}\\
\label{eq:TangentConeSequence}
=&\; \left\{v \in \mathcal{E} \mid \exists \begin{array}{l} (t_i)_{i \in \N} \text{ in } (0,\infty) \text{ converging to } 0 \\ (x_i)_{i \in \N} \text{ in } S \text{ converging to } x \end{array} : \lim_{i \to \infty} \frac{x_i-x}{t_i} = v\right\}
\end{align}
is a closed cone called the \emph{tangent cone} to $S$ at $x$ \cite[Definition~6.1 and Proposition~6.2]{RockafellarWets}.
The equality between \eqref{eq:TangentConeOuterLimit} and \eqref{eq:TangentConeSequence} follows from the first equality in \eqref{eq:OuterLimitCorrespondence} while the equality between \eqref{eq:TangentConeOuterLimit} and \eqref{eq:TangentConeDistance} follows from the second equality in~\eqref{eq:OuterLimitCorrespondence} and the identity
\begin{equation}
\label{eq:TangentVectorDistance}
\frac{\dist(x+tv, S)}{t} = d\Big(v, \frac{S-x}{t}\Big)
\end{equation}
holding for all $t \in (0, \infty)$ and all $v \in \mathcal{E}$.
The closedness of $\tancone{S}{x}$ follows from the fact that it is an outer limit. The fact that $\tancone{S}{x}$ is a cone is clear from~\eqref{eq:TangentConeDistance}.

By \cite[Definition~6.1]{RockafellarWets}, $v \in \tancone{S}{x}$ is said to be \emph{derivable} if there exists $\gamma : [0, \tau] \to \mathcal{E}$ with $\tau \in (0, \infty)$, $\gamma([0, \tau]) \subseteq S$, $\gamma(0) = x$, and $\gamma'(0) = v$, and $S$ is said to be \emph{geometrically derivable} at $x$ if every $v \in \tancone{S}{x}$ is derivable.
By \cite[Proposition~6.2]{RockafellarWets}, the set of all $v \in \tancone{S}{x}$ that are derivable is
\begin{equation*}
\inlim_{t \searrow 0} \frac{S-x}{t} = \left\{v \in \mathcal{E} \mid \lim_{t \searrow 0} \frac{\dist(x+tv, S)}{t} = 0\right\},
\end{equation*}
and, in particular, $S$ is geometrically derivable at $x$ if and only if
\begin{equation*}
\inlim_{t \searrow 0} \frac{S-x}{t} = \outlim_{t \searrow 0} \frac{S-x}{t}.
\end{equation*}

By \cite[Definition~6.3 and Proposition~6.5]{RockafellarWets},
\begin{equation}
\label{eq:RegularNormalCone}
\regnorcone{S}{x} := \tancone{S}{x}^*
\end{equation}
is called the \emph{regular normal cone} to $S$ at $x$,
\begin{equation}
\label{eq:NormalCone}
\norcone{S}{x} := \outlim_{S \ni z \to x} \regnorcone{S}{z}
\end{equation}
is a closed cone called the \emph{normal cone} to $S$ at $x$, the inclusion
\begin{equation*}
\regnorcone{S}{x} \subseteq \norcone{S}{x}
\end{equation*}
holds, and $S$ is said to be \emph{Clarke regular} at $x$ if it is locally closed at $x$ and $\regnorcone{S}{x} = \norcone{S}{x}$. By \cite[6(19)]{RockafellarWets}, the closure of the convex hull of $\norcone{S}{x}$ is denoted by $\connorcone{S}{x}$ and called the \emph{convexified normal cone} to $S$ at $x$. By \cite[Exercise~6.38]{RockafellarWets}, if $S$ is closed, then $\connorcone{S}{x}$ is also called the \emph{Clarke normal cone} to $S$ at $x$.

By \cite[Example~6.8]{RockafellarWets}, if $S$ is a smooth manifold in $\mathcal{E}$ around $x$, then $\tancone{S}{x}$ equals the tangent space to $S$ at $x$, and $\regnorcone{S}{x}$, $\norcone{S}{x}$, and $\connorcone{S}{x}$ equal the normal space to $S$ at $x$ which is the orthogonal complement of $\tancone{S}{x}$.

Proposition~\ref{prop:TangentNormalConesAmbientSpace} studies the influence of the ambient space on the tangent and normal cones, and is used in Section~\ref{subsubsec:NormalConesPSDconeBoundedRank}.

\begin{proposition}
\label{prop:TangentNormalConesAmbientSpace}
If $S$ is contained in a linear subspace $V$ of $\mathcal{E}$, then, for all $x \in S$,
\begin{align*}
\tancone{S}{x} \subseteq V,&&
\regnorcone{S}{x} = (\regnorcone{S}{x} \cap V) + V^\perp,&&
\norcone{S}{x} = (\norcone{S}{x} \cap V) + V^\perp.
\end{align*}
\end{proposition}

\begin{proof}
The inclusion is clear from \eqref{eq:TangentConeSequence}. The first equality follows from the inclusion and Proposition~\ref{prop:PolarAmbientSpace}. Let us establish the second equality. By \eqref{eq:NormalCone}, \eqref{eq:OuterLimitCorrespondence}, and the first equality, we have
\begin{align*}
\norcone{S}{x}
&= \outlim_{S \ni z \to x} \regnorcone{S}{z}\\
&= \left\{w \in \mathcal{E} \mid \liminf_{S \ni z \to x} d\left(w, \regnorcone{S}{z}\right) = 0\right\}\\
&= \left\{w_\parallel+w_\perp \mid w_\parallel \in V,\, w_\perp \in V^\perp,\, \liminf_{S \ni z \to x} d\left(w_\parallel+w_\perp, \left(\regnorcone{S}{z} \cap V\right) + V^\perp\right) = 0\right\}\\
&= \left\{w_\parallel+w_\perp \mid w_\parallel \in V,\, w_\perp \in V^\perp,\, \liminf_{S \ni z \to x} d\left(w_\parallel, \regnorcone{S}{z} \cap V\right) = 0\right\}\\
&= \left\{w_\parallel \in V \mid \liminf_{S \ni z \to x} d\left(w_\parallel, \regnorcone{S}{z} \cap V\right) = 0\right\} + V^\perp\\
&= \left(\outlim_{S \ni z \to x} \left(\regnorcone{S}{z} \cap V\right)\right) + V^\perp\\
&= \left(\norcone{S}{x} \cap V\right) + V^\perp,
\end{align*}
where the fourth equality follows from the fact that, for all $w_\parallel \in V$, all $w_\perp \in V^\perp$, and all $z \in S$,
\begin{align*}
d\left(w_\parallel+w_\perp, \left(\regnorcone{S}{z} \cap V\right) + V^\perp\right)^2
&= \inf_{\substack{v_\parallel \in \regnorcone{S}{z} \cap V \\ v_\perp \in V^\perp}} \norm{(w_\parallel+w_\perp)-(v_\parallel+v_\perp)}^2\\
&= \inf_{\substack{v_\parallel \in \regnorcone{S}{z} \cap V \\ v_\perp \in V^\perp}} \norm{(w_\parallel-v_\parallel)+(w_\perp-v_\perp)}^2\\
&= \inf_{\substack{v_\parallel \in \regnorcone{S}{z} \cap V \\ v_\perp \in V^\perp}} \left(\norm{w_\parallel-v_\parallel}^2+\norm{w_\perp-v_\perp}^2\right)\\
&= \inf_{v_\parallel \in \regnorcone{S}{z} \cap V} \norm{w_\parallel-v_\parallel}^2 + \inf_{v_\perp \in V^\perp} \norm{w_\perp-v_\perp}^2\\
&= d\left(w_\parallel, \regnorcone{S}{z} \cap V\right)^2.
\qedhere
\end{align*}
\end{proof}

Given $v \in \tancone{S}{x}$, the set
\begin{align}
\label{eq:SecondOrderTangentSetOuterLimit}
\sectancone{S}{x}{v}
:=&\; \outlim_{t \searrow 0} \frac{S-x-tv}{\frac{t^2}{2}}\\
\label{eq:SecondOrderTangentSetDistance}
=&\; \left\{w \in \mathcal{E} \mid \liminf_{t \searrow 0} \frac{\dist(x+tv+\frac{t^2}{2}w, S)}{\frac{t^2}{2}} = 0\right\}\\
\label{eq:SecondOrderTangentSetSequence}
=&\; \left\{w \in \mathcal{E} \mid \exists \begin{array}{l} (t_i)_{i \in \N} \text{ in } (0,\infty) \text{ converging to } 0 \\ (w_i)_{i \in \N} \text{ in } \mathcal{E} \text{ converging to } w \end{array} : x+t_iv+\frac{t_i^2}{2}w_i \in S \; \forall i \in \N\right\}
\end{align}
is closed and called the \emph{second-order tangent set} to $S$ at $x$ for $v$ \cite[Definition~13.11]{RockafellarWets}. The equality between \eqref{eq:SecondOrderTangentSetOuterLimit} and \eqref{eq:SecondOrderTangentSetSequence} follows from the first equality in~\eqref{eq:OuterLimitCorrespondence} while the equality between \eqref{eq:SecondOrderTangentSetOuterLimit} and \eqref{eq:SecondOrderTangentSetDistance} follows from the second equality in~\eqref{eq:OuterLimitCorrespondence} and the identity
\begin{equation}
\label{eq:TangentVectorDistance}
\frac{\dist(x+tv+\frac{t^2}{2}w, S)}{\frac{t^2}{2}} = d\bigg(w, \frac{S-x-tv}{\frac{t^2}{2}}\bigg)
\end{equation}
holding for all $t \in (0, \infty)$ and all $w \in \mathcal{E}$.
The closedness of $\sectancone{S}{x}{v}$ follows from the fact that it is an outer limit.

We say that $w \in \sectancone{S}{x}{v}$ is derivable for $v \in \tancone{S}{x}$ if there exists $\gamma : [0, \tau] \to \mathcal{E}$ with $\tau \in (0, \infty)$, $\gamma([0, \tau]) \subseteq S$, $\gamma(0) = x$, $\gamma'(0) = v$, and $\gamma''(0) = w$. By \cite[Definition~13.11]{RockafellarWets}, $S$ is said to be \emph{parabolically derivable} at $x$ for $v \in \tancone{S}{x}$ if $\sectancone{S}{x}{v} \ne \emptyset$ and every $w \in \sectancone{S}{x}{v}$ is derivable for $v$. 
The set of all $w \in \sectancone{S}{x}{v}$ that are derivable for $v$ is
\begin{equation*}
\inlim_{t \searrow 0} \frac{S-x-tv}{\frac{t^2}{2}} = \left\{w \in \mathcal{E} \mid \lim_{t \searrow 0} \frac{\dist(x+tv+\frac{t^2}{2}w, S)}{\frac{t^2}{2}} = 0\right\},
\end{equation*}
and, in particular, $S$ is parabolically derivable at $x$ if and only if
\begin{equation*}
\emptyset \ne \inlim_{t \searrow 0} \frac{S-x-tv}{\frac{t^2}{2}} = \outlim_{t \searrow 0} \frac{S-x-tv}{\frac{t^2}{2}}.
\end{equation*}

\subsection{Stationarity measure}
\label{subsec:StationarityMeasure}
In this section, after recalling basic properties of the stationarity measure $\s(\cdot; f, C)$ defined in~\eqref{eq:StationarityMeasure}, we review the complementary notions of apocalyptic and serendipitous points introduced in \cite{LevinKileelBoumal2022}. Finally, we prove that the continuity of the correspondence $\tancone{C}{\cdot}$ implies the continuity of the function $\s(\cdot; f, C)$ (Corollary~\ref{coro:ContinuityTangentConeImpliesContinuityStationarityMeasure}), a property that we use in Section~\ref{sec:ExamplesStratifiedSetsSatisfyingMainAssumption} to prove that $\R_{\le s}^n \cap \R_+^n$, $\R_{\le r}^{m \times n}$, and $\mathrm{S}_{\le r}^+(n)$ satisfy condition~3 of Assumption~\ref{assumption:Stratification}.

By Proposition~\ref{prop:NormProjectionOntoClosedCone}, $\s(\cdot; f, C)$ is well defined and, for all $x \in C$,
\begin{equation}
\label{eq:StationarityMeasureExplicitFormula}
\s(x; f, C) = \sqrt{\norm{\nabla f(x)}^2 - \dist(-\nabla f(x), \tancone{C}{x})^2}.
\end{equation}
The following result is stated in Section~\ref{sec:Introduction}.

\begin{proposition}
For every differentiable function $\phi : \mathcal{E} \to \R$:
\begin{enumerate}
\item the correspondence
\begin{equation*}
C \setmapsto \mathcal{E} : x \mapsto \proj{\tancone{C}{x}}{-\nabla \phi(x)}
\end{equation*}
depends on $\phi$ only through its restriction $\phi|_C$;
\item the following conditions are equivalent and are satisfied if $x \in C$ is a local minimizer of $\phi|_C$:
\begin{enumerate}
\item $\ip{\nabla \phi(x)}{v} \ge 0$ for all $v \in \tancone{C}{x}$;
\item $-\nabla \phi(x) \in \regnorcone{C}{x}$;
\item $\s(x; \phi, C) = 0$.
\end{enumerate}
\end{enumerate}
\end{proposition}

\begin{proof}
The first statement follows from \cite[Lemmas~A.7 and A.8]{LevinKileelBoumal2022}, and the second from \cite[Theorem~6.12]{RockafellarWets} and \cite[Proposition~2.5]{LevinKileelBoumal2022}.
\end{proof}

The notion of apocalyptic point is defined in Section~\ref{sec:Introduction}. By \cite[Definition~2.8]{LevinKileelBoumal2022}, $x \in C$ is said to be \emph{serendipitous} if there exist a sequence $(x_i)_{i \in \N}$ in $C$ converging to $x$, a continuously differentiable function $\phi : \mathcal{E} \to \R$, and $\varepsilon \in (0,\infty)$ such that $\s(x_i; \phi, C) > \varepsilon$ for all $i \in \N$, yet $\s(x; \phi, C) = 0$.

Proposition~\ref{prop:StationaryPointApocalypticSerendipitous} illustrates the complementarity between the notions of apocalyptic point and serendipitous point. It states that, for every sequence $(x_i)_{i \in \N}$ in $C$ converging to a point $x \in C$, for $x$ to be a stationary point of~\eqref{eq:OptiProblem}, the condition $\lim_{i \to \infty} \s(x_i; f, C) = 0$ is necessary if $x$ is not serendipitous and sufficient if $x$ is not apocalyptic. We use this result in Section~\ref{subsec:P2GDR} to deduce Corollary~\ref{coro:P2GDRPolakConvergence} from Theorem~\ref{thm:P2GDRPolakConvergence}.

\begin{proposition}
\label{prop:StationaryPointApocalypticSerendipitous}
Let $(x_i)_{i \in \N}$ be a sequence in $C$ converging to a point $x \in C$.
\begin{itemize}
\item If $x$ is not apocalyptic, then $\lim_{i \to \infty} \s(x_i; f, C) = 0$ implies $\s(x; f, C) = 0$.
\item If $x$ is not serendipitous, then $\s(x; f, C) = 0$ implies $\lim_{i \to \infty} \s(x_i; f, C) = 0$.
\end{itemize}
\end{proposition}

\begin{proof}
This is a direct consequence of the definitions of apocalyptic point and serendipitous point.
\end{proof}

Proposition~\ref{prop:StationaryPointApocalypticSerendipitous} has a practical interest. An iterative algorithm designed to find a stationary point of~\eqref{eq:OptiProblem} produces a sequence $(x_i)_{i \in \N}$ in $C$. However, in practice, the algorithm is stopped after a finite number of iterations based on a stopping criterion. By Proposition~\ref{prop:StationaryPointApocalypticSerendipitous}, if $C$ has no serendipitous point and $(x_i)_{i \in \N}$ has an accumulation point that is stationary for~\eqref{eq:OptiProblem}, then, for every $\varepsilon \in (0, \infty)$, the set $\{i \in \N \mid \s(x_i; f, C) \le \varepsilon\}$ is infinite, which provides us with a stopping criterion: given $\varepsilon \in (0, \infty)$, stop the algorithm at iteration
\begin{equation}
\label{eq:StoppingCriterion}
i_\varepsilon := \min\{i \in \N \mid \s(x_i; f, C) \le \varepsilon\}.
\end{equation}
Nevertheless, as pointed out in \cite[\S 1]{LevinKileelBoumal2022}, this stopping criterion must be treated with circumspection if $C$ has apocalyptic points since the condition $\lim_{i \to \infty} \s(x_i; f, C) = 0$ is not sufficient for $x$ to be a stationary point of~\eqref{eq:OptiProblem} if $x$ is apocalyptic. Corollary~\ref{coro:P2GDRPolakConvergence} shows that, if $C$ has no serendipitous point and the generated sequence does not diverge to infinity, then this stopping criterion is always eventually satisfied by the proposed $\ppgdr$ algorithm.

Proposition~\ref{prop:ConvergenceStationarityMeasureZeroUpperStratumMordukhovichStationary} states that, if $C$ satisfies condition~1 of Assumption~\ref{assumption:Stratification} and $(x_i)_{i \in \N}$ is a sequence in $S_p$, then the condition $\lim_{i \to \infty} \s(x_i; f, C) = 0$ is sufficient for the accumulation points of $(x_i)_{i \in \N}$ to be Mordukhovich stationary for~\eqref{eq:OptiProblem}.

\begin{proposition}
\label{prop:ConvergenceStationarityMeasureZeroUpperStratumMordukhovichStationary}
If $C$ satisfies condition~1 of Assumption~\ref{assumption:Stratification}, $(x_i)_{i \in \N}$ is a sequence in $S_p$, and $\lim_{i \to \infty} \s(x_i; f, C) = 0$, then every accumulation point of $(x_i)_{i \in \N}$ is Mordukhovich stationary for~\eqref{eq:OptiProblem}.
\end{proposition}

\begin{proof}
We follow the argument from \cite[\S 3.4]{HosseiniLukeUschmajew2019}. Assume that $C$ satisfies condition~1 of Assumption~\ref{assumption:Stratification} and let $(x_i)_{i \in \N}$ be a sequence in $S_p$ having $x \in C$ as accumulation point and such that $\lim_{i \to \infty} \s(x_i; f, C) = 0$. For every $i \in \N$,
\begin{equation*}
-\nabla f(x_i) = \proj{\tancone{S_p}{x_i}}{-\nabla f(x_i)} + \proj{\norcone{S_p}{x_i}}{-\nabla f(x_i)}.
\end{equation*}
Let $(x_{i_k})_{k \in \N}$ be a subsequence converging to $x$. Since $\lim_{k \to \infty} \s(x_{i_k}; f, C) = 0$, it holds that
\begin{equation*}
\lim_{k \to \infty} \proj{\norcone{S_p}{x_{i_k}}}{-\nabla f(x_{i_k})} = -\nabla f(x).
\end{equation*}
Thus, $-\nabla f(x) \in \outlim_{i \to \infty} \norcone{S_p}{x_i} = \outlim_{i \to \infty} \regnorcone{C}{x_i} \subseteq \norcone{C}{x}$.
\end{proof}

The characterization of apocalyptic and serendipitous points given in Proposition~\ref{prop:CharacterizationApocalypticSerendipitousPoint} allows us to prove Propositions~\ref{prop:NonnegativeSparseVectorsApocalypticSerendipitousPoints}, \ref{prop:RealDeterminantalVarietyNoSerendipitousPoint}, \ref{prop:PSDconeBoundedRankApocalypticSerendipitousPoints}, and \ref{prop:SparseVectorsApocalypticSerendipitousPoints}.

\begin{proposition}[{\cite[Theorems~2.13 and 2.17]{LevinKileelBoumal2022}}]
\label{prop:CharacterizationApocalypticSerendipitousPoint}
A point $x \in C$ is:
\begin{itemize}
\item apocalyptic if and only if there exists a sequence $(x_i)_{i \in \N}$ in $C$ converging to $x$ such that $\big(\outlim_{i \to \infty} \tancone{C}{x_i}\big)^*$ is not a subset of $\regnorcone{C}{x}$;
\item serendipitous if and only if there exists a sequence $(x_i)_{i \in \N}$ in $C$ converging to $x$ such that $\regnorcone{C}{x}$ is not a subset of $\big(\outlim_{i \to \infty} \tancone{C}{x_i}\big)^*$.
\end{itemize}
\end{proposition}

We close this section by proving Corollary~\ref{coro:ContinuityTangentConeImpliesContinuityStationarityMeasure} which states that the function $\s(\cdot;f, C)$ is continuous if the correspondence $\tancone{C}{\cdot}$ is continuous.

\begin{corollary}
\label{coro:ContinuityTangentConeImpliesContinuityStationarityMeasure}
For every nonempty subset $S$ of $C$ and every continuously differentiable function $\phi : \mathcal{E} \to \R$, if the tangent cone $\tancone{C}{\cdot}$ is continuous on $S$ relative to $S$, then the restriction $\s(\cdot; \phi, C)|_S$ is continuous.
\end{corollary}

\begin{proof}
Let $x \in S$. We have to prove that $\s(\cdot; \phi, C)$ is continuous at $x$ relative to $S$. By assumption, $\tancone{C}{\cdot}$ is continuous at $x$ relative to $S$. Thus, by Proposition~\ref{prop:DistanceToSetJointlyContinuous}, the function
\begin{equation*}
C \to \R : y \mapsto \dist(-\nabla \phi(y), \tancone{C}{y})
\end{equation*}
is continuous at $x$ relative to $S$. The result then follows from~\eqref{eq:StationarityMeasureExplicitFormula}.
\end{proof}

\section{The proposed algorithm and its convergence analysis}
\label{sec:ProposedAlgorithmConvergenceAnalysis}
In this section, under Assumption~\ref{assumption:Stratification}, we define $\ppgdr$ (Algorithm~\ref{algo:P2GDR}) and prove that it produces a sequence whose accumulation points are stationary for~\eqref{eq:OptiProblem} (Theorem~\ref{thm:P2GDRPolakConvergence}). The organization of the section is described hereafter and summarized in Table~\ref{tab:AssumptionsAlgorithmsMainResults}.
In Section~\ref{subsec:P2GDmap}, under Assumption~\ref{assumption:GlobalSecondOrderUpperBoundDistanceFromTangentLine}, we prove that the $\ppgd$ map (Algorithm~\ref{algo:P2GDmap}) produces a point satisfying an Armijo condition (Corollary~\ref{coro:P2GDmapArmijoCondition}).
In Section~\ref{subsec:P2GDRmap}, under Assumption~\ref{assumption:Stratification}, we introduce the $\ppgdr$ map (Algorithm~\ref{algo:P2GDRmap}), which uses the $\ppgd$ map as a subroutine, and, based on Corollary~\ref{coro:P2GDmapArmijoCondition}, prove Proposition~\ref{prop:P2GDRmapPolak}.
Finally, in Section~\ref{subsec:P2GDR}, we introduce the $\ppgdr$ algorithm and prove Theorem~\ref{thm:P2GDRPolakConvergence} based on Proposition~\ref{prop:P2GDRmapPolak}. Using the concept of serendipitous point (see Section~\ref{subsec:StationarityMeasure}), we also deduce Corollary~\ref{coro:P2GDRPolakConvergence} from Theorem~\ref{thm:P2GDRPolakConvergence}.

\begin{table}[h]
\begin{center}
\begin{tabular}{llll}
\hline
\emph{Section} & \emph{Assumption} & \emph{Algorithm} & \emph{Main result}\\
\hline
Section~\ref{subsec:P2GDmap} & Assumption~\ref{assumption:GlobalSecondOrderUpperBoundDistanceFromTangentLine} & $\ppgd$ map (Algorithm~\ref{algo:P2GDmap}) & Corollary~\ref{coro:P2GDmapArmijoCondition}\\
\hline
Section~\ref{subsec:P2GDRmap} & \multirow{2}{*}{Assumption~\ref{assumption:Stratification}} & $\ppgdr$ map (Algorithm~\ref{algo:P2GDRmap}) & Proposition~\ref{prop:P2GDRmapPolak}\\
\cline{1-1}\cline{3-4}
Section~\ref{subsec:P2GDR} & & $\ppgdr$ (Algorithm~\ref{algo:P2GDR}) & Theorem~\ref{thm:P2GDRPolakConvergence}\\
\hline
\end{tabular}
\end{center}
\caption{Assumptions, algorithms, and main results of Section~\ref{sec:ProposedAlgorithmConvergenceAnalysis}.}
\label{tab:AssumptionsAlgorithmsMainResults}
\end{table}

\subsection{The $\ppgd$ map}
\label{subsec:P2GDmap}
In this section, under Assumption~\ref{assumption:GlobalSecondOrderUpperBoundDistanceFromTangentLine}, we prove that the $\ppgd$ map (Algorithm~\ref{algo:P2GDmap}) is well defined and produces a point satisfying an Armijo condition (Corollary~\ref{coro:P2GDmapArmijoCondition}).
For convenience, we recall that Assumption~\ref{assumption:GlobalSecondOrderUpperBoundDistanceFromTangentLine} states that, for all $x \in C$,
\begin{equation*}
u(x) := \sup_{v \in \tancone{C}{x} \setminus \{0\}} \frac{\dist(x+v, C)}{\norm{v}^2} < \infty.
\end{equation*}
Note that, by Proposition~\ref{prop:LocalSecondOrderUpperBoundDistanceFromTangentLineParabolicDerivability}, if $C$ is parabolically derivable at $x \in C$ for $v \in \tancone{C}{x}$, then
\begin{equation*}
\sup_{t \in (0, \infty)} \frac{\dist(x+tv, S)}{t^2} < \infty.
\end{equation*}

If $C = \R_{\le r}^{m \times n}$, the $\ppgd$ map corresponds to the iteration map of \cite[Algorithm~3]{SchneiderUschmajew2015} except that the initial step size for the backtracking procedure is chosen in a given bounded interval.

\begin{algorithm}[H]
\caption{$\ppgd$ map}
\label{algo:P2GDmap}
\begin{algorithmic}[1]
\Require
$(\mathcal{E}, C, f, \ushort{\alpha}, \bar{\alpha}, \beta, c)$ where $\mathcal{E}$ is a Euclidean vector space, $C$ is a nonempty closed subset of $\mathcal{E}$ satisfying Assumption~\ref{assumption:GlobalSecondOrderUpperBoundDistanceFromTangentLine}, $f : \mathcal{E} \to \R$ is differentiable with $\nabla f$ locally Lipschitz continuous, $0 < \ushort{\alpha} \le \bar{\alpha} < \infty$, and $\beta, c \in (0,1)$.
\Input
$x \in C$ such that $\s(x; f, C) > 0$.
\Output
$y \in \ppgd(x; \mathcal{E}, C, f, \ushort{\alpha}, \bar{\alpha}, \beta, c)$.

\State
Choose $g \in \proj{\tancone{C}{x}}{-\nabla f(x)}$, $\alpha \in [\ushort{\alpha},\bar{\alpha}]$, and $y \in \proj{C}{x + \alpha g}$;
\While
{$f(y) > f(x) - c \, \alpha \s(x; f, C)^2$}
\State
$\alpha \gets \alpha \beta$;
\State
Choose $y \in \proj{C}{x + \alpha g}$;
\label{algo:P2GDmap:LineSearch}
\EndWhile
\State
Return $y$.
\end{algorithmic}
\end{algorithm}

Let us recall that, since $\nabla f$ is locally Lipschitz continuous, for every closed ball $\mathcal{B} \subsetneq \mathcal{E}$,
\begin{equation*}
\lip_{\mathcal{B}}(\nabla f) := \sup_{\substack{x, y \in \mathcal{B} \\ x \ne y}} \frac{\norm{\nabla f(x) - \nabla f(y)}}{\norm{x-y}} < \infty,
\end{equation*}
which implies, by \cite[Lemma~1.2.3]{Nesterov2018}, that, for all $x, y \in \mathcal{B}$,
\begin{equation}
\label{eq:InequalityLipschitzContinuousGradient}
|f(y) - f(x) - \ip{\nabla f(x)}{y-x}| \le \frac{\lip_{\mathcal{B}}(\nabla f)}{2} \norm{y-x}^2.
\end{equation}

\begin{proposition}
\label{prop:P2GDmapUpperBoundCost}
Let $x \in C$ and $\bar{\alpha} \in (0,\infty)$.
Let $\mathcal{B} \subsetneq \mathcal{E}$ be a closed ball such that, for all $g \in \proj{\tancone{C}{x}}{-\nabla f(x)}$ and all $\alpha \in [0,\bar{\alpha}]$, $\proj{C}{x + \alpha g} \subseteq \mathcal{B}$; an example of such a ball is $\ball[x, 2\bar{\alpha}\s(x; f, C)]$.
If $C$ satisfies Assumption~\ref{assumption:GlobalSecondOrderUpperBoundDistanceFromTangentLine}, then, for all $g \in \proj{\tancone{C}{x}}{-\nabla f(x)}$ and all $\alpha \in [0,\bar{\alpha}]$,
\begin{equation}
\label{eq:P2GDmapUpperBoundCost}
\sup f(\proj{C}{x + \alpha g}) \le f(x) + \s(x; f, C)^2 \alpha \left(-1+\kappa_\mathcal{B}(x; f, \bar{\alpha})\alpha\right),
\end{equation}
where
\begin{equation*}
\kappa_\mathcal{B}(x; f, \bar{\alpha})
:= u(x) \norm{\nabla f(x)} + \frac{1}{2} \lip\limits_\mathcal{B}(\nabla f) \left(\bar{\alpha} u(x) \s(x; f, C)+1\right)^2.
\end{equation*}
\end{proposition}

\begin{proof}
The example $\ball[x, 2\bar{\alpha}\s(x; f, C)]$ is correct because, for all $v \in \tancone{C}{x}$ and all $y \in \proj{C}{x+v}$,
\begin{equation*}
\norm{y-x}
= \norm{y-(x+v)+v}
\le \norm{y-(x+v)}+\norm{v}
= \dist(x+v, C)+\norm{v}
\le \norm{(x+v)-x}+\norm{v}
= 2\norm{v}.
\end{equation*}
Let $g \in \proj{\tancone{C}{x}}{-\nabla f(x)}$. The proof of~\eqref{eq:P2GDmapUpperBoundCost} is based on~\eqref{eq:InequalityLipschitzContinuousGradient} and the equality $\ip{\nabla f(x)}{g} = -\s(x; f, C)^2$ which holds by Proposition~\ref{prop:NormProjectionOntoClosedCone} since $\tancone{C}{x}$ is a closed cone. Let $L := \lip_\mathcal{B}(\nabla f)$. For all $\alpha \in [0, \bar{\alpha}]$ and all $y \in \proj{C}{x+\alpha g}$,
\begin{align*}
f(&y)-f(x)\\
\le~& \ip{\nabla f(x)}{y-x} + \frac{L}{2} \norm{y-x}^2\\
=~& \ip{\nabla f(x)}{y-(x+\alpha g)+\alpha g} + \frac{L}{2} \norm{y-(x+\alpha g)+\alpha g}^2\\
=~& - \alpha \s(x; f, C)^2 + \ip{\nabla f(x)}{y-(x+\alpha g)} + \frac{L}{2} \norm{y-(x+\alpha g)+\alpha g}^2\\
\le~& - \alpha \s(x; f, C)^2 + \norm{\nabla f(x)}\dist(x+\alpha g, C) + \frac{L}{2} \left(\dist(x+\alpha g, C)+\alpha \s(x; f, C)\right)^2\\
\le~& - \alpha \s(x; f, C)^2 + \norm{\nabla f(x)} u(x) \alpha^2 \s(x; f, C)^2 + \frac{L}{2} \left(u(x) \alpha^2 \s(x; f, C)^2+\alpha \s(x; f, C)\right)^2\\
=~& \alpha \s(x; f, C)^2 \left(-1+\alpha\left(\norm{\nabla f(x)} u(x) + \frac{L}{2} \left(u(x) \alpha \s(x; f, C)+1\right)^2\right)\right)\\
\le~& \alpha \s(x; f, C)^2 \left(-1+\alpha\kappa_\mathcal{B}(x; f, \bar{\alpha})\right),
\end{align*}
where the third inequality follows from Assumption~\ref{assumption:GlobalSecondOrderUpperBoundDistanceFromTangentLine}.
\end{proof}

In Proposition~\ref{prop:P2GDmapUpperBoundCost}, the existence of a ball $\mathcal{B}$ crucially relies on the upper bound $\bar{\alpha}$ required by Algorithm~\ref{algo:P2GDmap}.
Corollary~\ref{coro:P2GDmapArmijoCondition} states that the while loop in Algorithm~\ref{algo:P2GDmap} terminates and produces a point satisfying an Armijo condition. It plays an instrumental role in the proof of Proposition~\ref{prop:P2GDRmapPolak}.

\begin{corollary}
\label{coro:P2GDmapArmijoCondition}
The while loop in Algorithm~\ref{algo:P2GDmap} terminates and every $y \in \hyperref[algo:P2GDmap]{\ppgd}(x; \mathcal{E}, C, f, \ushort{\alpha}, \bar{\alpha}, \beta, c)$ satisfies the Armijo condition
\begin{equation*}
f(y) \le f(x) - c \, \alpha \s(x; f, C)^2
\end{equation*}
for some $\alpha \in \left[\min\left\{\ushort{\alpha}, \beta\frac{1-c}{\kappa_\mathcal{B}(x; f, \bar{\alpha})}\right\}, \bar{\alpha}\right]$, where $\mathcal{B}$ is any closed ball as in Proposition~\ref{prop:P2GDmapUpperBoundCost}.
\end{corollary}

\begin{proof}
For all $\alpha \in (0,\infty)$,
\begin{equation*}
f(x) + \s(x; f, C)^2 \alpha \big(-1+\kappa_\mathcal{B}(x; f, \bar{\alpha})\alpha\big) \le f(x) - c \s(x; f, C)^2 \alpha
\quad \text{iff} \quad
\alpha \le \frac{1-c}{\kappa_\mathcal{B}(x; f, \bar{\alpha})}.
\end{equation*}
Since the left-hand side of the first inequality is an upper bound on $f(\proj{C}{x + \alpha g})$ for all $\alpha \in (0,\bar{\alpha}]$, the Armijo condition is necessarily satisfied if $\alpha \in (0,\min\{\bar{\alpha}, \frac{1-c}{\kappa_\mathcal{B}(x; f, \bar{\alpha})}\}]$.
Therefore, either the initial step size chosen in $[\ushort{\alpha},\bar{\alpha}]$ satisfies the Armijo condition or the while loop ends with $\alpha$ such that $\frac{\alpha}{\beta} > \frac{1-c}{\kappa_\mathcal{B}(x; f, \bar{\alpha})}$.
\end{proof}

\subsection{The $\ppgdr$ map}
\label{subsec:P2GDRmap}
In this section, under Assumption~\ref{assumption:Stratification}, we introduce the $\ppgdr$ map (Algorithm~\ref{algo:P2GDRmap}) and prove Proposition~\ref{prop:P2GDRmapPolak} from which we deduce Theorem~\ref{thm:P2GDRPolakConvergence} in Section~\ref{subsec:P2GDR}.
For convenience, we recall that Assumption~\ref{assumption:Stratification} states that $C$ satisfies the following conditions:
\begin{enumerate}
\item there exist a positive integer $p$ and nonempty smooth submanifolds $S_0, \dots, S_p$ of $\mathcal{E}$ contained in $C$ such that:
\begin{enumerate}
\item for all $i, j \in \{0, \dots, p\}$, $i \ne j$ implies $S_i \cap S_j =\emptyset$;
\item $\overline{S_p} = C$ and, for all $i \in \{0, \dots, p\}$, $\overline{S_i} = \bigcup_{j=0}^i S_j$;
\item if $p \ge 2$, then, for all $i \in \{2, \dots, p\}$, all $x \in S_i$, and all $j \in \{1, \dots, i-1\}$, $\dist(x, S_j) < \dist(x, S_{j-1})$;
\end{enumerate}
\item Assumption~\ref{assumption:GlobalSecondOrderUpperBoundDistanceFromTangentLine} holds and, for every $i \in \{0, \dots, p\}$, $u|_{S_i}$ is locally bounded;
\item for every $i \in \{0, \dots, p\}$, $\tancone{C}{\cdot}$ is continuous on $S_i$ relative to $S_i$.
\end{enumerate}
Note that, under Assumption~\ref{assumption:Stratification}, if its input $x$ is in the stratum $S_p$, then the $\ppgd$ map (Algorithm~\ref{algo:P2GDmap}) performs an iteration of the Riemannian gradient descent on $S_p$. Indeed, as explained in Section~\ref{subsec:AssumptionsFeasibleSet}, for every $x \in S_p$, the tangent cone $\tancone{C}{x}$ equals the tangent space $\tancone{S_p}{x}$.

The $\ppgdr$ map involves the $\ppgd$ map and projections of a point onto its lower strata; the projection of a point onto each of its lower strata exists by the second statement of Proposition~\ref{prop:Conditions1(b)(c)MainAssumption}.

\begin{proposition}
\label{prop:Conditions1(b)(c)MainAssumption}
Assume that conditions~1(a) and 1(b) of Assumption~\ref{assumption:Stratification} hold.
\begin{enumerate}
\item For all $i \in \{1, \dots, p\}$ and all $x \in S_i$, $\dist(x, S_{i-1}) > 0$ and, for all $j \in \{1, \dots, i\}$, $\dist(x, S_j) \le \dist(x, S_{j-1})$.
\item Condition~1(c) is equivalent to the following property: for all $i \in \{0, \dots, p-1\}$ and all $x \in C \setminus \overline{S_i}$, $\proj{S_i}{x} = \proj{\overline{S_i}}{x}$.
\end{enumerate}
\end{proposition}

\begin{proof}
The first statement follows directly from condition~1(b).
To prove the second, we assume that condition~1(c) holds and show that, for all $i \in \{0, \dots, p-1\}$ and all $x \in C \setminus \overline{S_i}$, $\proj{\overline{S_i}}{x} \subseteq S_i$; this implies $\proj{S_i}{x} = \proj{\overline{S_i}}{x}$. If $i = 0$, this is because, by condition~1(b), $\overline{S_0} = S_0$. If $i \in \{1, \dots, p-1\}$, this is because, by conditions~1(b) and 1(c), $\dist(x, \overline{S_i} \setminus S_i) = \dist(x, \overline{S_{i-1}}) = \dist(x, S_{i-1}) > \dist(x, S_i)$.

Conversely, assume that the property holds and that $p \ge 2$. Let $i \in \{2, \dots, p\}$, $x \in S_i$, and $j \in \{1, \dots, i-1\}$. By the first statement, $\dist(x, S_j) \le \dist(x, S_{j-1})$. If $\dist(x, S_j) = \dist(x, S_{j-1})$, then $\overline{S_{j-1}} \cap \proj{\overline{S_j}}{x} \ne \emptyset$, in contradiction with the property. Thus, $\dist(x, S_j) < \dist(x, S_{j-1})$.
\end{proof}

The first statement of Proposition~\ref{prop:Conditions1(b)(c)MainAssumption} shows that adding condition~1(c) to condition~1(b) ensures that the second inequality of the first statement is strict. The second statement shows that condition~1(c) ensures that the projection of a point onto each of its lower strata exists. In general, condition~1(c) cannot be removed. For example, if $\mathcal{E} := \R^2$, $C := \R^2 \setminus (0, \infty)^2$, $S_0 := \{(0, 0)\}$, $S_1 := ((0, \infty) \times \{0\}) \cup (\{0\} \times (0, \infty))$, and $S_2 := \R^2 \setminus [0, \infty)^2$, then conditions~1(a) and 1(b) are satisfied but no point of $(-\infty, 0]^2$ has a projection onto $S_1$.

The $\ppgdr$ map is defined as Algorithm~\ref{algo:P2GDRmap}. Given $x \in C$ as input, it proceeds as follows: (i) it finds $i \in \{0, \dots, p\}$ such that $x \in S_i$ and computes $i_*$ as the smallest $j \in \{0, \dots, i\}$ such that $\dist(x, S_j) \le \Delta$ for some threshold $\Delta \in (0, \infty)$, (ii) for every $j \in \{i_*, \dots, i\}$, it applies the $\ppgd$ map (Algorithm~\ref{algo:P2GDmap}) to a projection $\hat{x}^j$ of $x$ onto $S_j$, thereby producing a point $\tilde{x}^j$, and (iv) it outputs a point among $\tilde{x}^{i_*}, \dots, \tilde{x}^i$ that maximally decreases $f$.

\begin{algorithm}[H]
\caption{$\ppgdr$ map}
\label{algo:P2GDRmap}
\begin{algorithmic}[1]
\Require
$(\mathcal{E}, C, f, \ushort{\alpha}, \bar{\alpha}, \beta, c, \Delta)$ where $\mathcal{E}$ is a Euclidean vector space, $C$ is a nonempty closed subset of $\mathcal{E}$ satisfying Assumption~\ref{assumption:Stratification}, $f : \mathcal{E} \to \R$ is differentiable with $\nabla f$ locally Lipschitz continuous, $0 < \ushort{\alpha} \le \bar{\alpha} < \infty$, $\beta, c \in (0,1)$, and $\Delta \in (0,\infty)$.
\Input
$x \in C$ such that $\s(x; f, C) > 0$.
\Output
$y \in \ppgdr(x; \mathcal{E}, C, f, \ushort{\alpha}, \bar{\alpha}, \beta, c, \Delta)$.

\State
Let $i \in \{0, \dots, p\}$ be such that $x \in S_i$;
\State
$i_* \gets \min\{j \in \{0, \dots, i\} \mid \dist(x, S_j) \le \Delta\}$;
\label{algo:P2GDRmap:DeltaCloseStrata}
\For
{$j \in \{i_*, \dots, i\}$}
\State
Choose $\hat{x}^j \in \proj{S_j}{x}$;
\State
Choose $\tilde{x}^j \in \hyperref[algo:P2GDmap]{\ppgd}(\hat{x}^j; \mathcal{E}, C, f, \ushort{\alpha}, \bar{\alpha}, \beta, c)$;
\EndFor
\State
Return $y \in \argmin_{\{\tilde{x}^j \mid j \in \{i_*, \dots, i\}\}} f$.
\end{algorithmic}
\end{algorithm}

Proposition~\ref{prop:P2GDRmapPolak} states that, given $\ushort{x} \in C$ such that $\s(\ushort{x}; f, C) > 0$, the minimum decrease of the cost function obtained by applying the $\ppgdr$ map to any $x$ sufficiently close to $\ushort{x}$ is bounded away from zero. This fundamental property, on which Theorem~\ref{thm:P2GDRPolakConvergence} is based, is not shared by the $\ppgd$ map. Indeed, the minimum decrease of the cost function obtained by applying the $\ppgd$ map to $x$ that is guaranteed by the Armijo condition given in Corollary~\ref{coro:P2GDmapArmijoCondition} can be arbitrarily small in any neighborhood of $\ushort{x}$: if $\s(\cdot; f, C)$ is not lower semicontinuous at $\ushort{x}$, which is the case if $\ushort{x}$ is apocalyptic for $f$, or $u$ is not locally bounded at $\ushort{x}$, then, even for arbitrarily small $\rho \in (0, \infty)$, it may happen that
\begin{equation*}
\inf_{x \in \ball[\ushort{x}, \rho]} c \min\left\{\ushort{\alpha}, \beta\frac{1-c}{\kappa_\mathcal{B}(x; f, \bar{\alpha})}\right\} \s(x; f, C)^2 = 0.
\end{equation*}
In contrast, if $x$ is sufficiently close to $\ushort{x}$, by exploring potentially several strata, the $\ppgdr$ map applies the $\ppgd$ map notably to a projection of $x$ onto the stratum containing $\ushort{x}$ which, by the standing assumptions and Corollary~\ref{coro:P2GDmapArmijoCondition}, produces a sufficient decrease of $f$.

\begin{proposition}
\label{prop:P2GDRmapPolak}
For every $\ushort{x} \in C$ such that $\s(\ushort{x}; f, C) > 0$, there exist $\varepsilon(\ushort{x}), \delta(\ushort{x}) \in (0,\infty)$ such that, for all $x \in \ball[\ushort{x},\varepsilon(\ushort{x})] \cap C$ and all $y \in \hyperref[algo:P2GDRmap]{\ppgdr}(x; \mathcal{E}, C, f, \ushort{\alpha}, \bar{\alpha}, \beta, c, \Delta)$,
\begin{equation}
\label{eq:PolakLocallyUniformSufficientDecrease}
f(y) - f(x) \le - \delta(\ushort{x}).
\end{equation}
\end{proposition}

\begin{proof}
Let $\ushort{x} \in C$ be such that $\s(\ushort{x}; f, C) > 0$. Let $\ushort{i} \in \{0, \dots, p\}$ be such that $\ushort{x} \in S_{\ushort{i}}$. The proof is divided into six steps.
First, using the positivity of $\dist(\ushort{x}, S_{\ushort{i}-1})$ if $\ushort{i} \ge 1$ (condition~1(b) of Assumption~\ref{assumption:Stratification}), the local boundedness of $u|_{S_{\ushort{i}}}$ at $\ushort{x}$ (condition~2 of Assumption~\ref{assumption:Stratification}), the continuity of $\s(\cdot; f, C)|_{S_{\ushort{i}}}$ at $\ushort{x}$ (which holds by condition~3 of Assumption~\ref{assumption:Stratification} and Corollary~\ref{coro:ContinuityTangentConeImpliesContinuityStationarityMeasure}), the continuity of $f$ at $\ushort{x}$, and the local Lipschitz continuity of $\nabla f$, we define $\bar{\rho}(\ushort{x})$, $\bar{\kappa}(\ushort{x}; f, \bar{\alpha})$, $\delta(\ushort{x})$, and $\varepsilon(\ushort{x})$ respectively as in~\eqref{eq:P2GDRmapPolakRho}, \eqref{eq:P2GDRmapPolakKappa}, \eqref{eq:P2GDRmapPolakDelta}, and \eqref{eq:P2GDRmapPolakEpsilon}. These definitions notably ensure that $f(\ball[\ushort{x}, 2\varepsilon(\ushort{x})]) \subseteq [f(\ushort{x})-\delta(\ushort{x}), f(\ushort{x})+\delta(\ushort{x})]$ and $\s(\ball[\ushort{x}, 2\varepsilon(\ushort{x})] \cap S_{\ushort{i}}; f, C) \subseteq [\frac{1}{2}\s(\ushort{x}; f, C), \frac{3}{2}\s(\ushort{x}; f, C)]$.
Second, we deduce from the preceding step that, for all $\hat{x} \in \ball[\ushort{x}, 2\varepsilon(\ushort{x})] \cap S_{\ushort{i}}$, $\kappa_{\ball[\ushort{x}, \bar{\rho}(\ushort{x})]}(\hat{x}; f, \bar{\alpha}) \le \bar{\kappa}(\ushort{x}; f, \bar{\alpha})$.
Third, we establish the inclusion $\ball[\ushort{x}, \varepsilon(\ushort{x})] \cap C \subseteq \bigcup_{i=\ushort{i}}^p S_i$.
Fourth, we prove that, given $x \in \ball[\ushort{x}, \varepsilon(\ushort{x})] \cap C$ as input, the $\ppgdr$ map considers $\hat{x}^{\ushort{i}} \in \proj{S_{\ushort{i}}}{x} \subseteq \ball[\ushort{x}, 2\varepsilon(\ushort{x})]$ and $\tilde{x}^{\ushort{i}} \in \hyperref[algo:P2GDmap]{\ppgd}(\hat{x}^{\ushort{i}}; \mathcal{E}, C, f, \ushort{\alpha}, \bar{\alpha}, \beta, c)$.
Fifth, we prove that Corollary~\ref{coro:P2GDmapArmijoCondition} applies to $\tilde{x}^{\ushort{i}}$ with the ball $\ball[\ushort{x}, \bar{\rho}(\ushort{x})]$.
Sixth, we deduce the inequality~\eqref{eq:PolakLocallyUniformSufficientDecrease} from the upper bound $\bar{\kappa}(\ushort{x}; f, \bar{\alpha})$ and the Armijo condition respectively obtained in the second and fifth steps.

\emph{Step~1: definition of $\bar{\rho}(\ushort{x})$, $\bar{\kappa}(\ushort{x}; f, \bar{\alpha})$, $\delta(\ushort{x})$, and $\varepsilon(\ushort{x})$.}
By condition 2 of Assumption~\ref{assumption:Stratification}, $u|_{S_{\ushort{i}}}$ is locally bounded at $\ushort{x}$ and thus there exist $\rho_u(\ushort{x}), \bar{u}(\ushort{x}) \in (0, \infty)$ such that $u(\ball[\ushort{x}, \rho_u(\ushort{x})] \cap S_{\ushort{i}}) \subseteq [0, \bar{u}(\ushort{x})]$.
Define
\begin{align}
\label{eq:P2GDRmapPolakRho}
\bar{\rho}(\ushort{x})
&:= 3 \bar{\alpha} \s(\ushort{x}; f, C) + \Delta,\\
\label{eq:P2GDRmapPolakKappa}
\bar{\kappa}(\ushort{x}; f, \bar{\alpha})
&:= \frac{3}{2} \bar{u}(\ushort{x}) \norm{\nabla f(\ushort{x})} + \frac{1}{2} \lip_{\ball[\ushort{x}, \bar{\rho}(\ushort{x})]}(\nabla f) \left(\frac{3}{2} \bar{\alpha} \bar{u}(\ushort{x}) \s(\ushort{x}; f, C)+1\right)^2,\\
\label{eq:P2GDRmapPolakDelta}
\delta(\ushort{x})
&:= \frac{c}{12} \min\left\{\ushort{\alpha}, \beta \frac{1-c}{\bar{\kappa}(\ushort{x}; f, \bar{\alpha})}\right\} \s(\ushort{x}; f, C)^2.
\end{align}
Since $f$ is continuous at $\ushort{x}$, there exists $\rho_f(\ushort{x}) \in (0, \infty)$ such that $f(\ball[\ushort{x}, \rho_f(\ushort{x})]) \subseteq [f(\ushort{x})-\delta(\ushort{x}), f(\ushort{x})+\delta(\ushort{x})]$.
By condition~3 of Assumption~\ref{assumption:Stratification} and Corollary~\ref{coro:ContinuityTangentConeImpliesContinuityStationarityMeasure}, $\s(\cdot; f, C)|_{S_{\ushort{i}}}$ is continuous at $\ushort{x}$ and thus there exists $\rho(\ushort{x}) \in (0, \infty)$ such that $\s(\ball[\ushort{x}, \rho(\ushort{x})] \cap S_{\ushort{i}}; f, C) \subseteq [\frac{1}{2}\s(\ushort{x}; f, C), \frac{3}{2}\s(\ushort{x}; f, C)]$.
By Proposition~\ref{prop:Conditions1(b)(c)MainAssumption}, if $\ushort{i} \ge 1$, then $\dist(\ushort{x}, S_{\ushort{i}-1}) > 0$.
Define
\begin{equation}
\label{eq:P2GDRmapPolakEpsilon}
\varepsilon(\ushort{x}) := \left\{\begin{array}{ll}
\frac{1}{2} \min\bigg\{\rho_u(\ushort{x}), \rho_f(\ushort{x}), \rho(\ushort{x}), \Delta, \frac{\norm{\nabla f(\ushort{x})}}{2\lip\limits_{\ball[\ushort{x}, \Delta]}(\nabla f)}, \dist(\ushort{x}, S_{\ushort{i}-1})\bigg\} & \text{if } \ushort{i} \ge 1,\\[2mm]
\frac{1}{2} \min\bigg\{\rho_u(\ushort{x}), \rho_f(\ushort{x}), \rho(\ushort{x}), \Delta, \frac{\norm{\nabla f(\ushort{x})}}{2\lip\limits_{\ball[\ushort{x}, \Delta]}(\nabla f)}\bigg\} & \text{if } \ushort{i} = 0.
\end{array}\right.
\end{equation}

\emph{Step~2: for all $\hat{x} \in \ball[\ushort{x}, 2\varepsilon(\ushort{x})] \cap S_{\ushort{i}}$, $\kappa_{\ball[\ushort{x}, \bar{\rho}(\ushort{x})]}(\hat{x}; f, \bar{\alpha}) \le \bar{\kappa}(\ushort{x}; f, \bar{\alpha})$.}
Let $\hat{x} \in \ball[\ushort{x}, 2\varepsilon(\ushort{x})] \cap S_{\ushort{i}}$. Since $2\varepsilon(\ushort{x}) \le \rho_u(\ushort{x})$, we have $u(\hat{x}) \le \bar{u}(\ushort{x})$. Since $2\varepsilon(\ushort{x}) \le \Delta$, it holds that $\ushort{x}, \hat{x} \in \ball[\ushort{x}, \Delta]$ and thus
\begin{equation*}
\norm{\nabla f(\hat{x})-\nabla f(\ushort{x})}
\le \lip_{\ball[\ushort{x}, \Delta]}(\nabla f) \norm{\hat{x}-\ushort{x}}
\le \lip_{\ball[\ushort{x}, \Delta]}(\nabla f) 2\varepsilon(\ushort{x})
\le \frac{1}{2}\norm{\nabla f(\ushort{x})},
\end{equation*}
where the last inequality follows from the inequality $2\varepsilon(\ushort{x}) \le \frac{\norm{\nabla f(\ushort{x})}}{2\lip_{\ball[\ushort{x}, \Delta]}(\nabla f)}$. Thus, since $|\norm{\nabla f(\hat{x})}-\norm{\nabla f(\ushort{x})}| \le \norm{\nabla f(\hat{x})-\nabla f(\ushort{x})}$, we have $\norm{\nabla f(\hat{x})} \in [\frac{1}{2}\norm{\nabla f(\ushort{x})}, \frac{3}{2}\norm{\nabla f(\ushort{x})}]$.
Therefore, by~\eqref{eq:P2GDRmapPolakKappa},
\begin{align*}
\kappa_{\ball[\ushort{x}, \bar{\rho}(\ushort{x})]}(\hat{x}; f, \bar{\alpha})
&= u(\hat{x}) \norm{\nabla f(\hat{x})} + \frac{1}{2} \lip\limits_{\ball[\ushort{x}, \bar{\rho}(\ushort{x})]}(\nabla f) \left(\bar{\alpha} u(\hat{x}) \s(\hat{x}; f, C)+1\right)^2\\
&\le \frac{3}{2} \bar{u}(\ushort{x}) \norm{\nabla f(\ushort{x})} + \frac{1}{2} \lip\limits_{\ball[\ushort{x}, \bar{\rho}(\ushort{x})]}(\nabla f) \left(\frac{3}{2} \bar{\alpha} \bar{u}(\ushort{x}) \s(\ushort{x}; f, C)+1\right)^2\\
&= \bar{\kappa}(\ushort{x}; f, \bar{\alpha}).
\end{align*}

\emph{Step~3: $\ball[\ushort{x}, \varepsilon(\ushort{x})] \cap C \subseteq \bigcup_{i=\ushort{i}}^p S_i$.}
Let $x \in \ball[\ushort{x}, \varepsilon(\ushort{x})] \cap C$. Let $i \in \{0, \dots, p\}$ be such that $x \in S_i$. Then, $i \in \{\ushort{i}, \dots, p\}$. This is obvious if $\ushort{i} = 0$. If $\ushort{i} \ge 1$, this is because
\begin{equation*}
\dist(x, S_{\ushort{i}-1})
\ge \dist(\ushort{x}, S_{\ushort{i}-1}) - \norm{x-\ushort{x}}
\ge \dist(\ushort{x}, S_{\ushort{i}-1}) - \varepsilon(\ushort{x})
\ge \frac{1}{2} \dist(\ushort{x}, S_{\ushort{i}-1})
> 0,
\end{equation*}
where the first inequality follows from \cite[Proposition~1.3.17]{Willem}.

\emph{Step~4: given $x \in \ball[\ushort{x}, \varepsilon(\ushort{x})] \cap C$ as input, the $\ppgdr$ map considers $\hat{x}^{\ushort{i}} \in \proj{S_{\ushort{i}}}{x} \subseteq \ball[\ushort{x}, 2\varepsilon(\ushort{x})]$ and $\tilde{x}^{\ushort{i}} \in \hyperref[algo:P2GDmap]{\ppgd}(\hat{x}^{\ushort{i}}; \mathcal{E}, C, f, \ushort{\alpha}, \bar{\alpha}, \beta, c)$.}
Since
\begin{equation*}
\dist(x, S_{\ushort{i}})
\le \norm{x-\ushort{x}}
\le \varepsilon(\ushort{x})
\le \Delta,
\end{equation*}
it holds that $i_* \le \ushort{i} \le i$ and thus, given $x$ as input, the $\ppgdr$ map considers $\hat{x}^{\ushort{i}} \in \proj{S_{\ushort{i}}}{x}$ and $\tilde{x}^{\ushort{i}} \in \hyperref[algo:P2GDmap]{\ppgd}(\hat{x}^{\ushort{i}}; \mathcal{E}, C, f, \ushort{\alpha}, \bar{\alpha}, \beta, c)$. Moreover,
\begin{equation*}
\norm{\hat{x}^{\ushort{i}}-\ushort{x}}
\le \norm{\hat{x}^{\ushort{i}}-x} + \norm{x-\ushort{x}}
\le \dist(x, S_{\ushort{i}}) + \varepsilon(\ushort{x})
\le 2 \varepsilon(\ushort{x}).
\end{equation*}

\emph{Step~5: Corollary~\ref{coro:P2GDmapArmijoCondition} applies to $\tilde{x}^{\ushort{i}}$ with the ball $\ball[\ushort{x}, \bar{\rho}(\ushort{x})]$.}
By the preceding step, $\hat{x}^{\ushort{i}} \in \ball[\ushort{x}, 2\varepsilon(\ushort{x})] \cap S_{\ushort{i}}$. Therefore, since $2 \varepsilon(\ushort{x}) \le \rho(\ushort{x})$, $\s(\hat{x}^{\ushort{i}}; f, C) \in [\frac{1}{2}\s(\ushort{x}; f, C), \frac{3}{2}\s(\ushort{x}; f, C)]$.
Thus, $\ball[\hat{x}^{\ushort{i}}, 2\bar{\alpha}\s(\hat{x}^{\ushort{i}}; f, C)] \subseteq \ball[\ushort{x}, \bar{\rho}(\ushort{x})]$; indeed, by~\eqref{eq:P2GDRmapPolakRho}, for all $y \in \ball[\hat{x}^{\ushort{i}}, 2\bar{\alpha}\s(\hat{x}^{\ushort{i}}; f, C)]$,
\begin{equation*}
\norm{y-\ushort{x}}
\le \norm{y-\hat{x}^{\ushort{i}}} + \norm{\hat{x}^{\ushort{i}}-\ushort{x}}
\le 2 \bar{\alpha} \s(\hat{x}^{\ushort{i}}; f, C) + 2 \varepsilon(\ushort{x})
\le 3 \bar{\alpha} \s(\ushort{x}; f, C) + \Delta
= \bar{\rho}(\ushort{x}).
\end{equation*}

\emph{Step~6: conclusion.}
Since $\hat{x}^{\ushort{i}} \in \ball[\ushort{x}, 2\varepsilon(\ushort{x})] \cap S_{\ushort{i}}$, it holds that $f(\hat{x}^{\ushort{i}}) \le f(x) + 2\delta(\ushort{x})$ (because $|f(x)-f(\ushort{x})| \le \delta(\ushort{x})$ and $|f(\hat{x}^{\ushort{i}})-f(\ushort{x})| \le \delta(\ushort{x})$) and $\kappa_{\ball[\ushort{x}, \bar{\rho}(\ushort{x})]}(\hat{x}^{\ushort{i}}; f, \bar{\alpha}) \le \bar{\kappa}(\ushort{x}; f, \bar{\alpha})$.
Therefore, by applying Corollary~\ref{coro:P2GDmapArmijoCondition} to $\tilde{x}^{\ushort{i}}$ with the ball $\ball[\ushort{x}, \bar{\rho}(\ushort{x})]$, we successively obtain
\begin{align*}
f(\tilde{x}^{\ushort{i}})
&\le f(\hat{x}^{\ushort{i}}) - c \min\left\{\ushort{\alpha}, \beta\frac{1-c}{\kappa_{\ball[\ushort{x}, \bar{\rho}(\ushort{x})]}(\hat{x}^{\ushort{i}}; f, \bar{\alpha})}\right\} \s(\hat{x}^{\ushort{i}}; f, C)^2\\
&\le f(x) + 2\delta(\ushort{x}) - \frac{c}{4} \min\left\{\ushort{\alpha}, \beta\frac{1-c}{\bar{\kappa}(\ushort{x}; f, \bar{\alpha})}\right\} \s(\ushort{x}; f, C)^2\\
&= f(x) - \delta(\ushort{x}).
\end{align*}
Thus, for all $y \in \hyperref[algo:P2GDRmap]{\ppgdr}(x; \mathcal{E}, C, f, \ushort{\alpha}, \bar{\alpha}, \beta, c, \Delta)$,
\begin{align*}
f(y)
\le f(\tilde{x}^{\ushort{i}})
\le f(x) - \delta(\ushort{x}),
\end{align*}
which completes the proof.
\end{proof}

\subsection{The $\ppgdr$ algorithm}
\label{subsec:P2GDR}
The $\ppgdr$ algorithm is defined as Algorithm~\ref{algo:P2GDR}. It produces a sequence along which $f$ is strictly decreasing.

\begin{algorithm}[H]
\caption{$\ppgdr$}
\label{algo:P2GDR}
\begin{algorithmic}[1]
\Require
$(\mathcal{E}, C, f, \ushort{\alpha}, \bar{\alpha}, \beta, c, \Delta)$ where $\mathcal{E}$ is a Euclidean vector space, $C$ is a nonempty closed subset of $\mathcal{E}$ satisfying Assumption~\ref{assumption:Stratification}, $f : \mathcal{E} \to \R$ is differentiable with $\nabla f$ locally Lipschitz continuous, $0 < \ushort{\alpha} \le \bar{\alpha} < \infty$, $\beta, c \in (0,1)$, and $\Delta \in (0,\infty)$.
\Input
$x_0 \in C$.
\Output
a sequence in $C$.

\State
$i \gets 0$;
\While
{$\s(x_i; f, C) > 0$}
\State
Choose $x_{i+1} \in \hyperref[algo:P2GDRmap]{\ppgdr}(x_i; \mathcal{E}, C, f, \ushort{\alpha}, \bar{\alpha}, \beta, c, \Delta)$;
\label{algo:P2GDR:P2GDRmap}
\State
$i \gets i+1$;
\EndWhile
\end{algorithmic}
\end{algorithm}

Theorem~\ref{thm:P2GDRPolakConvergence} states that $\ppgdr$ accumulates at stationary points of~\eqref{eq:OptiProblem} and is thus apocalypse-free. However, it does not state that an accumulation point necessarily exists.

\begin{theorem}
\label{thm:P2GDRPolakConvergence}
Consider a sequence constructed by $\ppgdr$ (Algorithm~\ref{algo:P2GDR}). If this sequence is finite, then its last element is stationary for~\eqref{eq:OptiProblem}, i.e., is a zero of the stationarity measure $\s(\cdot; f, C)$ defined in~\eqref{eq:StationarityMeasure}. If it is infinite, then all of its accumulation points are stationary for~\eqref{eq:OptiProblem}.
\end{theorem}

\begin{proof}
We use the framework proposed in \cite[\S 1.3]{Polak1971}.
Clearly, if $\ppgdr$ produces a finite sequence, then its last element is stationary. Let us therefore assume that $\ppgdr$ produces an infinite sequence $(x_i)_{i \in \N}$ and that a subsequence $(x_{i_k})_{k \in \N}$ converges to $x \in C$. For the sake of contradiction, assume that $x$ is not stationary for~\eqref{eq:OptiProblem} and let $\varepsilon(x)$ and $\delta(x)$ be given by Proposition~\ref{prop:P2GDRmapPolak}. There exists $K \in \N$ such that, for all integers $k \ge K$, $x_{i_k} \in \ball[x, \varepsilon(x)]$ and thus $f(x_{i_k+1})-f(x_{i_k}) \le -\delta(x)$. Thus, since $(f(x_i))_{i \in \N}$ is decreasing, for all integers $k \ge K$,
\begin{equation}
\label{eq:P2GDRPolakConvergence}
f(x_{i_{k+1}})-f(x_{i_k}) \le -\delta(x).
\end{equation}
Since $f$ is continuous, $(f(x_{i_k}))_{k \in \N}$ converges to $f(x)$. Therefore, letting $k$ tend to infinity in~\eqref{eq:P2GDRPolakConvergence} yields a contradiction.
\end{proof}

Corollary~\ref{coro:P2GDRPolakConvergence} considers a sequence $(x_i)_{i \in \N}$ produced by $\ppgdr$. It guarantees that, if $C$ has no serendipitous point, which is notably the case of $\R_{\le r}^{m \times n}$ (Proposition~\ref{prop:RealDeterminantalVarietyNoSerendipitousPoint}), and the sublevel set $\{x \in C \mid f(x) \le f(x_0)\}$ is bounded, then $\lim_{i \to \infty} \s(x_i; f, C) = 0$, and all accumulation points, of which there exists at least one, have the same image by $f$.

\begin{corollary}
\label{coro:P2GDRPolakConvergence}
Let $(x_i)_{i \in \N}$ be a sequence produced by $\ppgdr$ (Algorithm~\ref{algo:P2GDR}).
The sequence has at least one accumulation point if and only if $\liminf_{i \to \infty} \norm{x_i} < \infty$. If $C$ has no serendipitous point, then, for every convergent subsequence $(x_{i_k})_{k \in \N}$, $\lim_{k \to \infty} \s(x_{i_k}; f, C) = 0$.
If, moreover, $(x_i)_{i \in \N}$ is bounded, which is the case notably if the sublevel set $\{x \in C \mid f(x) \le f(x_0)\}$ is bounded, then $\lim_{i \to \infty} \s(x_i; f, C) = 0$, and all accumulation points have the same image by $f$.
\end{corollary}

\begin{proof}
The ``if and only if'' statement is a classical result.
Assume that $C$ has no serendipitous point. Then, Proposition~\ref{prop:StationaryPointApocalypticSerendipitous} implies that $\s(\cdot; f, C)$ goes to zero along every convergent subsequence of $(x_i)_{i \in \N}$.
Assume further that $(x_i)_{i \in \N}$ is bounded and let us prove that $\lim_{i \to \infty} \s(x_i; f, C) = 0$. Observe that $(\s(x_i; f, C))_{i \in \N}$ is bounded and let $s \in [0, \infty)$ be an accumulation point. It suffices to prove that $s = 0$. There exists a subsequence $(x_{i_k})_{k \in \N}$ such that $s = \lim_{k \to \infty} \s(x_{i_k}; f, C)$. Since $(x_{i_k})_{k \in \N}$ is bounded, it contains a convergent subsequence $(x_{i_{k_l}})_{l \in \N}$ and, by Proposition~\ref{prop:StationaryPointApocalypticSerendipitous}, $\lim_{l \to \infty} \s(x_{i_{k_l}}; f, C) = 0$, which establishes the result.
The final claim follows from the argument given in the proof of \cite[Theorem~65]{Polak1971}. Specifically, if $(x_i)_{i \in \N}$ is bounded, then it contains at least one convergent subsequence. Assume that $(x_{i_k})_{k \in \N}$ and $(x_{j_k})_{k \in \N}$ converge respectively to $\ushort{x}$ and $\oshort{x}$. The sequence $(f(x_i))_{i \in \N}$ is decreasing and, since $(x_i)_{i \in \N}$ is bounded and $f$ is continuous, it converges to $\inf_{i \in \N} f(x_i)$. Therefore, $f(\ushort{x}) = \lim_{k \to \infty} f(x_{i_k}) = \lim_{i \to \infty} f(x_i) = \lim_{k \to \infty} f(x_{j_k}) = f(\oshort{x})$.
\end{proof}

Corollary~\ref{coro:P2GDRPolakConvergence} shows that the stopping criterion defined by \eqref{eq:StoppingCriterion} is always eventually satisfied by $\ppgdr$ if $C$ has no serendipitous point and the generated sequence has an accumulation point.
Indeed, if $C$ has no serendipitous point and $\ppgdr$ produces a sequence $(x_i)_{i \in \N}$ that has an accumulation point, i.e., that does not diverge to infinity, then, for every $\varepsilon \in (0, \infty)$, the set $\{i \in \N \mid \s(x_i; f, C) \le \varepsilon\}$ is nonempty and thus its minimum $i_\varepsilon$ exists. We further discuss this stopping criterion in Section~\ref{subsubsec:ComparisonSixOptimizationAlgorithmsRealDeterminantalVariety} and use it in Section~\ref{subsubsec:LKB22instance}.

\section{Examples of stratified sets satisfying Assumption~\ref{assumption:Stratification}}
\label{sec:ExamplesStratifiedSetsSatisfyingMainAssumption}
In this section, we prove Theorem~\ref{thm:ExamplesStratifiedSetsSatisfyingMainAssumption}.

\begin{theorem}
\label{thm:ExamplesStratifiedSetsSatisfyingMainAssumption}
If $C$ is $\sparse{n}{s} \cap \R_+^n$, $\R_{\le r}^{m \times n}$, or $\mathrm{S}_{\le r}^+(n)$, then:
\begin{enumerate}
\item Assumption~\ref{assumption:Stratification} is satisfied;
\item there exists $a \in (0, 1)$ such that, for all $x \in C$,
\begin{equation*}
u(x) = \sup_{v \in \tancone{C}{x} \setminus \{0\}} \frac{\dist(x+v, C)}{\norm{v}^2} \in [a\tilde{u}(x), \tilde{u}(x)],
\end{equation*}
where $\tilde{u}(x) := 0$ if $x \in S_0$ and $\tilde{u}(x) := \frac{1}{\dist(x, S_{i-1})}$ if $x \in S_i$ with $i \in \{1, \dots, p\}$;
\item $u$ is not locally bounded at any point of $C \setminus S_p$;
\item the set of apocalyptic points of $C$ is $C \setminus S_p$;
\item for all $i \in \{0, \dots, p\}$ and all $x \in S_i$, $\connorcone{C}{x} = \norcone{S_i}{x}$.
\end{enumerate}
Moreover, $\R_{\le r}^{m \times n}$ has no serendipitous point while, if $C$ is $\sparse{n}{s} \cap \R_+^n$ or $\mathrm{S}_{\le r}^+(n)$, then the set of serendipitous points of $C$ is $C \setminus S_p$.
\end{theorem}

\begin{proof}
The sets $\sparse{n}{s} \cap \R_+^n$, $\R_{\le r}^{m \times n}$, and $\mathrm{S}_{\le r}^+(n)$ are studied in Sections~\ref{subsec:NonnegativeSparseVectors}, \ref{subsec:RealDeterminantalVariety}, and \ref{subsec:ConePSDmatricesBoundedRank}, respectively. As shown in Table~\ref{tab:ProofExamplesStratifiedSetsSatisfyingMainAssumption}, most of the statements of the theorem are proven in those three sections. We complete the proof here.
First, we establish the second statement for $x \in S_0$. If $C$ is one of the three sets, then $C$ is a closed cone, which implies $\tancone{C}{0} = C$, and $S_0 = \{0\}$. The result follows.
Second, we prove that condition~2 of Assumption~\ref{assumption:Stratification} is satisfied. For every $i \in \{0, \dots, p\}$, $\tilde{u}|_{S_i}$ is locally bounded. This is clear if $i = 0$ since $\tilde{u}|_{S_0} = 0$.
If $i \in \{1, \dots, p\}$, then $\tilde{u}|_{S_i}$ is continuous (as the inverse of $d(\cdot, S_{i-1})|_{S_i}$ which is continuous by \cite[Proposition~1.3.17]{Willem} and positive by Proposition~\ref{prop:Conditions1(b)(c)MainAssumption}) and thus locally bounded. The result then follows from the second statement.
Third, we prove that $u$ is not locally bounded at $\ushort{x} \in S_i$ if $i \in \{0, \dots, p-1\}$. Let $M, \delta \in (0, \infty)$. By condition~1(b) of Assumption~\ref{assumption:Stratification}, $\ushort{x} \in \overline{S_{i+1}}$ and thus there exists $x \in \ball[\ushort{x}, \min\{\delta, \frac{a}{M}\}] \cap S_{i+1}$. It follows that $u(x) \ge \frac{a}{d(x, S_i)} \ge \frac{a}{\norm{x-\ushort{x}}} \ge M$.
\begin{table}
\begin{center}
\begin{spacing}{1.4}
\begin{tabular}{*{4}{l}}
\hline
\emph{Result} & $C = \R_{\le s}^n \cap \R_+^n$ & $C = \R_{\le r}^{m \times n}$ & $C = \mathrm{S}_{\le r}^+(n)$\\[1mm]
\hline
Condition~1 of Assumption~\ref{assumption:Stratification} & Proposition~\ref{prop:StratificationNonnegativeSparseVectorsCondition1MainAssumption} & Section~\ref{subsubsec:StratificationRealDeterminantalVariety} & Section~\ref{subsubsec:StratificationPSDconeBoundedRank}\\
\hline
Second statement for $x \in C \setminus S_0$ & Proposition~\ref{prop:GlobalSecondOrderUpperBoundDistanceToNonnegativeSparseVectorsFromTangentLine} & Proposition~\ref{prop:GlobalSecondOrderUpperBoundDistanceToRealDeterminantalVarietyFromTangentLine} & Proposition~\ref{prop:GlobalSecondOrderUpperBoundDistanceToPSDconeBoundedRankFromTangentLine}\\
\hline
Condition~3 of Assumption~\ref{assumption:Stratification} & Proposition~\ref{prop:ContinuityTangentConeStratumNonnegativeSparseVectors} & \cite[Theorem~4.1]{OlikierAbsil2022} & Proposition~\ref{prop:ContinuityTangentConeStratumPSDconeBoundedRank}\\
\hline
Fourth and ``Moreover'' & \multirow{2}{*}{Proposition~\ref{prop:NonnegativeSparseVectorsApocalypticSerendipitousPoints}} & \cite[Proposition~2.10]{LevinKileelBoumal2022} & \multirow{2}{*}{Proposition~\ref{prop:PSDconeBoundedRankApocalypticSerendipitousPoints}}\\
statements & & Proposition~\ref{prop:RealDeterminantalVarietyNoSerendipitousPoint} &\\
\hline
Fifth statement & Corollary~\ref{coro:ClarkeNormalConeNonnegativeSparseVectors} & Proposition~\ref{prop:NormalConesRealDeterminantalVariety} & Corollary~\ref{coro:ClarkeNormalConePSDconeBoundedRank}\\
\hline
\end{tabular}
\vspace*{-8mm}
\end{spacing}
\end{center}
\caption{Proof of Theorem~\ref{thm:ExamplesStratifiedSetsSatisfyingMainAssumption}.}
\label{tab:ProofExamplesStratifiedSetsSatisfyingMainAssumption}
\end{table}
\end{proof}

Applications of problem~\eqref{eq:OptiProblem} with $C$ being $\sparse{n}{s} \cap \R_+^n$, $\R_{\le r}^{m \times n}$, or $\mathrm{S}_{\le r}^+(n)$ abound; we give some examples here and refer to \cite[\S 5]{JiaEtAl2022}, \cite{NadisicCohenVandaeleGillis}, and the references therein.
If $C = \R_{\le r}^{m \times n}$, applications include matrix equations, model reduction, matrix sensing, and matrix completion; see, e.g., \cite{SchneiderUschmajew2015, HaLiuBarber2020} and the references therein.
Problem~\eqref{eq:OptiProblem} with $C = \mathrm{S}_{\le r}^+(n)$ appears as the relaxation of combinatorial optimization problems. Motivated by that application, the feasible set $\mathrm{S}_{\le r}^+(n)$ is studied in \cite{BurerMonteiro2003, BurerMonteiro2005} with a linear cost function and in \cite{JourneeBachAbsilSepulchre2010} with a smooth cost function.
Problem~\eqref{eq:OptiProblem} with $C = \R_{\le s}^n \cap \R_+^n$ includes the sparse nonnegative least squares problem as a particular instance; see \cite{NadisicCohenVandaeleGillis} and the references therein.

In what follows, for all integers $i$ and $j$, we write $\delta_{i, j} := 1$ if $i = j$ and $\delta_{i, j} := 0$ if $i \ne j$.

\subsection{The set of nonnegative sparse vectors}
\label{subsec:NonnegativeSparseVectors}
In this section, $\mathcal{E} := \R^n$ and $C := \sparse{n}{s} \cap \R_+^n$ for some positive integers $n$ and $s < n$, and $\R^n$ is endowed with the \emph{dot product} $\ip{x}{y} := \sum_{i=1}^n x_i y_i$. We use the following notation throughout the section. If $x \in \R^n$, then, for every $i \in \{1, \dots, n\}$, $x_i$ denotes the $i$th component of $x$, and we also write $x$ as $(x_i)_{i \in \{1, \dots, n\}}$. Thus,
\begin{align*}
\R_+^n = \{x \in \R^n \mid x_i \ge 0 \; \forall i \in \{1, \dots, n\}\},&&
\R_-^n = \{x \in \R^n \mid x_i \le 0 \; \forall i \in \{1, \dots, n\}\}.
\end{align*}
A sequence in $\R^n$ is denoted by $(x^i)_{i \in \N}$. By \cite[Definition~2.1]{FoucartRauhut}, the \emph{support} of $x \in \R^n$ is defined as
\begin{align*}
\supp(x) := \{i \in \{1, \dots, n\} \mid x_i \ne 0\}.
\end{align*}
Then,
\begin{equation*}
\sparse{n}{s} = \{x \in \R^n \mid |\supp(x)| \le s\}.
\end{equation*}
We also write
\begin{equation*}
\StrictSparsity{n}{s} := \{x \in \R^n \mid |\supp(x)| < s\}
\end{equation*}
and, for every $i \in \{0, \dots, s\}$,
\begin{equation*}
\FixedSparsity{n}{i} := \{x \in \R^n \mid |\supp(x)| = i\}.
\end{equation*}
For all $x, y \in \R^n$, it holds that
\begin{equation}
\label{eq:SubadditivitySupport}
\supp(x+y) \subseteq \supp(x) \cup \supp(y).
\end{equation}

In Section~\ref{subsubsec:StratificationNonnegativeSparseVectors}, projection onto $\sparse{n}{s} \cap \R_+^n$ is constructed (Proposition~\ref{prop:ProjectionNonnegativeSparseVectors}) and it is proven that $\sparse{n}{s} \cap \R_+^n$ admits a stratification satisfying condition~1 of Assumption~\ref{assumption:Stratification} (Proposition~\ref{prop:StratificationNonnegativeSparseVectorsCondition1MainAssumption}). In Section~\ref{subsubsec:TangentConeNonnegativeSparseVectors}, we determine the tangent cone to $\sparse{n}{s} \cap \R_+^n$ (Proposition~\ref{prop:TangentConeNonnegativeSparseVectors}) and deduce that $\sparse{n}{s} \cap \R_+^n$ satisfies the second statement of Theorem~\ref{thm:ExamplesStratifiedSetsSatisfyingMainAssumption} (Proposition~\ref{prop:GlobalSecondOrderUpperBoundDistanceToNonnegativeSparseVectorsFromTangentLine}) and condition~3 of Assumption~\ref{assumption:Stratification} (Proposition~\ref{prop:ContinuityTangentConeStratumNonnegativeSparseVectors}). In Section~\ref{subsubsec:NormalConesNonnegativeSparseVectors}, we deduce the regular normal cone, the normal cone, and the Clarke normal cone to $\sparse{n}{s} \cap \R_+^n$, and prove that the sets of apocalyptic and serendipitous points of $\sparse{n}{s} \cap \R_+^n$ both equal $\StrictSparsity{n}{s} \cap \R_+^n$ (Proposition~\ref{prop:NonnegativeSparseVectorsApocalypticSerendipitousPoints}). Finally, in Section~\ref{subsubsec:NonnegativeSparseVectorsP2GDapocalypseExample}, we present an example of $\ppgd$ following an apocalypse on $\sparse{n}{s} \cap \R_+^n$.

\subsubsection{Stratification of the set of nonnegative sparse vectors}
\label{subsubsec:StratificationNonnegativeSparseVectors}
In this section, we prove that $\sparse{n}{s} \cap \R_+^n$ admits a stratification satisfying condition~1 of Assumption~\ref{assumption:Stratification}.
The number of nonzero components stratifies $\sparse{n}{s} \cap \R_+^n$:
\begin{equation*}
\sparse{n}{s} \cap \R_+^n = \bigcup_{i=0}^s \FixedSparsity{n}{i} \cap \R_+^n.
\end{equation*}

Proposition~\ref{prop:ProjectionNonnegativeSparseVectors} shows how to project onto $\sparse{n}{s} \cap \R_+^n$ and is used in the proof that $\sparse{n}{s} \cap \R_+^n$ satisfies condition~1(c) of Assumption~\ref{assumption:Stratification}.

\begin{proposition}[{projection onto the set of nonnegative sparse vectors \cite[Proposition~3.2]{Tam2017}}]
\label{prop:ProjectionNonnegativeSparseVectors}
For every $x \in \R^n$, $\proj{\sparse{n}{s} \cap \R_+^n}{x}$ is the set of all possible outputs of Algorithm~\ref{algo:ProjectionNonnegativeSparseVectors}, and $\dist(x, \sparse{n}{s} \cap \R_+^n)$ is the sum of the absolute values of the components of $x$ that have been set to zero by the projection.
\end{proposition}

\begin{algorithm}[H]
\caption{Projection onto the set of nonnegative sparse vectors}
\label{algo:ProjectionNonnegativeSparseVectors}
\begin{algorithmic}[1]
\Require
$(n, s)$ where $n$ and $s$ are positive integers such that $s < n$.
\Input
$x \in \R^n$.
\Output
$y \in \proj{\sparse{n}{s} \cap \R_+^n}{x}$.

\State
$y \gets x$;
\For
{$i \in \supp(x)$}
	\If
	{$x_i < 0$}
		\State
		$y_i \gets 0$;
	\EndIf
\EndFor
\While
{$|\supp(y)| > s$}
	\State
	Choose $i \in \argmin_{j \in \supp(y)} |y_j|$;
	\State
	$y_i \gets 0$;
\EndWhile
\State
Return $y$.
\end{algorithmic}
\end{algorithm}

Based on Proposition~\ref{prop:ProjectionNonnegativeSparseVectors}, we now prove that $\sparse{n}{s} \cap \R_+^n$ satisfies condition~1 of Assumption~\ref{assumption:Stratification}.

\begin{proposition}[stratification of the set of nonnegative sparse vectors]
\label{prop:StratificationNonnegativeSparseVectorsCondition1MainAssumption}
The stratification $\{\FixedSparsity{n}{0} \cap \R_+^n, \dots, \FixedSparsity{n}{s} \cap \R_+^n\}$ of $\sparse{n}{s} \cap \R_+^n$ satisfies condition~1 of Assumption~\ref{assumption:Stratification}.
\end{proposition}

\begin{proof}
Using the submanifold property \cite[Proposition~3.3.2]{AbsilMahonySepulchre}, we first prove that, for every $i \in \{0, \dots, s\}$, $\FixedSparsity{n}{i} \cap \R_+^n$ is an $i$-dimensional embedded submanifold of $\R^n$.
For $\FixedSparsity{n}{0} \cap \R_+^n = \{0\}^n$, we take $U := \R^n$ and $\varphi : \R^n \to \R^n : x \mapsto x$, and we have
\begin{equation*}
\{x \in U \mid \varphi(x) \in \{0\}^n\}
= \{0\}^n
= (\FixedSparsity{n}{0} \cap \R_+^n) \cap U.
\end{equation*}
Let $\ushort{x} \in \FixedSparsity{n}{i} \cap \R_+^n$ with $i \in \{1, \dots, s\}$. For $U := \ball(\ushort{x}, \dist(\ushort{x}, \FixedSparsity{n}{i-1} \cap \R_+^n))$ and
\begin{equation*}
\varphi : U \to \R^n : x \mapsto ((x_j)_{j \in \supp(\ushort{x})}, (x_j)_{j \in \{1, \dots, n\} \setminus \supp(\ushort{x})}),
\end{equation*}
we have
\begin{equation*}
\{x \in U \mid \varphi(x) \in \R^i \times \{0\}^{n-i}\}
= \{x \in U \mid x_j = 0 \, \forall j \in \{1, \dots, n\} \setminus \supp(\ushort{x})\}
= (\FixedSparsity{n}{i} \cap \R_+^n) \cap U.
\end{equation*}
Thus, condition~1(a) is satisfied. By Proposition~\ref{prop:ProjectionNonnegativeSparseVectors}, condition~1(c) is satisfied too. To establish condition~1(b), it suffices to prove that, for all $i, j \in \{0, \dots, s\}$, $(\FixedSparsity{n}{j} \cap \R_+^n) \cap \overline{\FixedSparsity{n}{i} \cap \R_+^n} = \emptyset$ if $j > i$ and $\FixedSparsity{n}{j} \cap \R_+^n \subseteq \overline{\FixedSparsity{n}{i} \cap \R_+^n}$ if $j \le i$.
Let $i, j \in \{0, \dots, s\}$. If $j > i$, then, by Proposition~\ref{prop:ProjectionNonnegativeSparseVectors}, for all $x \in \FixedSparsity{n}{j} \cap \R_+^n$, $\ball(x, \dist(x, \FixedSparsity{n}{j-1} \cap \R_+^n)) \cap (\FixedSparsity{n}{i} \cap \R_+^n) = \emptyset$ and thus $x \not\in \overline{\FixedSparsity{n}{i} \cap \R_+^n}$. If $j \le i$, then, for all $x \in \FixedSparsity{n}{j} \cap \R_+^n$ and all $\varepsilon \in (0, \infty)$, $\ball[x, \varepsilon] \cap (\FixedSparsity{n}{i} \cap \R_+^n) \ne \emptyset$. This is clear if $j = i$. Let us prove it in the case where $j < i$. Let $x \in \FixedSparsity{n}{j} \cap \R_+^n$ and $\varepsilon \in (0, \infty)$. Let $I(x) \subseteq \{1, \dots, n\} \setminus \supp(x)$ such that $|I(x)| = i-j$. Define $y \in \R^n$ by $y_k := \frac{\varepsilon}{\sqrt{n}}$ if $k \in I(x)$ and $y_k := 0$ otherwise. Then, $x+y \in \ball[x, \varepsilon] \cap (\FixedSparsity{n}{i} \cap \R_+^n)$.
\end{proof}

\subsubsection{Tangent cone to the set of nonnegative sparse vectors}
\label{subsubsec:TangentConeNonnegativeSparseVectors}
In this section, we give an explicit description of the tangent cone to $\sparse{n}{s} \cap \R_+^n$ and show how to project onto it (Proposition~\ref{prop:TangentConeNonnegativeSparseVectors}). Then, we prove that $\sparse{n}{s} \cap \R_+^n$ satisfies the second statement of Theorem~\ref{thm:ExamplesStratifiedSetsSatisfyingMainAssumption} (Proposition~\ref{prop:GlobalSecondOrderUpperBoundDistanceToNonnegativeSparseVectorsFromTangentLine}) and condition~3 of Assumption~\ref{assumption:Stratification} (Proposition~\ref{prop:ContinuityTangentConeStratumNonnegativeSparseVectors}).

\begin{proposition}[tangent cone to the set of nonnegative sparse vectors]
\label{prop:TangentConeNonnegativeSparseVectors}
For every $x \in \sparse{n}{s} \cap \R_+^n$,
\begin{equation*}
\tancone{\sparse{n}{s} \cap \R_+^n}{x} = \left\{v \in \R^n \mid |\supp(x) \cup \supp(v)| \le s,\, v_i \ge 0 \; \forall i \in \{1, \dots, n\} \setminus \supp(x)\right\},
\end{equation*}
$\sparse{n}{s} \cap \R_+^n$ is geometrically derivable at $x$, and, for every $v \in \R^n$, $\proj{\tancone{\sparse{n}{s} \cap \R_+^n}{x}}{v}$ is the set of all possible outputs of Algorithm~\ref{algo:ProjectionTangentConeNonnegativeSparseVectors}.
\end{proposition}

\begin{proof}
We establish the equation for the tangent cone; the projection onto it follows.
Since $\sparse{n}{s} \cap \R_+^n$ is a closed cone, $\tancone{\sparse{n}{s} \cap \R_+^n}{0} = \sparse{n}{s} \cap \R_+^n$, in agreement with the equation, and $\sparse{n}{s} \cap \R_+^n$ is geometrically derivable at $0$.
Let $x \in \FixedSparsity{n}{j} \cap \R_+^n$ with $j \in \{1, \dots, s\}$ and $v \in \R^n \setminus \{0\}$.
Then, $\dist(x, \FixedSparsity{n}{j-1} \cap \R_+^n) = \min_{i \in \supp(x)} x_i$ and for all $t \in \Big(0, \frac{\dist(x, \FixedSparsity{n}{j-1} \cap \R_+^n)}{2\norm{v}_\infty}\Big)$:
\begin{itemize}
\item for all $i \in \supp(x) \setminus \supp(v)$, $x_i+tv_i = x_i \ge \dist(x, \FixedSparsity{n}{j-1} \cap \R_+^n)$;
\item for all $i \in \supp(v) \setminus \supp(x)$, $|x_i+tv_i| = t|v_i| \in (0, \frac{1}{2}\dist(x, \FixedSparsity{n}{j-1} \cap \R_+^n))$;
\item for all $i \in \supp(x) \cap \supp(v)$, $x_i+tv_i \ge x_i-t|v_i| > \frac{1}{2}\dist(x, \FixedSparsity{n}{j-1} \cap \R_+^n)$;
\item for all $i \in \{1, \dots, n\} \setminus (\supp(x) \cup \supp(v))$, $x_i+tv_i = 0$.
\end{itemize}
Thus, for all $t \in \Big(0, \frac{\dist(x, \FixedSparsity{n}{j-1} \cap \R_+^n)}{2\norm{v}_\infty}\Big)$, $\supp(x+tv) = \supp(x) \cup \supp(v)$.

We establish the inclusion $\supseteq$. Assume that $|\supp(x) \cup \supp(v)| \le s$ and $v_i \ge 0$ for all $i \in \{1, \dots, n\} \setminus \supp(x)$. Then, for all $t \in \Big(0, \frac{\dist(x, \FixedSparsity{n}{j-1} \cap \R_+^n)}{2\norm{v}_\infty}\Big)$, $x+tv \in \sparse{n}{s} \cap \R_+^n$ and thus $\dist(x+tv, \sparse{n}{s} \cap \R_+^n)/t = 0$. Therefore, $\lim_{t \searrow 0} \dist(x+tv, \sparse{n}{s} \cap \R_+^n)/t = 0$ and it follows that $v \in \tancone{\sparse{n}{s} \cap \R_+^n}{x}$ and $v$ is geometrically derivable.

We now establish the inclusion $\subseteq$. Assume that $v$ is not in the right-hand side.
We first consider the case where there exists $i \in \{1, \dots, n\} \setminus \supp(x)$ such that $v_i < 0$. Then, for all $t \in \Big(0, \frac{\dist(x, \FixedSparsity{n}{j-1} \cap \R_+^n)}{2\norm{v}_\infty}\Big)$, $x+tv \not\in \sparse{n}{s} \cap \R_+^n$ and $\dist(x+tv, \sparse{n}{s} \cap \R_+^n)/t \ge -v_i$. Therefore, $\lim_{t \searrow 0} \dist(x+tv, \sparse{n}{s} \cap \R_+^n)/t = -v_i > 0$ and it follows that $v \not\in \tancone{\sparse{n}{s} \cap \R_+^n}{x}$.
We now consider the case where $|\supp(x) \cup \supp(v)| > s$. Then, $\supp(v) \setminus \supp(x) \ne \emptyset$ and, for all $t \in \Big(0, \frac{\dist(x, \FixedSparsity{n}{j-1} \cap \R_+^n)}{2\norm{v}_\infty}\Big)$, $x+tv \not\in \sparse{n}{s} \cap \R_+^n$ and $\dist(x+tv, \sparse{n}{s} \cap \R_+^n)/t \ge \min_{i \in \supp(v) \setminus \supp(x)} |v_i|$. Therefore, $\lim_{t \searrow 0} \dist(x+tv, \sparse{n}{s} \cap \R_+^n)/t \ge \min_{i \in \supp(v) \setminus \supp(x)} |v_i| > 0$, which shows that $v \not\in \tancone{\sparse{n}{s} \cap \R_+^n}{x}$.
\end{proof}

\begin{algorithm}[H]
\caption{Projection onto the tangent cone to the set of nonnegative sparse vectors}
\label{algo:ProjectionTangentConeNonnegativeSparseVectors}
\begin{algorithmic}[1]
\Require
$(n, s, x)$ where $n$ and $s$ are positive integers such that $s < n$, and $x \in \sparse{n}{s} \cap \R_+^n$.
\Input
$v \in \R^n$.
\Output
$w \in \proj{\tancone{\sparse{n}{s} \cap \R_+^n}{x}}{v}$.

\State
$w \gets v$;
\For
{$i \in \supp(v) \setminus \supp(x)$}
	\If
	{$v_i < 0$}
		\State
		$w_i \gets 0$;
	\EndIf
\EndFor
\While
{$|\supp(x) \cup \supp(w)| > s$}
	\State
	Choose $i \in \argmin_{j \in \supp(w) \setminus \supp(x)} |w_j|$;
	\State
	$w_i \gets 0$;
\EndWhile
\State
Return $w$.
\end{algorithmic}
\end{algorithm}

Proposition~\ref{prop:GlobalSecondOrderUpperBoundDistanceToNonnegativeSparseVectorsFromTangentLine} shows that $\sparse{n}{s} \cap \R_+^n$ satisfies the second statement of Theorem~\ref{thm:ExamplesStratifiedSetsSatisfyingMainAssumption}.

\begin{proposition}
\label{prop:GlobalSecondOrderUpperBoundDistanceToNonnegativeSparseVectorsFromTangentLine}
For all $x \in \sparse{n}{s} \cap \R_+^n \setminus \{0\}^n$,
\begin{equation*}
\sup_{v \in \tancone{\sparse{n}{s} \cap \R_+^n}{x} \setminus \{0\}^n} \frac{\dist(x+v, \sparse{n}{s} \cap \R_+^n)}{\norm{v}^2} \in \left[\frac{1}{4}\tilde{u}(x), \tilde{u}(x)\right],
\end{equation*}
where $\tilde{u}(x) := \frac{1}{\dist(x, \FixedSparsity{n}{i-1} \cap \R_+^n)}$ if $x \in \FixedSparsity{n}{i} \cap \R_+^n$ with $i \in \{1, \dots, s\}$.
\end{proposition}

\begin{proof}
Let $x \in \FixedSparsity{n}{j} \cap \R_+^n$ with $j \in \{1, \dots, s\}$. We first establish the upper bound. If $\norm{v} < \frac{1}{\tilde{u}(x)}$, then $x+v \in \sparse{n}{s} \cap \R_+^n$. Indeed:
\begin{itemize}
\item for all $i \in \supp(x) \setminus \supp(v)$, $x_i+v_i = x_i \ge \dist(x, \FixedSparsity{n}{j-1} \cap \R_+^n)$;
\item for all $i \in \supp(v) \setminus \supp(x)$, $x_i+v_i = v_i \in (0, \dist(x, \FixedSparsity{n}{j-1} \cap \R_+^n))$;
\item for all $i \in \supp(x) \cap \supp(v)$, $x_i+v_i \ge x_i-|v_i| > 0$;
\item for all $i \in \{1, \dots, n\} \setminus (\supp(x) \cup \supp(v))$, $x_i+v_i = 0$.
\end{itemize}
Thus, if $\norm{v} < \frac{1}{\tilde{u}(x)}$, then $\dist(x+v, \sparse{n}{s} \cap \R_+^n) = 0$ and $\dist(x+v, \sparse{n}{s} \cap \R_+^n)/\norm{v}^2 = 0 \le \frac{1}{\tilde{u}(x)}$. If $\norm{v} \ge \frac{1}{\tilde{u}(x)}$, then
\begin{equation*}
\frac{\dist(x+v, \sparse{n}{s} \cap \R_+^n)}{\norm{v}^2}
\le \frac{\norm{(x+v)-x}}{\norm{v}^2}
= \frac{1}{\norm{v}}
\le \tilde{u}(x).
\end{equation*}
The lower bound follows from the fact that, if $i \in \supp(x)$, $x_i = \dist(x, \FixedSparsity{n}{j-1} \cap \R_+^n)$, and $v := (-2x_i\delta_{i, k})_{k \in \{1, \dots, n\}} \in \tancone{\sparse{n}{s} \cap \R_+^n}{x} \setminus \{0\}^n$, then $\dist(x+v, \sparse{n}{s} \cap \R_+^n)/\norm{v}^2 = \tilde{u}(x)/4$.
\end{proof}

Proposition~\ref{prop:ContinuityTangentConeStratumNonnegativeSparseVectors} states that $\sparse{n}{s} \cap \R_+^n$ satisfies condition~3 of Assumption~\ref{assumption:Stratification}.

\begin{proposition}
\label{prop:ContinuityTangentConeStratumNonnegativeSparseVectors}
For every $i \in \{0, \dots, s\}$, the correspondence $\tancone{\sparse{n}{s} \cap \R_+^n}{\cdot}$ is continuous at every $x \in \FixedSparsity{n}{i} \cap \R_+^n$ relative to $\FixedSparsity{n}{i} \cap \R_+^n$.
\end{proposition}

\begin{proof}
The result is clear if $i = 0$ since $\FixedSparsity{n}{0} \cap \R_+^n = \{0\}^n$. Let us therefore consider $i \in \{1, \dots, s\}$. We have to prove that, for every sequence $(x^j)_{j \in \N}$ in $\FixedSparsity{n}{i} \cap \R_+^n$ converging to $x \in \FixedSparsity{n}{i} \cap \R_+^n$, it holds that
\begin{equation*}
\outlim_{j \to \infty} \tancone{\sparse{n}{s} \cap \R_+^n}{x^j}
\subseteq \tancone{\sparse{n}{s} \cap \R_+^n}{x}
\subseteq \inlim_{j \to \infty} \tancone{\sparse{n}{s} \cap \R_+^n}{x^j}.
\end{equation*}
Let $x \in \FixedSparsity{n}{i} \cap \R_+^n$. Then, for all $y \in \ball(x, \dist(x, \FixedSparsity{n}{i-1} \cap \R_+^n)) \cap (\FixedSparsity{n}{i} \cap \R_+^n)$, $\supp(y) = \supp(x)$ and thus, by Proposition~\ref{prop:TangentConeNonnegativeSparseVectors}, $\tancone{\sparse{n}{s} \cap \R_+^n}{y} = \tancone{\sparse{n}{s} \cap \R_+^n}{x}$. Thus, the result follows from the fact that a sequence $(x^j)_{j \in \N}$ in $\FixedSparsity{n}{i} \cap \R_+^n$ converging to $x$ contains finitely many elements in $\FixedSparsity{n}{i} \cap \R_+^n \setminus \ball(x, \dist(x, \FixedSparsity{n}{i-1} \cap \R_+^n))$.
\end{proof}

\subsubsection{Normal cones to the set of nonnegative sparse vectors}
\label{subsubsec:NormalConesNonnegativeSparseVectors}
In this section, we determine the regular normal cone to $\sparse{n}{s} \cap \R_+^n$ (Proposition~\ref{prop:RegularNormalConeNonnegativeSparseVectors}). Based on that, we deduce the normal cone and the Clarke normal cone to $\sparse{n}{s} \cap \R_+^n$ (Proposition~\ref{prop:NormalConeNonnegativeSparseVectors} and Corollary~\ref{coro:ClarkeNormalConeNonnegativeSparseVectors}), and we prove that the sets of apocalyptic and serendipitous points of $\sparse{n}{s} \cap \R_+^n$ both equal $\StrictSparsity{n}{s} \cap \R_+^n$ (Proposition~\ref{prop:NonnegativeSparseVectorsApocalypticSerendipitousPoints}).

\begin{proposition}[regular normal cone to the set of nonnegative sparse vectors]
\label{prop:RegularNormalConeNonnegativeSparseVectors}
For all $x \in \sparse{n}{s} \cap \R_+^n$,
\begin{equation*}
\regnorcone{\sparse{n}{s} \cap \R_+^n}{x}
= \left\{\begin{array}{ll}
\{w \in \R^n \mid \supp(w) \subseteq \{1, \dots, n\} \setminus \supp(x)\} & \text{if } x \in \FixedSparsity{n}{s} \cap \R_+^n,\\[1mm]
\{w \in \R_-^n \mid \supp(w) \subseteq \{1, \dots, n\} \setminus \supp(x)\} & \text{if } x \in \StrictSparsity{n}{s} \cap \R_+^n.
\end{array}\right.
\end{equation*}
\end{proposition}

\begin{proof}
The proof is based on Proposition~\ref{prop:TangentConeNonnegativeSparseVectors}. Let $x \in \sparse{n}{s} \cap \R_+^n$. By \eqref{eq:RegularNormalCone},
\begin{equation*}
\regnorcone{\sparse{n}{s} \cap \R_+^n}{x}
= \left\{w \in \R^n \mid \ip{w}{v} \le 0 \; \forall v \in \tancone{\sparse{n}{s} \cap \R_+^n}{x}\right\}.
\end{equation*}
Assume that $x \in \StrictSparsity{n}{s} \cap \R_+^n$. Then, for all $i \in \{1, \dots, n\}$, $v := (\delta_{i, j})_{j \in \{1, \dots, n\}} \in \tancone{\sparse{n}{s} \cap \R_+^n}{x}$ and, for all $w \in \R^n$, $\ip{w}{v} = w_i$. Moreover, for all $i \in \supp(x)$, $v := (-\delta_{i, j})_{j \in \{1, \dots, n\}} \in \tancone{\sparse{n}{s} \cap \R_+^n}{x}$ and, for all $w \in \R^n$, $\ip{w}{v} = -w_i$. Thus, $\regnorcone{\sparse{n}{s} \cap \R_+^n}{x} = \{w \in \R_-^n \mid \supp(w) \subseteq \{1, \dots, n\} \setminus \supp(x)\}$.
Assume now that $x \in \FixedSparsity{n}{s} \cap \R_+^n$. Then, for all $v \in \tancone{\sparse{n}{s} \cap \R_+^n}{x}$ and all $i \in \{1, \dots, n\} \setminus \supp(x)$, $v_i = 0$. Thus, for all $w \in \R^n$ and all $v \in \tancone{\sparse{n}{s} \cap \R_+^n}{x}$, $\ip{w}{v} = \sum_{i \in \supp(x)} w_i v_i$. Since, for all $i \in \supp(x)$, $v := (\delta_{i, j})_{j \in \{1, \dots, n\}} \in \tancone{\sparse{n}{s} \cap \R_+^n}{x}$, $-v \in \tancone{\sparse{n}{s} \cap \R_+^n}{x}$, and, for all $w \in \R^n$, $\ip{w}{v} = w_i$, we have $\regnorcone{\sparse{n}{s} \cap \R_+^n}{x} \subseteq \{w \in \R^n \mid \supp(w) \subseteq \{1, \dots, n\} \setminus \supp(x)\}$. The converse inclusion also holds, and the result follows.
\end{proof}

\begin{proposition}[{normal cone to the set of nonnegative sparse vectors \cite[Theorem~3.4]{Tam2017}}]
\label{prop:NormalConeNonnegativeSparseVectors}
For all $x \in \sparse{n}{s} \cap \R_+^n$,
\begin{equation*}
\norcone{\sparse{n}{s} \cap \R_+^n}{x} = \regnorcone{\sparse{n}{s} \cap \R_+^n}{x} \cup \{w \in \sparse{n}{n-s} \mid \supp(w) \subseteq \{1, \dots, n\} \setminus \supp(x)\}.
\end{equation*}
In particular, $\sparse{n}{s} \cap \R_+^n$ is not Clarke regular on $\StrictSparsity{n}{s} \cap \R_+^n$.
\end{proposition}

\begin{proof}
We provide an alternative proof to the one of \cite[Theorem~3.4]{Tam2017}. This argument is based on the definition \eqref{eq:NormalCone} of the normal cone and is used again in the proof of Proposition~\ref{prop:NonnegativeSparseVectorsApocalypticSerendipitousPoints}.

By \cite[Example~6.8]{RockafellarWets}, the result follows from Proposition~\ref{prop:RegularNormalConeNonnegativeSparseVectors} if $x \in \FixedSparsity{n}{s} \cap \R_+^n$.
Let $x \in \StrictSparsity{n}{s} \cap \R_+^n$.
We first establish the inclusion $\subseteq$. Let $(x^k)_{k \in \N}$ be a sequence in $\sparse{n}{s} \cap \R_+^n$ converging to $x$. We have to prove that
\begin{equation*}
\outlim_{k \to \infty} \regnorcone{\sparse{n}{s} \cap \R_+^n}{x^k} \subseteq  \regnorcone{\sparse{n}{s} \cap \R_+^n}{x} \cup \{w \in \sparse{n}{n-s} \mid \supp(w) \subseteq \{1, \dots, n\} \setminus \supp(x)\}.
\end{equation*}
Let $w \in \outlim_{k \to \infty} \regnorcone{\sparse{n}{s} \cap \R_+^n}{x^k}$. Then, $w$ is an accumulation point of a sequence $(w^k)_{k \in \N}$ such that, for all $k \in \N$, $w^k \in \regnorcone{\sparse{n}{s} \cap \R_+^n}{x^k}$.
If $(x^k)_{k \in \N}$ contains finitely many elements in $\FixedSparsity{n}{s}$, then $w \in \regnorcone{\sparse{n}{s} \cap \R_+^n}{x}$. Indeed, in that case, there exists a strictly increasing sequence $(k_l)_{l \in \N}$ in $\N$ such that $(w^{k_l})_{l \in \N}$ converges to $w$ and, for all $l \in \N$, $x^{k_l} \in \StrictSparsity{n}{s}$, $\supp(x) \subseteq \supp(x^{k_l})$, and $\supp(w) \subseteq \supp(w^{k_l})$. Thus, for all $l \in \N$, since $w^{k_l} \in \regnorcone{\sparse{n}{s} \cap \R_+^n}{x^{k_l}}$, it holds that $w^{k_l} \in \R_-^n$ and $\supp(w^{k_l}) \subseteq \{1, \dots, n\} \setminus \supp(x^{k_l})$. Therefore, $w \in \R_-^n$ and $\supp(w) \subseteq \supp(w^{k_0}) \subseteq \{1, \dots, n\} \setminus \supp(x^{k_0}) \subseteq \{1, \dots, n\} \setminus \supp(x)$, which shows that $w \in \regnorcone{\sparse{n}{s} \cap \R_+^n}{x}$.
If $(x^k)_{k \in \N}$ contains infinitely many elements in $\FixedSparsity{n}{s}$, then $w \in \sparse{n}{n-s}$ and $\supp(w) \subseteq \{1, \dots, n\} \setminus \supp(x)$. Indeed, in that case, there exists a strictly increasing sequence $(k_l)_{l \in \N}$ in $\N$ such that $(w^{k_l})_{l \in \N}$ converges to $w$ and, for all $l \in \N$, $x^{k_l} \in \FixedSparsity{n}{s}$, $\supp(x) \subseteq \supp(x^{k_l})$, and $\supp(w) \subseteq \supp(w^{k_l})$. Thus, for all $l \in \N$, since $w^{k_l} \in \regnorcone{\sparse{n}{s} \cap \R_+^n}{x^{k_l}}$, it holds that $\supp(w^{k_l}) \subseteq \{1, \dots, n\} \setminus \supp(x^{k_l})$. Therefore, $\supp(w) \subseteq \supp(w^{k_0}) \subseteq \{1, \dots, n\} \setminus \supp(x^{k_0}) \subseteq \{1, \dots, n\} \setminus \supp(x)$ and $|\supp(w)| \le n-s$.

We now establish the inclusion $\supseteq$. The inclusion $\norcone{\sparse{n}{s} \cap \R_+^n}{x} \supseteq \regnorcone{\sparse{n}{s} \cap \R_+^n}{x}$ holds by definition of the normal cone. Let $w \in \sparse{n}{n-s}$ such that $\supp(w) \subseteq \{1, \dots, n\} \setminus \supp(x)$. Let $I(x) \subseteq \{1, \dots, n\} \setminus (\supp(w) \cup \supp(x))$ such that $|I(x)| = s-|\supp(x)|$; this is possible since $|\{1, \dots, n\} \setminus (\supp(w) \cup \supp(x))| = n-|\supp(w)|-|\supp(x)| \ge s-|\supp(x)|$. For all $k \in \N$ and all $i \in \{1, \dots, n\}$, let
\begin{equation*}
x_i^k := \left\{\begin{array}{ll}
\frac{1}{k+1} & \text{if } i \in I(x),\\
x_i & \text{otherwise}.
\end{array}\right.
\end{equation*}
Then, for all $k \in \N$, $\supp(x^k) = \supp(x) \cup I(x)$, thus $\supp(w) \subseteq \{1, \dots, n\} \setminus \supp(x^k)$, and therefore $w \in \regnorcone{\sparse{n}{s} \cap \R_+^n}{x^k}$. It follows that $w \in \outlim_{k \to \infty} \regnorcone{\sparse{n}{s} \cap \R_+^n}{x^k}$.
\end{proof}

\begin{corollary}[Clarke normal cone to the set of nonnegative sparse vectors]
\label{coro:ClarkeNormalConeNonnegativeSparseVectors}
For all $x \in \sparse{n}{s} \cap \R_+^n$,
\begin{equation*}
\connorcone{\sparse{n}{s} \cap \R_+^n}{x} = \{w \in \R^n \mid \supp(w) \subseteq \{1, \dots, n\} \setminus \supp(x)\}.
\end{equation*}
\end{corollary}

By Proposition~\ref{prop:NormalConeNonnegativeSparseVectors}, $\sparse{n}{s} \cap \R_+^n$ is not Clarke regular on $\StrictSparsity{n}{s} \cap \R_+^n$. Proposition~\ref{prop:NonnegativeSparseVectorsApocalypticSerendipitousPoints} states that every point of $\StrictSparsity{n}{s} \cap \R_+^n$ is apocalyptic, which is a stronger result by \cite[Corollary~2.15]{LevinKileelBoumal2022}.

\begin{proposition}
\label{prop:NonnegativeSparseVectorsApocalypticSerendipitousPoints}
The set of apocalyptic points of $\sparse{n}{s} \cap \R_+^n$ and the set of serendipitous points of $\sparse{n}{s} \cap \R_+^n$ both equal $\StrictSparsity{n}{s} \cap \R_+^n$.
\end{proposition}

\begin{proof}
We use Proposition~\ref{prop:CharacterizationApocalypticSerendipitousPoint}.
Let $x \in \FixedSparsity{n}{s} \cap \R_+^n$ and $(x^k)_{k \in \N}$ be a sequence in $\sparse{n}{s} \cap \R_+^n$ converging to $x$. Since $(x^k)_{k \in \N}$ contains finitely many elements not in $\ball(x, \dist(x, \FixedSparsity{n}{s-1} \cap \R_+^n)) \subseteq \FixedSparsity{n}{s} \cap \R_+^n$, we can assume that $(x^k)_{k \in \N}$ is in $\FixedSparsity{n}{s} \cap \R_+^n$. Therefore, by Proposition~\ref{prop:ContinuityTangentConeStratumNonnegativeSparseVectors}, $\outlim_{k \to \infty} \tancone{\sparse{n}{s} \cap \R_+^n}{x^k} = \tancone{\sparse{n}{s} \cap \R_+^n}{x}$, and thus $\big(\outlim_{k \to \infty} \tancone{\sparse{n}{s} \cap \R_+^n}{x^k}\big)^* = \regnorcone{\sparse{n}{s} \cap \R_+^n}{x}$. It follows that $x$ is neither apocalyptic nor serendipitous.

Let $x \in \FixedSparsity{n}{j} \cap \R_+^n$ with $j \in \{0, \dots, s-1\}$. Let $I(x) \subseteq \{1, \dots, n\} \setminus \supp(x)$ such that $|I(x)| = s-j$. For all $k \in \N$, define
\begin{equation*}
x_i^k := \left\{\begin{array}{ll}
x_i & \text{if } i \in \{1, \dots, n\} \setminus I(x),\\
\frac{1}{k+1} & \text{if } i \in I(x).
\end{array}\right.
\end{equation*}
Then, $(x^k)_{k \in \N}$ converges to $x$. Moreover, for all $k \in \N$, since $\supp(x^k) = \supp(x) \cup I(x)$, it holds that $x^k \in \FixedSparsity{n}{s} \cap \R_+^n$ and
\begin{equation*}
\tancone{\sparse{n}{s} \cap \R_+^n}{x^k} = \tancone{\sparse{n}{s} \cap \R_+^n}{x^0}.
\end{equation*}
Thus,
\begin{equation*}
\setlim_{k \to \infty} \tancone{\sparse{n}{s} \cap \R_+^n}{x^k} = \tancone{\sparse{n}{s} \cap \R_+^n}{x^0}.
\end{equation*}
Therefore,
\begin{equation*}
\Big(\setlim_{k \to \infty} \tancone{\sparse{n}{s} \cap \R_+^n}{x^k}\Big)^* = \regnorcone{\sparse{n}{s} \cap \R_+^n}{x^0}.
\end{equation*}
The result follows from Proposition~\ref{prop:CharacterizationApocalypticSerendipitousPoint} since, by Proposition~\ref{prop:RegularNormalConeNonnegativeSparseVectors}, neither of $\regnorcone{\sparse{n}{s} \cap \R_+^n}{x}$ and $\regnorcone{\sparse{n}{s} \cap \R_+^n}{x^0}$ is a subset of the other.
\end{proof}

\subsubsection{$\ppgd$ following an apocalypse on $\sparse{n}{s} \cap \R_+^n$}
\label{subsubsec:NonnegativeSparseVectorsP2GDapocalypseExample}
On $\sparse{n}{s} \cap \R_+^n$, $\ppgd$ follows the same apocalypse as the one on $\sparse{n}{s}$ described in Proposition~\ref{prop:SparseVectorsP2GDapocalypseExample}. This is because, for all $k \in \N$ and all $j \in \{1, \dots ,s\}$, it holds that $x^{k, j} \in \sparse{n}{s} \cap \R_+^n$ and $\proj{\tancone{\sparse{n}{s}}{x^{k, j}}}{-\nabla f(x^{k, j})} = \proj{\tancone{\sparse{n}{s} \cap \R_+^n}{x^{k, j}}}{-\nabla f(x^{k, j})}$.

\subsection{The real determinantal variety}
\label{subsec:RealDeterminantalVariety}
In this section, $\mathcal{E} := \R^{m \times n}$ and $C := \R_{\le r}^{m \times n}$ for some positive integers $m$, $n$, and $r < \min\{m,n\}$, $\R^{m \times n}$ is endowed with the \emph{Frobenius inner product} $\ip{X}{Y} := \tr Y^\tp X$, and $\norm{\cdot}$ denotes the Frobenius norm. The spectral norm is denoted by $\norm{\cdot}_2$.

In Section~\ref{subsubsec:StratificationRealDeterminantalVariety}, we review the stratification of $\R_{\le r}^{m \times n}$ by the rank, which satisfies conditions~1(a) and 1(b) of Assumption~\ref{assumption:Stratification}. We also recall basic facts about singular values showing that the stratification satisfies condition~1(c) and providing a formula to project onto $\R_{\le r}^{m \times n}$ and its strata. In Section~\ref{subsubsec:TangentConeRealDeterminantalVariety}, we review an explicit description of the tangent cone to $\R_{\le r}^{m \times n}$ and a formula to project onto it (Proposition~\ref{prop:ProjectionOntoTangentConeRealDeterminantalVariety}). Based on this description, we prove that the second statement of Theorem~\ref{thm:ExamplesStratifiedSetsSatisfyingMainAssumption} holds (Proposition~\ref{prop:GlobalSecondOrderUpperBoundDistanceToRealDeterminantalVarietyFromTangentLine}). In Section~\ref{subsubsec:NormalConesRealDeterminantalVariety}, we review the regular normal cone, the normal cone, and the Clarke normal cone to $\R_{\le r}^{m \times n}$ (Proposition~\ref{prop:NormalConesRealDeterminantalVariety}) and prove that $\R_{\le r}^{m \times n}$ has no serendipitous point (Proposition~\ref{prop:RealDeterminantalVarietyNoSerendipitousPoint}). In Section~\ref{subsubsec:P2GDRrealDeterminantalVariety}, we present an alternative version of the $\ppgdr$ map on $\R_{\le r}^{m \times n}$ (Algorithm~\ref{algo:P2GDRmapRealDeterminantalVariety}) and show that the general theory developed in Section~\ref{sec:ProposedAlgorithmConvergenceAnalysis} also applies to that version (Proposition~\ref{prop:P2GDRmapRealDeterminantalVarietyPolak}). We notably deduce Corollary~\ref{coro:P2GDRrealDeterminantalVarietyPolakConvergence}. In Section~\ref{subsubsec:PracticalImplementationLKB22Algo1RealDeterminantalVariety}, we discuss the practical implementation of \cite[Algorithm~1]{LevinKileelBoumal2022}. In Section~\ref{subsubsec:ComparisonSixOptimizationAlgorithmsRealDeterminantalVariety}, we compare $\pgd$, $\ppgd$, $\rfd$, $\ppgdr$, $\rfdr$, and \cite[Algorithm~1]{LevinKileelBoumal2022} based on the computational cost per iteration (Table~\ref{tab:ComparisonComputationalCostPerIterationAlgorithmsRealDeterminantalVariety}) and the convergence guarantees (Table~\ref{tab:ConvergenceAlgorithmsRealDeterminantalVariety}). In Section~\ref{subsubsec:LKB22instance}, we numerically compare $\pgd$, $\ppgd$, $\ppgdr$, and \cite[Algorithm~1]{LevinKileelBoumal2022} on the example of apocalypse presented in \cite[\S 2.2]{LevinKileelBoumal2022}. Finally, in Section~\ref{subsubsec:P2GDfollowingApocalypseSize2*2}, we give an example of $\ppgd$ following an apocalypse on $\R_{\le 1}^{2 \times 2}$.

\subsubsection{Stratification of the determinantal variety}
\label{subsubsec:StratificationRealDeterminantalVariety}
The rank stratifies the determinantal variety $\R_{\le r}^{m \times n}$:
\begin{equation*}
\R_{\le r}^{m \times n} = \bigcup_{i=0}^r \R_i^{m \times n}
\end{equation*}
where, for every $i \in \{0, \dots, r\}$,
\begin{equation}
\label{eq:RealFixedRankManifold}
\R_i^{m \times n} := \{X \in \R^{m \times n} \mid \rank X = i\}
\end{equation}
is the smooth manifold of real $m \times n$ rank-$i$ matrices \cite[Proposition~4.1]{HelmkeShayman1995}. Observe that $\R_{\le 0}^{m \times n} = \R_0^{m \times n} = \{0_{m \times n}\}$. Thus, $\R_{\le r}^{m \times n}$ satisfies condition~1(a) of Assumption~\ref{assumption:Stratification}. By \cite[Proposition~2.1]{OlikierAbsil2022}, condition~1(b) is satisfied too.
To establish condition~1(c), we first review basic facts about singular values and rank reduction.

In what follows, the singular values of $X \in \R^{m \times n}$ are denoted by $\sigma_1(X) \ge \dots \ge \sigma_{\min\{m,n\}}(X) \ge 0$, as in \cite[\S 2.4.1]{GolubVanLoan}. Moreover, if $X \ne 0_{m \times n}$, then $\sigma_1(X)$ and $\sigma_{\rank X}(X)$ are respectively denoted by $\sigma_{\max}(X)$ and $\sigma_{\min}(X)$.
By reducing the rank of $X \in \R^{m \times n} \setminus \{0_{m \times n}\}$, we mean computing an element of $\proj{\R_{\le \ushort{r}}^{m \times n}}{X}$ for some nonnegative integer $\ushort{r} < \rank X$. According to the Eckart--Young theorem \cite{EckartYoung1936}, this can be achieved by truncating an SVD of $X$. In particular, for every nonnegative integer $\ushort{r} < \min\{m,n\}$ and every $X \in \R^{m \times n}$:
\begin{enumerate}
\item if $\rank X \le \ushort{r}$, then $\proj{\R_{\le \ushort{r}}^{m \times n}}{X} = X$;
\item if $\rank X > \ushort{r}$, then $\dist(X, \R_{\le \ushort{r}}^{m \times n}) = \dist(X, \R_{\ushort{r}}^{m \times n}) = \sqrt{\sum_{j=\ushort{r}+1}^{\rank X} \sigma_j^2(X)}$ and $\proj{\R_{\le \ushort{r}}^{m \times n}}{X} = \proj{\R_{\ushort{r}}^{m \times n}}{X}$.
\end{enumerate}
In particular, condition~1(c) of Assumption~\ref{assumption:Stratification} is satisfied.

\subsubsection{Tangent cone to the determinantal variety}
\label{subsubsec:TangentConeRealDeterminantalVariety}
We review in Proposition~\ref{prop:ProjectionOntoTangentConeRealDeterminantalVariety} formulas describing $\tancone{\R_{\le r}^{m \times n}}{X}$ and $\proj{\tancone{\R_{\le r}^{m \times n}}{X}}{Z}$ for every $X \in \R_{\le r}^{m \times n}$ and every $Z \in \R^{m \times n}$ based on orthonormal bases of $\im X$, $\im X^\tp$, and their orthogonal complements.
Those formulas can be obtained from~\eqref{eq:TangentConeSequence}.
For every $i, q \in \N$ such that $i \le q$,
\begin{equation}
\label{eq:StiefelManifold}
\st(i, q) := \{U \in \R^{q \times i} \mid U^\tp U = I_i\}
\end{equation}
is a Stiefel manifold \cite[\S 3.3.2]{AbsilMahonySepulchre}. For every $q \in \N$, $\mathrm{O}(q) := \st(q, q)$ is an orthogonal group.

\begin{proposition}[{tangent cone to $\R_{\le r}^{m \times n}$ \cite[Theorem~3.2 and Corollary~3.3]{SchneiderUschmajew2015}}]
\label{prop:ProjectionOntoTangentConeRealDeterminantalVariety}
Let $\ushort{r} \in \{0, \dots, r\}$, $X \in \R_{\ushort{r}}^{m \times n}$, $U \in \st(\ushort{r},m)$, $U_\perp \in \st(m-\ushort{r},m)$, $V \in \st(\ushort{r},n)$, $V_\perp \in \st(n-\ushort{r},n)$, $\im U = \im X$, $\im U_\perp = (\im X)^\perp$, $\im V = \im X^\tp$, and $\im V_\perp = (\im X^\tp)^\perp$.
Then,
\begin{equation*}
\tancone{\R_{\le r}^{m \times n}}{X} = [U \; U_\perp] \begin{bmatrix} \R^{\ushort{r} \times \ushort{r}} & \R^{\ushort{r} \times n-\ushort{r}} \\ \R^{m-\ushort{r} \times \ushort{r}} & \R_{\le r-\ushort{r}}^{m-\ushort{r} \times n-\ushort{r}} \end{bmatrix} [V \; V_\perp]^\tp.
\end{equation*}
Moreover, if $Z \in \R^{m \times n}$ is written as
\begin{equation*}
Z = [U \; U_\perp] \begin{bmatrix} A & B \\ D & E \end{bmatrix} [V \; V_\perp]^\tp
\end{equation*}
with $A = U^\tp Z V$, $B = U^\tp Z V_\perp$, $D = U_\perp^\tp Z V$, and $E = U_\perp^\tp Z V_\perp$, then
\begin{equation*}
\proj{\tancone{\R_{\le r}^{m \times n}}{X}}{Z} = [U \; U_\perp] \begin{bmatrix} A & B \\ D & \proj{\R_{\le r-\ushort{r}}^{m-\ushort{r} \times n-\ushort{r}}}{E} \end{bmatrix} [V \; V_\perp]^\tp
\end{equation*}
and
\begin{equation}
\label{eq:NormProjectionTangentConeDeterminantalVariety}
\norm{Z}
\ge \norm{\proj{\tancone{\R_{\le r}^{m \times n}}{X}}{Z}}
\ge \sqrt{\frac{r-\ushort{r}}{\min\{m,n\}-\ushort{r}}} \norm{Z}.
\end{equation}
\end{proposition}

For efficiency note the following.
\begin{enumerate}
\item In practice, the projection onto $\tancone{\R_{\le r}^{m \times n}}{X}$ can be computed by \cite[Algorithm~2]{SchneiderUschmajew2015}. This does not rely on $U_\perp$ and $V_\perp$, which are huge in the frequently encountered case where $r \ll \min\{m,n\}$.

\item For all $X \in \R_{\le r}^{m \times n}$ and all $Z \in \tancone{\R_{\le r}^{m \times n}}{X}$, $X+Z \in \R_{\le 2r}^{m \times n}$. Indeed, if
\begin{equation*}
X = [U \; U_\perp] \diag(\Sigma, 0_{m-\ushort{r} \times n-\ushort{r}}) [V \; V_\perp]^\tp
\end{equation*}
is an SVD and $Z \in \tancone{\R_{\le r}^{m \times n}}{X}$ is written as
\begin{equation*}
Z = [U \; U_\perp] \begin{bmatrix} A & B \\ D & E \end{bmatrix} [V \; V_\perp]^\tp
\end{equation*}
with $A \in \R^{\ushort{r} \times \ushort{r}}$, $B \in \R^{\ushort{r} \times n-\ushort{r}}$, $D \in \R^{m-\ushort{r} \times \ushort{r}}$, and $E \in \R_{\le r-\ushort{r}}^{m-\ushort{r} \times n-\ushort{r}}$, then, by \cite[Proposition~3.1]{OlikierAbsil2022},
\begin{equation*}
X+Z = [U \; U_\perp] \begin{bmatrix} \Sigma+A & B \\ D & E \end{bmatrix} [V \; V_\perp]^\tp \in \R_{\le r+\ushort{r}}^{m \times n} \subseteq \R_{\le 2r}^{m \times n}.
\end{equation*}
\end{enumerate}

Proposition~\ref{prop:GlobalSecondOrderUpperBoundDistanceToRealDeterminantalVarietyFromTangentLine} shows that $\R_{\le r}^{m \times n}$ satisfies the second statement of Theorem~\ref{thm:ExamplesStratifiedSetsSatisfyingMainAssumption}.

\begin{proposition}
\label{prop:GlobalSecondOrderUpperBoundDistanceToRealDeterminantalVarietyFromTangentLine}
For all $X \in \R_{\le r}^{m \times n} \setminus \{0_{m \times n}\}$,
\begin{equation*}
\frac{\sqrt{5}-1}{r+1} \frac{1}{2\sigma_{\min}(X)}
\le \sup_{Z \in \tancone{\R_{\le r}^{m \times n}}{X} \setminus \{0_{m \times n}\}} \frac{\dist(X+Z, \R_{\le r}^{m \times n})}{\norm{Z}^2}
\le \frac{1}{2\sigma_{\min}(X)}.
\end{equation*}
\end{proposition}

\begin{proof}
Let $\ushort{r} := \rank X$. We first establish the upper bound. Let
\begin{equation*}
X =
[U \; U_\perp]
\begin{bmatrix}
\Sigma & \\ & 0_{m-\ushort{r} \times n-\ushort{r}}
\end{bmatrix}
[V \; V_\perp]^\tp
\end{equation*}
be an SVD, and $Z \in \tancone{\R_{\le r}^{m \times n}}{X} \setminus \{0_{m \times n}\}$. By Proposition~\ref{prop:ProjectionOntoTangentConeRealDeterminantalVariety}, there are $A \in \R^{\ushort{r} \times \ushort{r}}$, $B \in \R^{\ushort{r} \times n-\ushort{r}}$, $D \in \R^{m-\ushort{r} \times \ushort{r}}$, and $E \in \R_{\le r-\ushort{r}}^{m-\ushort{r} \times n-\ushort{r}}$ such that
\begin{equation*}
Z =
[U \; U_\perp]
\begin{bmatrix}
A & B \\ D & E
\end{bmatrix}
[V \; V_\perp]^\tp.
\end{equation*}
Define the function
\begin{equation*}
\gamma : [0,\infty) \to \R_{\le r}^{m \times n} : t \mapsto \big(U+t(U_\perp D + {\textstyle\frac{1}{2}}UA)\Sigma^{-1}\big) \Sigma \big(V+t(V_\perp B^\tp + {\textstyle\frac{1}{2}}VA^\tp)\Sigma^{-1}\big)^\tp + tU_\perp E V_\perp^\tp,
\end{equation*}
where the first term is inspired from \cite[(13)]{ZhouEtAl2016}; $\gamma$ is well defined since the ranks of the two terms are respectively upper bounded by $\ushort{r}$ and $r-\ushort{r}$.
For all $t \in [0,\infty)$,
\begin{equation*}
\gamma(t)
= X + t Z + \frac{t^2}{4}
[U \; U_\perp]
\begin{bmatrix}
A \\ 2D
\end{bmatrix}
\Sigma^{-1}
\begin{bmatrix}
A & 2B
\end{bmatrix}
[V \; V_\perp]^\tp
\end{equation*}
thus
\begin{equation*}
\dist(X+tZ, \R_{\le r}^{m \times n})
\le \norm{(X+tZ)-\gamma(t)}
= \frac{t^2}{4} \left\|\begin{bmatrix} A\Sigma^{-1}A & 2A\Sigma^{-1}B \\ 2D\Sigma^{-1}A & 4D\Sigma^{-1}B\end{bmatrix}\right\|.
\end{equation*}
Observe that
\begin{align*}
\left\|\begin{bmatrix} A\Sigma^{-1}A & 2A\Sigma^{-1}B \\ 2D\Sigma^{-1}A & 4D\Sigma^{-1}B\end{bmatrix}\right\|^2
&= \norm{A\Sigma^{-1}A}^2 + 4 \norm{A\Sigma^{-1}B}^2 + 4 \norm{D\Sigma^{-1}A}^2 + 16 \norm{D\Sigma^{-1}B}^2\\
&\le \norm{\Sigma^{-1}}_2^2 \left(\norm{A}^4 + 4 \norm{A}^2 \norm{B}^2 + 4 \norm{A}^2 \norm{D}^2 + 16 \norm{B}^2 \norm{D}^2\right)\\
&\le \norm{\Sigma^{-1}}_2^2 \norm{Z}^4 \max_{\substack{x, y, z \in \R \\ x^2+y^2+z^2 = 1}} x^4 + 4 x^2 y^2 + 4 x^2 z^2 + 16 y^2 z^2\\
&= 4 \norm{\Sigma^{-1}}_2^2 \norm{Z}^4\\
&= \frac{4}{\sigma_{\ushort{r}}^2(X)} \norm{Z}^4.
\end{align*}
Therefore, for all $t \in [0,\infty)$,
\begin{equation*}
\dist(X+tZ, \R_{\le r}^{m \times n}) \le t^2 \frac{1}{2 \sigma_{\ushort{r}}(X)} \norm{Z}^2.
\end{equation*}
Choosing $t = 1$ yields the upper bound.

We now establish the lower bound. Let
\begin{equation*}
X = [U \; U_\perp] \diag(\sigma_1, \dots, \sigma_{\ushort{r}}, 0_{m-\ushort{r} \times n-\ushort{r}}) [V \; V_\perp]^\tp
\end{equation*}
be an SVD, and observe that
\begin{equation*}
Z
:= [U \; U_\perp] \sigma_{\ushort{r}} \diag(0_{\ushort{r}-1}, \left[\begin{smallmatrix} 0 & 1 \\ 1 & 0 \end{smallmatrix}\right], I_{r-\ushort{r}}, 0_{m-r-1 \times n-r-1}) [V \; V_\perp]^\tp
\in \tancone{\R_{\le r}^{m \times n}}{X}.
\end{equation*}
The nonzero singular values of
\begin{equation*}
X+Z = [U \; U_\perp] \diag(\sigma_1, \dots, \sigma_{\ushort{r}-1}, \sigma_{\ushort{r}} \left[\begin{smallmatrix} 1 & 1 \\ 1 & 0 \end{smallmatrix}\right], \sigma_{\ushort{r}} I_{r-\ushort{r}}, 0_{m-r-1 \times n-r-1}) [V \; V_\perp]^\tp
\end{equation*}
are $\sigma_1, \dots, \sigma_{\ushort{r}}$ and the absolute values of the eigenvalues of $\left[\begin{smallmatrix} 1 & 1 \\ 1 & 0 \end{smallmatrix}\right]$ multiplied by $\sigma_{\ushort{r}}$, i.e., $\frac{\sqrt{5}+1}{2}\sigma_{\ushort{r}}$ and $\frac{\sqrt{5}-1}{2}\sigma_{\ushort{r}}$. Thus, $\dist(X+Z, \R_{\le r}^{m \times n}) = \frac{\sqrt{5}-1}{2}\sigma_{\ushort{r}}$, $\norm{Z}^2 = (r-\ushort{r}+2)\sigma_{\ushort{r}}^2 \le (r+1)\sigma_{\ushort{r}}^2$, and the lower bound follows.
\end{proof}

Proposition~\ref{prop:GlobalSecondOrderUpperBoundDistanceToRealDeterminantalVarietyFromTangentLine} can be related to geometric principles. Let $\gamma$ be a curve on the submanifold $\R_r^{m \times n}$ of $\R^{m \times n}$. In view of the Gauss formula along a curve~\cite[Corollary~8.3]{Lee2018}, the normal part of the acceleration of $\gamma$ is given by $\mathrm{I\!I}(\gamma',\gamma')$, where $\mathrm{I\!I}$ denotes the second fundamental form. In view of~\cite[\S 4]{FepponLermusiaux2018}, the largest principal curvature of $\R_r^{m \times n}$ at $X$ is $1/\sigma_r(X)$; hence $\|\mathrm{I\!I}(\gamma'(t),\gamma'(t))\| \leq \|\gamma'(t)\|^2/\sigma_r(\gamma(t))$, and the bound is attained when $\gamma'(t)$ is along the corresponding principal direction.

\subsubsection{Normal cones to the determinantal variety}
\label{subsubsec:NormalConesRealDeterminantalVariety}
In Proposition~\ref{prop:NormalConesRealDeterminantalVariety}, we review formulas describing $\regnorcone{\R_{\le r}^{m \times n}}{X}$, $\norcone{\R_{\le r}^{m \times n}}{X}$, and $\connorcone{\R_{\le r}^{m \times n}}{X}$ for every $X \in \R_{\le r}^{m \times n}$ based on orthonormal bases of $(\im X)^\perp$ and $(\im X^\tp)^\perp$. Then, in Proposition~\ref{prop:RealDeterminantalVarietyNoSerendipitousPoint}, we deduce that $\R_{\le r}^{m \times n}$ has no serendipitous point.

Proposition~\ref{prop:NormalConesRealDeterminantalVariety} shows that a point $X \in \R_{\ushort{r}}^{m \times n}$ with $\ushort{r} \in \{0, \dots, r\}$ is Mordukhovich stationary if and only if it is stationary for the problem of minimizing $f$ on $\R_{\le \ushort{r}}^{m \times n}$ and $\rank \nabla f(X) \le \min\{m,n\} - r$.

\begin{proposition}[normal cones to $\R_{\le r}^{m \times n}$]
\label{prop:NormalConesRealDeterminantalVariety}
Let $\ushort{r} \in \{0, \dots, r\}$, $X \in \R_{\ushort{r}}^{m \times n}$, $U_\perp \in \st(m-\ushort{r},m)$, $V_\perp \in \st(n-\ushort{r},n)$, $\im U_\perp = (\im X)^\perp$, and $\im V_\perp = (\im X^\tp)^\perp$.
If $\ushort{r} = r$, then
\begin{equation*}
\regnorcone{\R_{\le r}^{m \times n}}{X}
= \norcone{\R_{\le r}^{m \times n}}{X}
= \connorcone{\R_{\le r}^{m \times n}}{X}
= \norcone{\R_r^{m \times n}}{X}
= U_\perp \R^{m-r \times n-r} V_\perp^\tp.
\end{equation*}
If $\ushort{r} < r$, then
\begin{align}
\label{eq:RegularNormalConeRealDeterminantalVariety}
\regnorcone{\R_{\le r}^{m \times n}}{X}
&= \{0_{m \times n}\},\\
\nonumber
\norcone{\R_{\le r}^{m \times n}}{X}
&= \norcone{\R_{\ushort{r}}^{m \times n}}{X} \cap \R_{\le \min\{m,n\}-r}^{m \times n},\\
\nonumber
\connorcone{\R_{\le r}^{m \times n}}{X}
&= \norcone{\R_{\ushort{r}}^{m \times n}}{X}.
\end{align}
\end{proposition}

\begin{proof}
The tangent space to $\R_{\ushort{r}}$ is given in \cite[Proposition 4.1]{HelmkeShayman1995} and the normal space to $\R_{\ushort{r}}$ is its orthogonal complement. The regular normal cone, the normal cone, and the Clarke normal cone to $\R_{\le r}^{m \times n}$ are described respectively in \cite[Corollary~2.3]{HosseiniLukeUschmajew2019}, \cite[Theorem~3.1]{HosseiniLukeUschmajew2019}, and \cite[Corollary~3.2]{HosseiniLukeUschmajew2019}.
\end{proof}

\begin{proposition}
\label{prop:RealDeterminantalVarietyNoSerendipitousPoint}
No point of $\R_{\le r}^{m \times n}$ is serendipitous.
\end{proposition}

\begin{proof}
We use Proposition~\ref{prop:CharacterizationApocalypticSerendipitousPoint}. Let $X \in \R_{\le r}^{m \times n}$. Let us prove that $X$ is not serendipitous. Let $(X_i)_{i \in \N}$ be a sequence in $\R_{\le r}^{m \times n}$ converging to $X$.
If $\rank X = r$, then, by \cite[Proposition~4.3]{OlikierAbsil2022}, $\outlim_{i \to \infty} \tancone{\R_{\le r}^{m \times n}}{X_i} = \tancone{\R_{\le r}^{m \times n}}{X}$, and thus $\big(\outlim_{i \to \infty} \tancone{\R_{\le r}^{m \times n}}{X_i}\big)^* = \regnorcone{\R_{\le r}^{m \times n}}{X}$.
If $\rank X < r$, then
\begin{equation*}
\regnorcone{\R_{\le r}^{m \times n}}{X}
= \{0_{m \times n}\}
\subseteq \Big(\outlim_{i \to \infty} \tancone{\R_{\le r}^{m \times n}}{X_i}\Big)^*,
\end{equation*}
where the equality follows from \eqref{eq:RegularNormalConeRealDeterminantalVariety}.
\end{proof}

\subsubsection{A variant of the $\ppgdr$ map on the determinantal variety}
\label{subsubsec:P2GDRrealDeterminantalVariety}
In this section, we propose as Algorithm~\ref{algo:P2GDRmapRealDeterminantalVariety} the variant of the $\ppgdr$ map on $\R_{\le r}^{m \times n}$ obtained by measuring the distance from the input to the lower strata using the the spectral norm instead of the Frobenius norm. To this end, we recall from \cite[(5.4.5)]{GolubVanLoan} that, given $\Delta \in [0, \infty)$, the \emph{$\Delta$-rank} of $X \in \R^{m \times n}$, denoted $\rank_\Delta X$, is defined as the number of singular values of $X$ that are larger than $\Delta$:
\begin{equation*}
\rank_\Delta X := |\{i \in \{1, \dots, \min\{m,n\}\} \mid \sigma_i(X) > \Delta\}|.
\end{equation*}
Proposition~\ref{prop:P2GDRrealDeterminantalVariety} states that the minimum computed in line~\ref{algo:P2GDRmap:DeltaCloseStrata} of Algorithm~\ref{algo:P2GDRmap} equals the $\Delta$-rank if the distance is measured with respect to the spectral norm.
For every nonempty subset $\mathcal{S}$ of $\R^{m \times n}$ and every $X \in \R^{m \times n}$, let $d_2(X, \mathcal{S}) := \inf_{Y \in \mathcal{S}} \norm{X-Y}_2$.

\begin{proposition}
\label{prop:P2GDRrealDeterminantalVariety}
For all $\Delta \in [0, \infty)$ and all $X \in \R^{m \times n}$,
\begin{equation*}
\min\{j \in \{0, \dots, \rank X\} \mid d_2(X, \R_j^{m \times n}) \le \Delta\} = \rank_\Delta X.
\end{equation*}
\end{proposition}

\begin{proof}
Let $\Delta \in [0, \infty)$ and $X \in \R^{m \times n}$. By the Eckart--Young theorem \cite{EckartYoung1936}, for all $j \in \{0, \dots, \min\{m,n\}\}$,
\begin{equation*}
d_2(X, \R_j^{m \times n}) = \left\{\begin{array}{ll}
0 & \text{if } \rank X \le j,\\
\sigma_{j+1}(X) & \text{if } \rank X > j.
\end{array}\right.
\end{equation*}
Thus, the result is clear if $\rank_\Delta X = \rank X$. If $\rank_\Delta X < \rank X$, then
\begin{align*}
\rank_\Delta X
&= |\{j \in \{1, \dots, \rank X\} \mid \sigma_j(X) > \Delta\}|\\
&= \rank X - |\{j \in \{1, \dots, \rank X\} \mid \sigma_j(X) \le \Delta\}|\\
&= \rank X - |\{j \in \{0, \dots, \rank X-1\} \mid \sigma_{j+1}(X) \le \Delta\}|\\
&= \min\{j \in \{0, \dots, \rank X - 1\} \mid \sigma_{j+1}(X) \le \Delta\}\\
&= \min\{j \in \{0, \dots, \rank X - 1\} \mid d_2(X, \R_j^{m \times n}) \le \Delta\}\\
&= \min\{j \in \{0, \dots, \rank X\} \mid d_2(X, \R_j^{m \times n}) \le \Delta\}.
\qedhere
\end{align*}
\end{proof}

Given $X \in \R_{\le r}^{m \times n}$ as input, Algorithm~\ref{algo:P2GDRmapRealDeterminantalVariety} proceeds as follows: (i) it applies the $\ppgd$ map to $X$, thereby producing a point $\tilde{X}^0$, (ii) if $\rank_\Delta X < \rank X$, it applies the $\ppgd$ map to $\hat{X}^j \in \proj{\R_{\rank X - j}^{m \times n}}{X}$ for every $j \in \{1, \dots, \rank X - \rank_\Delta X\}$, then producing points $\tilde{X}^1, \dots, \tilde{X}^{\rank X - \rank_\Delta X}$, and (iii) it outputs a point among $\tilde{X}^0, \dots, \tilde{X}^{\rank X - \rank_\Delta X}$ that maximally decreases $f$.

\begin{algorithm}[H]
\caption{Variant of the $\ppgdr$ map on $\R_{\le r}^{m \times n}$}
\label{algo:P2GDRmapRealDeterminantalVariety}
\begin{algorithmic}[1]
\Require
$(f, r, \ushort{\alpha}, \bar{\alpha}, \beta, c, \Delta)$ where $f : \R^{m \times n} \to \R$ is differentiable with $\nabla f$ locally Lipschitz continuous, $r < \min\{m,n\}$ is a positive integer, $0 < \ushort{\alpha} \le \bar{\alpha} < \infty$, $\beta, c \in (0,1)$, and $\Delta \in (0,\infty)$.
\Input
$X \in \R_{\le r}^{m \times n}$ such that $\s(X; f, \R_{\le r}^{m \times n}) > 0$.
\Output
$Y \in \text{Algorithm~\ref{algo:P2GDRmapRealDeterminantalVariety}}(X; f, r, \ushort{\alpha}, \bar{\alpha}, \beta, c, \Delta)$.

\For
{$j \in \{0, \dots, \rank X - \rank_\Delta X\}$}
\State
Choose $\hat{X}^j \in \proj{\R_{\rank X - j}^{m \times n}}{X}$;
\State
Choose $\tilde{X}^j \in \hyperref[algo:P2GDmap]{\ppgd}(\hat{X}^j; \R^{m \times n}, \R_{\le r}^{m \times n}, f, \ushort{\alpha}, \bar{\alpha}, \beta, c)$;
\EndFor
\State
Return $Y \in \argmin_{\{\tilde{X}^j \mid j \in \{0, \dots, \rank X - \rank_\Delta X\}\}} f$.
\end{algorithmic}
\end{algorithm}

Proposition~\ref{prop:P2GDRmapRealDeterminantalVarietyPolak} states that this variant satisfies the same decrease guarantee as the original version.

\begin{proposition}
\label{prop:P2GDRmapRealDeterminantalVarietyPolak}
Proposition~\ref{prop:P2GDRmapPolak} holds for Algorithm~\ref{algo:P2GDRmapRealDeterminantalVariety}.
\end{proposition}

\begin{proof}
For all $X \in \R^{m \times n}$, since $\norm{X}_2 \le \norm{X}$, it holds that
\begin{equation*}
\{j \in \{0, \dots, \rank X\} \mid d_2(X, \R_j^{m \times n}) \le \Delta\} \supseteq \{j \in \{0, \dots, \rank X\} \mid \dist(X, \R_j^{m \times n}) \le \Delta\}
\end{equation*}
and thus
\begin{equation*}
\min\{j \in \{0, \dots, \rank X\} \mid d_2(X, \R_j^{m \times n}) \le \Delta\} \le \min\{j \in \{0, \dots, \rank X\} \mid \dist(X, \R_j^{m \times n}) \le \Delta\}.
\end{equation*}
Therefore, given $X \in \R_{\le r}^{m \times n}$ as input, Algorithm~\ref{algo:P2GDRmapRealDeterminantalVariety} explores the strata explored by Algorithm~\ref{algo:P2GDRmap}, and hence produces a point of cost not larger than the cost of the point produced by Algorithm~\ref{algo:P2GDRmap}.
\end{proof}

It is possible to prove Proposition~\ref{prop:P2GDRmapRealDeterminantalVarietyPolak} without relying on condition~3 of Assumption~\ref{assumption:Stratification} and Corollary~\ref{coro:ContinuityTangentConeImpliesContinuityStationarityMeasure}. Indeed, using the same notation as in the proof of Proposition~\ref{prop:P2GDRmapPolak}, if $\rank \ushort{X} = r$, then $\s(\cdot; f, \R_{\le r}^{m \times n})$ is continuous at $\ushort{X}$ since $\R_{\le r}^{m \times n}$ is identical to the smooth manifold $\R_r^{m \times n}$ around $\ushort{X}$, i.e., $\R_{\le r}^{m \times n} \cap \ball(\ushort{X}, \sigma_r(\ushort{X})) = \R_r^{m \times n} \cap \ball(\ushort{X}, \sigma_r(\ushort{X}))$, and, on $\R_r^{m \times n} \cap \ball(\ushort{X}, \sigma_r(\ushort{X}))$, $\s(\cdot; f, \R_{\le r}^{m \times n})$ therefore coincides with the norm of the Riemannian gradient of $f|_{\R_r^{m \times n}}$, which is continuous. If $\rank \ushort{X} < r$, then, in view of~\eqref{eq:NormProjectionTangentConeDeterminantalVariety} and the continuity of $\nabla f$, $\s(\cdot; f, \R_{\le r}^{m \times n})$ is bounded away from zero on the intersection of $\R_{< r}^{m \times n}$ and a sufficiently small ball centered at $\ushort{X}$. However, we chose to keep the proof of Proposition~\ref{prop:P2GDRmapRealDeterminantalVarietyPolak} as it stands to highlight that it follows from Proposition~\ref{prop:P2GDRmapPolak} in our abstract setting.

We can now state the main result of this section.

\begin{corollary}
\label{coro:P2GDRrealDeterminantalVarietyPolakConvergence}
Consider $\ppgdr$ on $\R_{\le r}^{m \times n}$, i.e., Algorithm~\ref{algo:P2GDR} using either Algorithm~\ref{algo:P2GDRmap} or Algorithm~\ref{algo:P2GDRmapRealDeterminantalVariety} in line~\ref{algo:P2GDR:P2GDRmap}. Then, Theorem~\ref{thm:P2GDRPolakConvergence} holds.
Let $(X_i)_{i \in \N}$ be a sequence produced by $\ppgdr$.
The sequence has at least one accumulation point if and only if $\liminf_{i \to \infty} \norm{X_i} < \infty$. For every convergent subsequence $(X_{i_k})_{k \in \N}$, $\lim_{k \to \infty} \s(X_{i_k}; f, \R_{\le r}^{m \times n}) = 0$.
If $(X_i)_{i \in \N}$ is bounded, which is the case notably if the sublevel set $\{X \in \R_{\le r}^{m \times n} \mid f(X) \le f(X_0)\}$ is bounded, then $\lim_{i \to \infty} \s(X_i; f, \R_{\le r}^{m \times n}) = 0$, and all accumulation points have the same image by $f$.
\end{corollary}

\begin{proof}
The first claim follows from Proposition~\ref{prop:P2GDRmapRealDeterminantalVarietyPolak}. The rest follows from Corollary~\ref{coro:P2GDRPolakConvergence} and Proposition~\ref{prop:RealDeterminantalVarietyNoSerendipitousPoint} since $\R_{\le r}^{m \times n}$ satisfies Assumption~\ref{assumption:Stratification}.
\end{proof}

\subsubsection{Practical implementation of \cite[Algorithm~1]{LevinKileelBoumal2022} on the determinantal variety}
\label{subsubsec:PracticalImplementationLKB22Algo1RealDeterminantalVariety}
In this section, we detail the implementation of \cite[Algorithm~1]{LevinKileelBoumal2022} with the rank factorization lift
\begin{equation*}
\varphi : \R^{m \times r} \times \R^{n \times r} \to \R^{m \times n} : (L, R) \mapsto LR^\tp
\end{equation*}
given in \cite[(1.1)]{LevinKileelBoumal2022} and the hook of \cite[Example 3.11]{LevinKileelBoumal2022}.

We consider the lifted cost function $g := f \circ \varphi$.
We assume that $\nabla f : \R^{m \times n} \to \R^{m \times n}$ is continuously differentiable and let $\nabla^2 f$ denote the Hessian of $f$, defined as the derivative of $\nabla f$, i.e., $\nabla^2 f : \R^{m \times n} \to \mathcal{L}(\R^{m \times n}) : X \mapsto (\nabla f)'(X)$, where $\mathcal{L}(\R^{m \times n})$ denotes the Banach space of all continuous linear operators on $\R^{m \times n}$ (see Appendix~\ref{sec:GradientHessianRealValuedFunctionOnHilbertSpace}).

\begin{proposition}
For all $(L, R), (\dot{L}, \dot{R}) \in \R^{m \times r} \times \R^{n \times r}$,
\begin{align*}
\nabla g(L, R)
=~& (\nabla f(LR^\tp) R, \nabla f(LR^\tp)^\tp L),\\
\nabla^2 g(L, R)(\dot{L}, \dot{R})
=~& (\nabla^2 f(LR^\tp)(\dot{L}R^\tp+L\dot{R}^\tp)R+\nabla f(LR^\tp)\dot{R},\\
&\hphantom{(}\nabla^2 f(LR^\tp)(\dot{L}R^\tp+L\dot{R}^\tp)^\tp L+\nabla f(LR^\tp)^\tp\dot{L}).
\end{align*}
\end{proposition}

\begin{proof}
We have
\begin{equation*}
\varphi'(L, R)(\dot{L}, \dot{R}) = \dot{L}R^\tp+L\dot{R}^\tp.
\end{equation*}
By the chain rule,
\begin{equation*}
g'(L, R) = f'(\varphi(L, R)) \circ \varphi'(L, R).
\end{equation*}
Thus,
\begin{align*}
g'(L, R)(\dot{L}, \dot{R})
&= f'(\varphi(L, R))(\varphi'(L, R)(\dot{L}, \dot{R}))\\
&= \ip{\nabla f(\varphi(L, R))}{\varphi'(L, R)(\dot{L}, \dot{R})}\\
&= \ip{\nabla f(\varphi(L, R))}{\dot{L}R^\tp} + \ip{\nabla f(\varphi(L, R))}{L\dot{R}^\tp}\\
&= \ip{\nabla f(\varphi(L, R))R}{\dot{L}} + \ip{\nabla f(\varphi(L, R))^\tp L}{\dot{R}}\\
&= \ip{(((\nabla f) \circ \varphi)(L, R)R, ((\nabla f) \circ \varphi)(L, R)^\tp L)}{(\dot{L}, \dot{R})}
\end{align*}
and we deduce that
\begin{equation*}
\nabla g(L, R) = (((\nabla f) \circ \varphi)(L, R)R, ((\nabla f) \circ \varphi)(L, R)^\tp L).
\end{equation*}
The formula for $\nabla^2 g(L, R)(\dot{L}, \dot{R})$ follows from the product rule and the chain rule.
\end{proof}

Given $(L, R) \in \R^{m \times r} \times \R^{n \times r}$, the eigenvalues of $\nabla^2 g(L, R)$, which is a self-adjoint linear operator on $\R^{m \times r} \times \R^{n \times r}$, are the eigenvalues of the matrix representing it in any basis of $\R^{m \times r} \times \R^{n \times r}$. Here, we use the orthonormal basis formed by the concatenation of the sequences $((\delta_{p,l})_{i,k=1}^{m,r}, 0_{n \times r})_{p,l=1}^{m,r}$ and $(0_{m \times r}, (\delta_{q,l})_{j,k=1}^{n,r})_{q,l=1}^{n,r}$. The matrix $H(L, R)$ representing $\nabla^2 g(L, R)$ in that basis can be formed as follows. For every $(i,k) \in \{1, \dots, m\} \times \{1, \dots, r\}$, the $(i-1)r+k$th column of $H(L, R)$ is the vector formed by concatenating the rows of $\nabla^2 f(LR^\tp)((\delta_{i,k})_{p,l=1}^{m,r}R^\tp)R$ and those of $\nabla^2 f(LR^\tp)((\delta_{i,k})_{p,l=1}^{m,r}R^\tp)^\tp L+\nabla f(LR^\tp)^\tp(\delta_{i,k})_{p,l=1}^{m,r}$. Then, for every $(j,k) \in \{1, \dots, n\} \times \{1, \dots, r\}$, the $mr+(j-1)r+k$th column of $H(L, R)$ is the vector formed by concatenating the rows of $\nabla^2 f(LR^\tp)(L{(\delta_{j,k})_{q,l=1}^{n,r}}^\tp)R+\nabla f(LR^\tp)(\delta_{j,k})_{q,l=1}^{n,r}$ and those of $\nabla^2 f(LR^\tp)(L{(\delta_{j,k})_{q,l=1}^{n,r}}^\tp)^\tp L$.

\subsubsection{Comparison of six optimization algorithms on the determinantal variety}
\label{subsubsec:ComparisonSixOptimizationAlgorithmsRealDeterminantalVariety}
In this section, we compare the six algorithms listed in Table~\ref{tab:IndexAlgorithmsRealDeterminantalVariety} based on the computational cost per iteration (Table~\ref{tab:ComparisonComputationalCostPerIterationAlgorithmsRealDeterminantalVariety}) and the convergence guarantees (Table~\ref{tab:ConvergenceAlgorithmsRealDeterminantalVariety}). As in Section~\ref{subsubsec:PracticalImplementationLKB22Algo1RealDeterminantalVariety}, we consider \cite[Algorithm~1]{LevinKileelBoumal2022} with the rank factorization lift and the hook of \cite[Example 3.11]{LevinKileelBoumal2022}.

\begin{table}[h]
\begin{center}
\begin{spacing}{1.2}
\begin{tabular}{*{3}{l}}
\hline
\emph{Algorithm} & \emph{Original paper} & \emph{Citations in this paper}\\
\hline
$\pgd$ & \cite[Algorithm~3.1]{JiaEtAl2022} & Section~\ref{sec:Introduction}\\
\hline
$\ppgd$ & \cite[Algorithm~3]{SchneiderUschmajew2015} & Sections~\ref{subsec:OverviewProposedAlgorithm} and \ref{subsec:P2GDmap}\\
\hline
$\rfd$ & \cite[Algorithm~4]{SchneiderUschmajew2015} & Sections~\ref{subsec:RFDRreview} and \ref{subsubsec:RFDmap}\\
\hline
$\ppgdr$ & Algorithm~\ref{algo:P2GDR} & Sections~\ref{subsec:OverviewProposedAlgorithm}, \ref{subsec:P2GDRmap}, and \ref{subsec:P2GDR}\\
\hline
$\rfdr$ & \cite[Algorithm~3]{OlikierAbsil2022RFDR} & Sections~\ref{subsec:RFDRreview}, \ref{subsubsec:RFDRmap}, and \ref{subsubsec:RFDR}\\
\hline
\cite[Algorithm~1]{LevinKileelBoumal2022} & \cite[\S 3]{LevinKileelBoumal2022} & Sections~\ref{subsec:OptimizationThroughSmoothLift} and \ref{subsubsec:PracticalImplementationLKB22Algo1RealDeterminantalVariety}\\
\hline
\end{tabular}
\vspace*{-8mm}
\end{spacing}
\end{center}
\caption{Index of algorithms aiming to solve problem~\eqref{eq:OptiProblem} with $C = \R_{\le r}^{m \times n}$.}
\label{tab:IndexAlgorithmsRealDeterminantalVariety}
\end{table}

The respective computational costs per iteration of $\ppgd$, $\rfd$, $\ppgdr$, and $\rfdr$ are compared in \cite[\S 7]{OlikierAbsil2022RFDR} based on detailed implementations of these algorithms involving only evaluations of $f$ and $\nabla f$ and some operations from linear algebra:
\begin{enumerate}
\item matrix multiplication;
\item thin QR factorization with column pivoting (see, e.g., \cite[Algorithm~5.4.1]{GolubVanLoan});
\item \emph{small scale} (truncated) SVD, i.e., the smallest dimension of the matrix to decompose is at most $2r$;
\item \emph{large scale} truncated SVD, i.e., truncated SVD that is not small scale.
\end{enumerate}
In this list, only the (truncated) SVD cannot be executed within a finite number of arithmetic operations. Before including $\pgd$ and \cite[Algorithm~1]{LevinKileelBoumal2022} in the comparison, we recall available upper bounds on the number of iterations in the backtracking procedures of $\ppgd$, $\rfd$, $\ppgdr$, and $\rfdr$.
By \cite[(17)]{OlikierAbsil2022RFDR}, given $X \in \R_{\le r}^{m \times n}$ as input and using $\alpha \in (0, \infty)$ as initial step size for the backtracking procedure, the $\ppgd$ map evaluates $f$ from $1$ to
\begin{equation}
\label{eq:MaxNumEvaluationCostP2GDmapRealDeterminantalVariety}
1+\max\left\{0, \left\lceil\ln\left(\frac{1-c}{\alpha\kappa_{\ball[X, \alpha\s(X; f, \R_{\le r}^{m \times n})]}(X; f, \alpha)}\right)/\ln\beta\right\rceil\right\}
\end{equation}
times, where the second term is the maximum number of iterations in the backtracking loop.
By \cite[(18)]{OlikierAbsil2022RFDR}, given $X \in \R_{\le r}^{m \times n}$ as input and using $\alpha \in (0, \infty)$ as initial step size for the backtracking procedure, the $\rfd$ map evaluates $f$ from $1$ to
\begin{equation}
\label{eq:MaxNumEvaluationCostRFDmapRealDeterminantalVariety}
1+\max\left\{0, \left\lceil\ln\left(\frac{2(1-c)}{\alpha\lip_{\ball[X, \alpha\s(X; f, \R_{\le r}^{m \times n})]}(\nabla f)}\right)/\ln\beta\right\rceil\right\}
\end{equation}
times, where the second term is the maximum number of iterations in the backtracking loop.

We now analyze the computational cost per iteration of $\pgd$ and \cite[Algorithm~1]{LevinKileelBoumal2022}.
The paper \cite{JiaEtAl2022} gives no explicit bound on the number of inner iterations performed by $\pgd$. Nevertheless, every (outer) iteration of $\pgd$ requires projecting $X - \alpha \nabla f(X)$ onto $\R_{\le r}^{m \times n}$, $X \in \R_{\le r}^{m \times n}$ and $\alpha \in (0, \infty)$ being respectively the current iterate and the step size, which, in general, involves the computation of a large scale truncated SVD.

From the analysis conducted in Section~\ref{subsubsec:PracticalImplementationLKB22Algo1RealDeterminantalVariety}, we know that every iteration of \cite[Algorithm~1]{LevinKileelBoumal2022} requires the computation of a smallest eigenvalue of the order-$(m+n)r$ matrix representing $\nabla^2 g(L, R)$ in a given basis of $\R^{m \times r} \times \R^{n \times r}$, $(L, R)$ being the current iterate, which involves the computation of a large scale truncated SVD. Moreover, each iteration updating the current iterate requires a hook. The hook of \cite[Example 3.11]{LevinKileelBoumal2022} involves the computation of a thin QR factorization of $LR^\tp$ and of an SVD of the R factor, which is a small scale SVD.

\begin{table}[h]
\begin{center}
\begin{spacing}{1.2}
\begin{tabular}{l*{6}{c}}
\hline
\emph{Algorithm} & $f$ & $\nabla f$ & $\nabla^2 f$ & QR & small SVD & large SVD \\
\hline
$\pgd$ & $1+i_*$ & $1$ & $0$ & $0$ & $0$ & $1+i_*$\\
\hline
\multirow{2}{*}{$\ppgd$} & $1$ & \multirow{2}{*}{$1$} & \multirow{2}{*}{$0$} & $2$ & $1$ & $0$\\
\cline{2-2}\cline{5-7}
& \eqref{eq:MaxNumEvaluationCostP2GDmapRealDeterminantalVariety} & & & $4$ & \eqref{eq:MaxNumEvaluationCostP2GDmapRealDeterminantalVariety} & $1$\\
\hline
\multirow{2}{*}{$\rfd$} & $1$ & \multirow{2}{*}{$1$} & \multirow{2}{*}{$0$} & \multirow{2}{*}{$1$} & \multirow{2}{*}{$0$} & $0$\\
\cline{2-2}\cline{7-7}
& \eqref{eq:MaxNumEvaluationCostRFDmapRealDeterminantalVariety} & & & & & $1$\\
\hline
\multirow{2}{*}{$\ppgdr$} & $1$ & $1$ & \multirow{2}{*}{$0$} & $2$ & $1$ & $0$\\
\cline{2-3}\cline{5-7}
& $\eqref{eq:MaxNumEvaluationCostRFDmapRealDeterminantalVariety} + r \cdot \eqref{eq:MaxNumEvaluationCostP2GDmapRealDeterminantalVariety}$ & $r+1$ & & $4r-2$ & $r \cdot \eqref{eq:MaxNumEvaluationCostP2GDmapRealDeterminantalVariety}$ & $r$ \\
\hline
\multirow{2}{*}{$\rfdr$} & $1$ & $1$ & \multirow{2}{*}{$0$} & \multirow{2}{*}{$0$} & $1$ & $0$\\
\cline{2-3}\cline{6-7}
& $2 \cdot \eqref{eq:MaxNumEvaluationCostRFDmapRealDeterminantalVariety}$ & $2$ & & & $2$ & $1$\\
\hline
\multirow{2}{*}{\cite[Algorithm~1]{LevinKileelBoumal2022}} & \multirow{2}{*}{$1$} & \multirow{2}{*}{$1$} & \multirow{2}{*}{$1$} & $0$ & $0$ & \multirow{2}{*}{$1$}\\
\cline{5-6}
& & & & $1$ & $1$ & \\
\hline
\end{tabular}
\vspace*{-8mm}
\end{spacing}
\end{center}
\caption{Operations required by six algorithms aiming to solve problem~\eqref{eq:OptiProblem} with $C = \R_{\le r}^{m \times n}$ to perform one iteration, matrix multiplication excluded. The fields ``$f$'', ``$\nabla f$'', ``$\nabla^2 f$'', ``QR'', ``small SVD'', and ``large SVD'' respectively correspond to ``evaluation of $f$'', ``evaluation of $\nabla f$'', ``evaluation of $\nabla^2 f$'', ``QR factorization with column pivoting'', ``small scale (truncated) SVD'', and ``large scale truncated SVD''. When two subrows appear in a row, the upper entry corresponds the the best case and the lower one to the worst case. In the ``$\pgd$'' line, $i_*$ denotes the number of inner iteration(s) performed. An iteration of \cite[Algorithm~1]{LevinKileelBoumal2022} performs no QR and no small SVD if and only if it does not change the iterate, i.e., the algorithm does not progress.}
\label{tab:ComparisonComputationalCostPerIterationAlgorithmsRealDeterminantalVariety}
\end{table}

The convergence guarantees offered by the six algorithms are compared in Table~\ref{tab:ConvergenceAlgorithmsRealDeterminantalVariety}. We make three remarks.

\begin{table}[h]
\begin{center}
\begin{spacing}{1.2}
\begin{tabular}{l*{4}{c}}
\hline
\emph{Algorithm} & \emph{Order} & $\s(\cdot; f, \R_{\le r}^{m \times n}) \to 0$ & Mordukhovich & Stationary \\
\hline
$\pgd$ & $1$ & ? & \ding{51} & ?\\
\hline
$\ppgd$ & $1$ & ? & ? & \ding{55}\\
\hline
$\rfd$ & $1$ & \ding{51} & ? & \ding{55}\\
\hline
$\ppgdr$ & $1$ & \ding{51} & \ding{51} & \ding{51}\\
\hline
$\rfdr$ & $1$ & \ding{51} & \ding{51} & \ding{51}\\
\hline
\cite[Algorithm~1]{LevinKileelBoumal2022} & $2$ & \ding{51} & \ding{51} & \ding{51}\\
\hline
\end{tabular}
\vspace*{-8mm}
\end{spacing}
\end{center}
\caption{Comparison of the convergence guarantees of six algorithms aiming to solve problem~\eqref{eq:OptiProblem} with $C = \R_{\le r}^{m \times n}$. The property ``$\s(\cdot; f, \R_{\le r}^{m \times n}) \to 0$'' means that $\s(\cdot; f, \R_{\le r}^{m \times n})$ goes to zero along every convergent subsequence of the generated sequence. The properties ``Mordukhovich'' and ``Stationary'' mean that the algorithm accumulates at Mordukhovich stationary points and stationary points, respectively. The symbols ``\ding{51}'', ``\ding{55}'', and ``?'' respectively mean ``yes'', ``no'', and ``open question''.}
\label{tab:ConvergenceAlgorithmsRealDeterminantalVariety}
\end{table}

First, we recall from Proposition~\ref{prop:NormalConesRealDeterminantalVariety} that Mordukhovich stationarity is weaker than stationarity at every $X \in \R_{\ushort{r}}^{m \times n}$ with $\ushort{r} \in \{0, \dots, r-1\}$. Indeed, while the latter amounts to $\nabla f(X) = 0_{m \times n}$, the former only amounts to the stationarity of $X$ on $\R_{\le \ushort{r}}^{m \times n}$ together with the inequality $\rank \nabla f(X) \le \min\{m,n\}-r$.

Second, we recall from Section~\ref{subsec:RFDRreview} that the property ``$\s(\cdot; f, \R_{\le r}^{m \times n}) \to 0$'' holds for $\rfd$ if $f$ is real-analytic and bounded from below.
By Proposition~\ref{prop:ConvergenceStationarityMeasureZeroUpperStratumMordukhovichStationary}, if, moreover, the sequence generated by $\rfd$ is contained in $\R_r^{m \times n}$, then the property ``Mordukhovich'' also holds.

Third, Table~\ref{tab:ConvergenceAlgorithmsRealDeterminantalVariety} does not give any information on the performance of the algorithms after a finite number of iterations, which depends on the stopping criterion. In the rest of this section, we discuss this for the three apocalypse-free algorithms, namely $\ppgdr$, $\rfdr$, and \cite[Algorithm~1]{LevinKileelBoumal2022}.

Since $\R_{\le r}^{m \times n}$ has no serendipitous point (Proposition~\ref{prop:RealDeterminantalVarietyNoSerendipitousPoint}), if $(X_i)_{i \in \N}$ is a sequence generated by any of the three algorithms and does not diverge to infinity, then, for every $\varepsilon \in (0, \infty)$, the set $\{i \in \N \mid \s(X_i; f, \R_{\le r}^{m \times n}) \le \varepsilon\}$ is nonempty and thus the stopping criterion defined by \eqref{eq:StoppingCriterion} is eventually satisfied: stop the algorithm at iteration
\begin{equation*}
i_\varepsilon := \min\{i \in \N \mid \s(X_i; f, \R_{\le r}^{m \times n}) \le \varepsilon\}.
\end{equation*}
Two facts should be pointed out, however.
First, no a priori upper bound on $i_\varepsilon$ is known, as explained in the next paragraph.
Second, since the set of apocalyptic points of $\R_{\le r}^{m \times n}$ is $\R_{< r}^{m \times n}$, $X_{i_\varepsilon}$ can be close either to an apocalyptic point or a stationary point, as explained after Proposition~\ref{prop:StationaryPointApocalypticSerendipitous}, and none of the two cases can be excluded a priori, which makes the choice of the parameter $\Delta \in (0, \infty)$ of $\ppgdr$ significant, as illustrated in Sections~\ref{subsubsec:LKB22instance} and \ref{subsubsec:P2GDfollowingApocalypseSize2*2}.
 
No a priori upper bound on $i_\varepsilon$ is known for $\ppgdr$ and $\rfdr$. In contrast, \cite[Algorithm~1]{LevinKileelBoumal2022} enjoys the guarantees given in \cite[Theorems~3.4 and 3.16]{LevinKileelBoumal2022}. However, those guarantees do not yield the upper bound sought, as explained next.
Given $\varepsilon_1, \varepsilon_2 \in (0, \infty)$, \cite[Theorem~3.4]{LevinKileelBoumal2022} provides, under reasonable conditions, an upper bound on the number of iterations required by \cite[Algorithm~1]{LevinKileelBoumal2022} to produce a point $(L, R) \in \R^{m \times r} \times \R^{n \times r}$ that is $(\varepsilon_1, \varepsilon_2)$-approximate second-order stationary for $\min_{\R^{m \times r} \times \R^{n \times r}} g$, i.e., $\norm{\nabla g(L, R)}_{\R^{m \times r} \times \R^{n \times r}} \le \varepsilon_1$ and $\lambda_{\min}(\nabla^2 g(L, R)) \ge -\varepsilon_2$. Moreover, by \cite[Theorem~3.16]{LevinKileelBoumal2022}, if such a point $(L, R)$ is obtained by \cite[Algorithm~1]{LevinKileelBoumal2022} and $X = LR^\tp$, then
\begin{equation}
\label{eq:LKB22(3.16)}
\s(X; f, \R_{\le r}^{m \times n}) \le \min\left\{\sqrt{\frac{2}{\sigma_r(X)}}\varepsilon_1, \sqrt{\rank \nabla f(X)}\bigg(\varepsilon_2+2\lip_{\R^{m \times n}}(\nabla f)\sigma_r(X)\bigg)\right\}.
\end{equation}
Unfortunately, this does not give us an upper bound on $i_\varepsilon$ since $X_{i_\varepsilon}$ is unknown. Indeed, based on~\eqref{eq:LKB22(3.16)}, if nothing is known on $\sigma_r(X)$, in particular whether it is zero or not, then the only way to ensure that $\s(X; f, \R_{\le r}^{m \times n}) \le \varepsilon$ is to take $\varepsilon_1 \le \varepsilon \sqrt{\sigma_r(X)/2}$ and $\varepsilon_2 \le \varepsilon/\sqrt{\min\{m,n\}}$. The upper bound on $\varepsilon_2$ is exploitable but not the one on $\varepsilon_1$ since $\sigma_r(X)$ is unknown, even in a bounded sublevel set.

\subsubsection{Numerical comparison of four algorithms on the instance of \cite[\S 2.2]{LevinKileelBoumal2022}}
\label{subsubsec:LKB22instance}
In this section, we compare numerically $\pgd$, $\ppgd$, $\ppgdr$ (Algorithm~\ref{algo:P2GDR} using Algorithm~\ref{algo:P2GDRmapRealDeterminantalVariety} in line~\ref{algo:P2GDR:P2GDRmap}), and \cite[Algorithm~1]{LevinKileelBoumal2022} on the problem presented in \cite[\S 2.2]{LevinKileelBoumal2022}. We recall from Section~\ref{subsec:RFDRreview} that, on this problem, $\rfd$ produces the same sequence as $\ppgd$.

In \cite[\S 2.2]{LevinKileelBoumal2022}, the following instance of~\eqref{eq:OptiProblem} with $C = \R_{\le 2}^{3 \times 3}$ is considered: minimizing
\begin{equation*}
f : \R^{3 \times 3} \to \R : X \mapsto Q(X(1\mathord{:}2, 1\mathord{:}2)) + \phi(X(3, 3))
\end{equation*}
on $\R_{\le 2}^{3 \times 3}$, where $X(1\mathord{:}2, 1\mathord{:}2)$ is the upper-left $2 \times 2$ submatrix of $X$, $X(3, 3)$ its bottom-right entry, $\phi : \R \to \R : x \mapsto \frac{x^4}{4}-\frac{(x+1)^2}{2}$, $Q : \R^{2 \times 2} \to \R : Y \mapsto \frac{1}{2}\norm{D(Y-Y^*)}^2$, $D := \diag(1,\frac{1}{2})$, and $Y^* := \diag(1,0)$.
First, it is observed that $\argmin \phi = x_0 \approx 1.32471795724475$,
\begin{align*}
\argmin f &= \left\{\begin{bmatrix}
1 & 0 & a_1 \\ 0 & 0 & a_2 \\ a_3 & a_4 & x_0
\end{bmatrix} \mid a_1, a_2, a_3, a_4 \in \R\right\},\\
\argmin_{\R_{\le 2}^{3 \times 3}} f &= \left\{\begin{bmatrix}
1 & 0 & a_1 \\ 0 & 0 & a_2 \\ a_3 & a_4 & x_0
\end{bmatrix} \mid a_1, a_2, a_3, a_4 \in \R, a_2a_4 = 0\right\},
\end{align*}
and $f^* := \min f = \min_{\R_{\le 2}^{3 \times 3}} f = \phi(x_0) \approx -1.932257884495233$.
Second, it is proven analytically that $\ppgd$ follows the apocalypse $(\diag(1, 0, 0), (\diag(1+(-\frac{3}{5})^i, (\frac{3}{5})^i, 0))_{i \in \N}, f)$ if used on this problem with $X_0 := \diag(2,1,0)$, $\ushort{\alpha} := \bar{\alpha} := \frac{8}{5}$, any $\beta \in (0,1)$, and $c := \frac{1}{5}$.

In this section, we compare numerically $\pgd$, $\ppgd$, $\ppgdr$, and \cite[Algorithm~1]{LevinKileelBoumal2022} on that instance, based on our Matlab implementations of these algorithms.\footnote{Those implementations are available at \url{https://github.com/golikier/ApocalypseFreeLowRankOptimization/blob/main/README.md}.} 
We use the parameters described in Table~\ref{tab:AlgorithmsParameters}. Furthermore, we use \cite[Algorithm~1]{LevinKileelBoumal2022} with the rank factorization lift, i.e., $\varphi : \R^{m \times r} \times \R^{n \times r} \to \R^{m \times n} : (L, R) \mapsto LR^\tp$, the hook of \cite[Example~3.11]{LevinKileelBoumal2022}, and the Cauchy point at each iteration. Thus, we apply \cite[Algorithm~1]{LevinKileelBoumal2022} to $g := f \circ \varphi$.
For $\pgd$, $\ppgd$, and $\ppgdr$, we use the stopping criterion defined by \eqref{eq:StoppingCriterion}; specifically, we stop them as soon as the stationarity measure $\s(\cdot;f, \R_{\le 2}^{3 \times 3})$ becomes smaller than or equal to $3 \cdot 10^{-9}$. In contrast, we let \cite[Algorithm~1]{LevinKileelBoumal2022} run for $5.5\cdot10^5$ iterations. We obtain the plots in Figure~\ref{subfig:PGDvsP2GDvsP2GDRvsLKB22Algo1_LKB22} where we observe that \cite[Algorithm~1]{LevinKileelBoumal2022} stagnates after less than $5.5\cdot10^5$ iterations.
\begin{table}[h]
\begin{center}
\begin{spacing}{1.8}
\begin{tabular}{ll}
\hline
$\pgd$ \cite[Algorithm~3.1]{JiaEtAl2022} & $X_0 := \diag(2, 1, 0)$, $\gamma_{\min} := \gamma_{\max} := \frac{5}{8}$, $\tau := 2$, $\sigma := \frac{1}{5}$, $m := 0$ \\[2mm]
\hline
$\ppgd$ and $\ppgdr$ & $X_0 := \diag(2,1,0)$, $\ushort{\alpha} := \bar{\alpha} := \frac{8}{5}$, $\beta := \frac{1}{2}$, $c := \frac{1}{5}$, $\Delta := \frac{1}{10}$\\[2mm]
\hline
\cite[Algorithm~1]{LevinKileelBoumal2022} & $L_0 := R_0 := \left(\begin{smallmatrix} \sqrt{2} & 0 \\ 0 & 1 \\ 0 & 0 \end{smallmatrix}\right)$, $\ushort{\gamma} := \oshort{\gamma} := 1$, $\gamma_\text{c} := \frac{1}{2}$, $\eta := \frac{1}{10}$\\[2mm]
\hline
\end{tabular}
\vspace*{-8mm}
\end{spacing}
\end{center}
\caption{Parameters of the algorithms for the numerical comparison of Section~\ref{subsubsec:LKB22instance}.}
\label{tab:AlgorithmsParameters}
\end{table}

We observe that $\ppgd$ behaves as predicted in \cite[\S 2.2]{LevinKileelBoumal2022} and is the only algorithm among the four to follow an apocalypse. In particular, $\ppgdr$ behaves in agreement with Theorem~\ref{thm:P2GDRPolakConvergence} and Corollary~\ref{coro:P2GDRPolakConvergence}, and \cite[Algorithm~1]{LevinKileelBoumal2022} in agreement with \cite[Theorem~1.1]{LevinKileelBoumal2022}; these results hold since all sublevel sets of $f$ are bounded.
We also observe that \cite[Algorithm~1]{LevinKileelBoumal2022} converges much more slowly than $\pgd$ and $\ppgdr$. The slow convergence of \cite[Algorithm~1]{LevinKileelBoumal2022} is in agreement with \cite[Theorem~3.4]{LevinKileelBoumal2022}. Indeed, for the parameters of Table~\ref{tab:AlgorithmsParameters}, $k(\varepsilon_1, \varepsilon_2) := \lfloor 160\max\{\varepsilon_1^{-2}, \varepsilon_2^{-3}\}\rfloor$ is a lower bound on the upper bound on the number of iterations required by \cite[Algorithm~1]{LevinKileelBoumal2022} to produce an $(\varepsilon_1, \varepsilon_2)$-approximate second-order stationary point given in \cite[(3.11)]{LevinKileelBoumal2022}, and we observe in Figure~\ref{subfig:LKB22Algo1_LKB22_BoundConvergence} that every iteration number $i$ satisfies $i \le k_i := k(\norm{\nabla g(L_i, R_i)}, -\lambda_i)$, where $\lambda_i$ denotes a smallest eigenvalue of $\nabla^2 g(L_i, R_i)$, which implies that the upper bound given in \cite[(3.11)]{LevinKileelBoumal2022} is respected. Figure~\ref{subfig:LKB22Algo1_LKB22_BoundConvergence} also shows that the upper bound given in \cite[(3.11)]{LevinKileelBoumal2022} is very pessimistic in this experiment.

The only iteration of $\ppgdr$ that differs from a $\ppgd$ iteration is the fifth one, where $\rank_\Delta X_5 = 1$, $\hat{X}_5^1 = \diag(1, 0, 0)$, and $X_6$ is selected in $\ppgd(\hat{X}_5^1; \R^{3 \times 3}, \R_{\le 2}^{3 \times 3}, f, \ushort{\alpha}, \bar{\alpha}, \beta, c)$. This unique intervention of the rank reduction mechanism prevents $\ppgdr$ from following the apocalypse.

\begin{figure}
\begin{center}
\subfloat[{Evolutions of the stationarity measure, the cost function, and the distance to a global minimizer along the sequences produced by $\pgd$, $\ppgd$, $\ppgdr$, and \cite[Algorithm~1]{LevinKileelBoumal2022}.}]{
\includegraphics[scale=0.75]{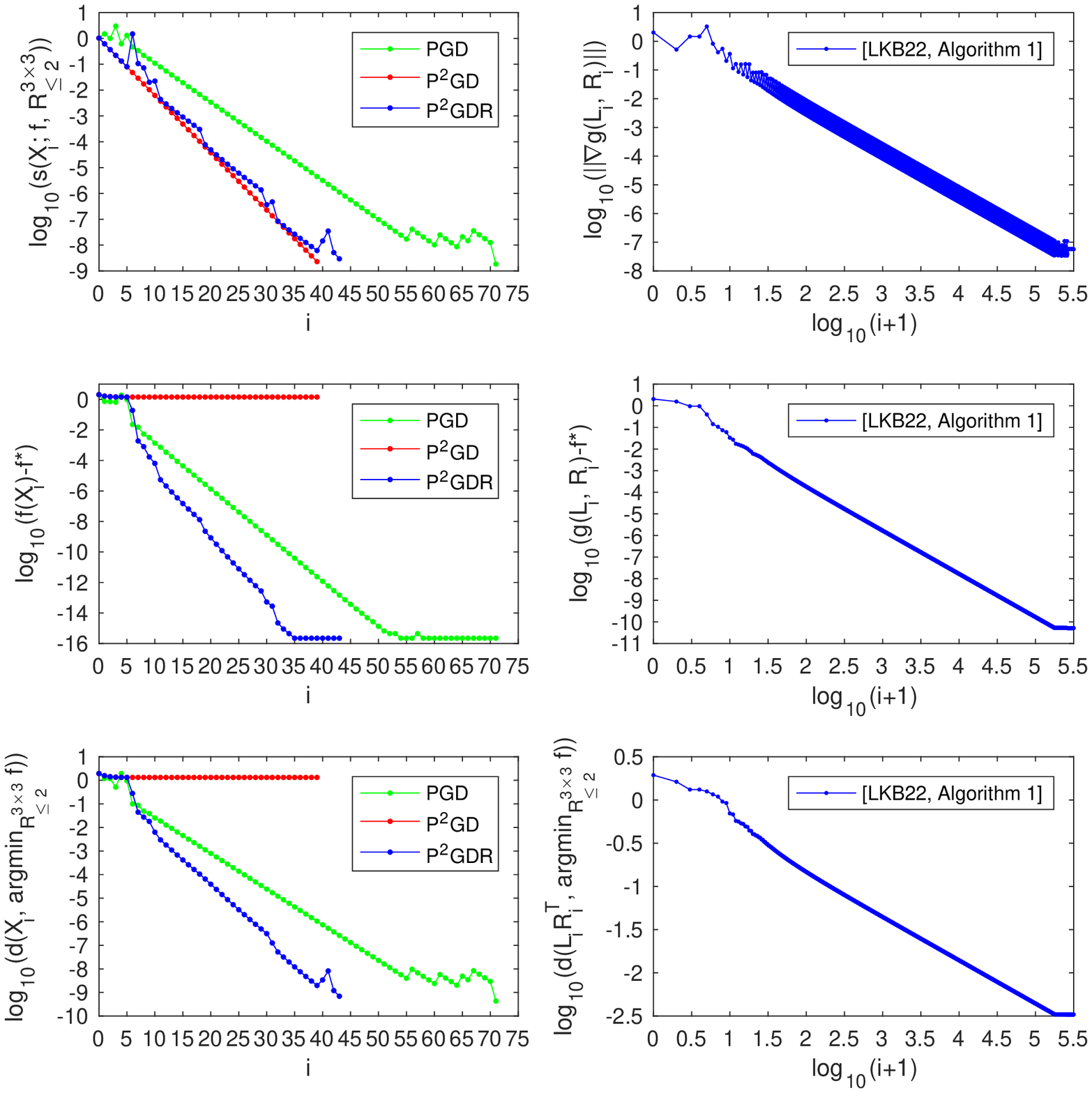}
\label{subfig:PGDvsP2GDvsP2GDRvsLKB22Algo1_LKB22}}

\subfloat[{Evolutions of the smallest eigenvalue of the Hessian and the lower bound on the right-hand side of \cite[(3.11)]{LevinKileelBoumal2022} along the sequence produced by \cite[Algorithm~1]{LevinKileelBoumal2022}.}]{
\includegraphics[scale=0.75]{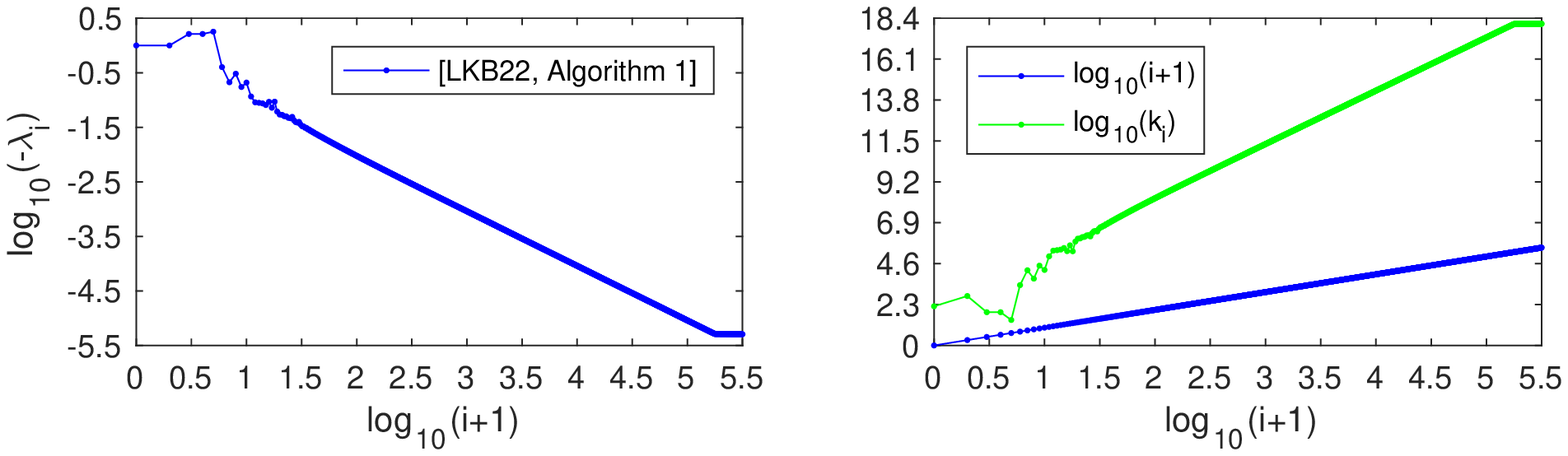}
\label{subfig:LKB22Algo1_LKB22_BoundConvergence}}
\end{center}
\caption{Numerical comparison of $\pgd$, $\ppgd$, $\ppgdr$, and \cite[Algorithm~1]{LevinKileelBoumal2022} on the problem of Section~\ref{subsubsec:LKB22instance} with the parameters of Table~\ref{tab:AlgorithmsParameters}.}
\label{fig:PGDvsP2GDvsP2GDRvsLKB22Algo1_LKB22}
\end{figure}

If the behavior of $\ppgdr$ on this example seems satisfying, it should however be noted that, if $\Delta < (\frac{3}{5})^{38}$, then $\ppgdr$ produces the exact same (finite) sequence of iterates as $\ppgd$ because $\rank_\Delta X_{38} = 2$ and $\s(X_{39}; f, \R_{\le 2}^{3 \times 3}) \le 3\cdot10^{-9}$. This shows that, in a practical implementation of $\ppgdr$ where the stopping criterion defined by \eqref{eq:StoppingCriterion} is used, i.e., the algorithm is stopped as soon as the stationarity measure becomes smaller than or equal to some threshold $\varepsilon \in (0, \infty)$, it is important to choose $\Delta$ in such a way that the algorithm does not stop while it is heading towards an apocalyptic point, which is $\diag(1, 0, 0)$ in this case, in the sense that, if we had continued with $\varepsilon := \Delta := 0$, an apocalypse would have occurred.

\subsubsection{$\ppgd$ following an apocalypse on $\R_{\le 1}^{2 \times 2}$}
\label{subsubsec:P2GDfollowingApocalypseSize2*2}
In this section, we present an example of $\ppgd$ following an apocalypse on $\R_{\le 1}^{2 \times 2}$.
For the function
\begin{equation*}
f : \R^{2 \times 2} \to \R : X \mapsto \frac{X(1, 1)^2+(X(2, 2)-1)^2+(X(1, 2)-X(2, 1))^2}{2},
\end{equation*}
we have $\min_{\R_{\le 1}^{2 \times 2}} f = 0$ and $\argmin_{\R_{\le 1}^{2 \times 2}} f = \diag(0,1)$.
Proposition~\ref{prop:ExampleApocalypseSmallestSize} states that $\ppgd$ used with an initial step size for the backtracking procedure smaller than $1$ can follow an apocalypse by trying to minimize $f$ on $\R_{\le 1}^{2 \times 2}$. Before introducing that proposition, we give an intuitive explanation of the result. Given any point $\diag(x_0,0)$ with $x_0 \in (0,\infty)$, $\ppgd$ produces a sequence converging to $0_{2 \times 2}$, thereby minimizing the first term of $f$. However, no iteration affects the second term because the search direction $\diag(0,1)$, which would enable the minimization of the second term, is not available until $0_{2 \times 2}$ is reached, which never happens. The third term of $f$ makes its global minimizer on $\R_{\le 1}^{2 \times 2}$ unique without affecting the iterations.

\begin{proposition}
\label{prop:ExampleApocalypseSmallestSize}
Let $x_0 \in (0,\infty)$ and $\alpha \in (0,1)$.
With $f$ on $\R_{\le 1}^{2 \times 2}$ as defined above, starting from $X_0 := \diag(x_0,0)$, and using $\ushort{\alpha} := \bar{\alpha} := \alpha$, $\beta \in (0,1)$, and $c \in (0,\frac{1}{2}]$, $\ppgd$ produces the sequence $(X_i)_{i \in \N}$ defined by
\begin{equation}
\label{eq:ExampleApocalypseSmallestSize}
X_i := \diag((1-\alpha)^ix_0,0)
\end{equation}
for all $i \in \N$. Moreover, $\s(X_i; f, \R_{\le 1}^{2 \times 2}) = (1-\alpha)^ix_0$ for all $i \in \N$. In particular, since $\s(0_{2 \times 2}; f, \R_{\le 1}^{2 \times 2}) = \norm{\nabla f(0_{2 \times 2})} = 1$, $(0_{2 \times 2}, (X_i)_{i \in \N}, f)$ is an apocalypse.
\end{proposition}

\begin{proof}
The formula~\eqref{eq:ExampleApocalypseSmallestSize} holds for $i = 0$. 
Furthermore, for all $X \in \R^{2 \times 2}$,
\begin{equation*}
\nabla f(X) = X - \begin{bmatrix} 0 & X(2, 1) \\ X(1, 2) & 1 \end{bmatrix}.
\end{equation*}
Therefore, for every $i \in \N$,
\begin{align*}
-\nabla f(X_i) = \diag(-(1-\alpha)^ix_0,1),&&
\proj{\tancone{\R_{\le 1}^{2 \times 2}}{X_i}}{-\nabla f(X_i)} = \diag(-(1-\alpha)^ix_0,0),
\end{align*}
and the formula for $\s(X_i; f, \R_{\le 1}^{2 \times 2})$ is valid.
Thus, for every $i \in \N$,
\begin{align*}
X_{i+1} &= X_i + \alpha \proj{\tancone{\R_{\le 1}^{2 \times 2}}{X_i}}{-\nabla f(X_i)},\\
f(X_{i+1}) &\le f(X_i) - c \, \alpha \, \s(X_i; f, \R_{\le 1}^{2 \times 2})^2,
\end{align*}
which shows that the sequence defined by~\eqref{eq:ExampleApocalypseSmallestSize} is indeed the one produced by $\ppgd$.
The expression for $\s(0_{2 \times 2}; f, \R_{\le 1}^{2 \times 2})$ follows from the fact that $-\nabla f(0_{2 \times 2}) = \diag(0,1) \in \R_{\le 1}^{2 \times 2} = \tancone{\R_{\le 1}^{2 \times 2}}{0_{2 \times 2}}$.
\end{proof}

Proposition~\ref{prop:ExampleApocalypseSmallestSizeP2GDR} shows that $\ppgdr$ (Algorithm~\ref{algo:P2GDR} using Algorithm~\ref{algo:P2GDRmapRealDeterminantalVariety} in line~\ref{algo:P2GDR:P2GDRmap}) escapes the apocalypse due to its rank reduction mechanism. During the first iterations, $\ppgdr$ produces the same iterates as $\ppgd$. However, when the numerical rank of the iterate becomes smaller than its rank, i.e., when its smallest singular value becomes smaller than or equal to $\Delta$, $\ppgdr$ realizes that a stronger decrease of $f$ is obtained by first reducing the rank and then applying an iteration of $\ppgd$. As a result, the first term of $f$ is minimized within a finite number of iterations, after which the minimization of the second term can start.

\begin{proposition}
\label{prop:ExampleApocalypseSmallestSizeP2GDR}
Consider the same problem as in Proposition~\ref{prop:ExampleApocalypseSmallestSize} with the same parameters and $\Delta \in (0,\infty)$.
Then, $\ppgdr$ produces the sequence $(X_i)_{i \in \N}$ defined by
\begin{equation}
\label{eq:ExampleApocalypseSmallestSizeP2GDR}
X_i := \left\{\begin{array}{ll}
\diag((1-\alpha)^ix_0,0) & \text{if } i \le i_\Delta\\
\diag(0,1-(1-\alpha)^{i-i_\Delta}) & \text{if } i > i_\Delta\\
\end{array}\right.
\end{equation}
where $i_\Delta := \max\bigg\{\bigg\lceil\frac{\ln(\frac{\Delta}{x_0})}{\ln(1-\alpha)}\bigg\rceil,0\bigg\}$.
In particular, $(X_i)_{i \in \N}$ converges to $\diag(0,1)$ and $\lim_{i \to \infty} \s(X_i; f, \R_{\le 1}^{2 \times 2}) = 0$.
\end{proposition}

\begin{proof}
The formula~\eqref{eq:ExampleApocalypseSmallestSizeP2GDR} is correct for $i = 0$.
If $i_\Delta > 0$, then $(1-\alpha)^ix_0 > \Delta$ for every $i \in \{0, \dots, i_\Delta-1\}$, and \eqref{eq:ExampleApocalypseSmallestSizeP2GDR} thus holds for every $i \in \{1, \dots, i_\Delta\}$ in view of Proposition~\ref{prop:ExampleApocalypseSmallestSize}.
It remains to prove~\eqref{eq:ExampleApocalypseSmallestSizeP2GDR} for every integer $i > i_\Delta$.
Let us look at iteration $i_\Delta$. Since $\hat{X}_{i_\Delta}^1 = 0_{2 \times 2}$, $-\nabla f(0_{2 \times 2}) = \diag(0,1) \in \R_{\le 1}^{2 \times 2} = \tancone{\R_{\le 1}^{2 \times 2}}{0_{2 \times 2}}$, $\s(0_{2 \times 2}; f, \R_{\le 1}^{2 \times 2}) = 1$, $\hat{X}_{i_\Delta}^1 - \alpha \nabla f(\hat{X}_{i_\Delta}^1) = \diag(0,\alpha)$, and
\begin{equation*}
f(\hat{X}_{i_\Delta}^1)-f(\diag(0,\alpha)) \ge c \, \alpha \, \s(\hat{X}_{i_\Delta}^1; f, \R_{\le 1}^{2 \times 2})^2,
\end{equation*}
we have $\tilde{X}_{i_\Delta}^1 = \diag(0,\alpha)$.
As $\hat{X}_{i_\Delta}^0 = X_{i_\Delta}$, Proposition~\ref{prop:ExampleApocalypseSmallestSize} yields $\tilde{X}_{i_\Delta}^0 = \diag((1-\alpha)^{i_\Delta+1}x_0,0)$. Since
\begin{equation*}
f(\tilde{X}_{i_\Delta}^1)
= \frac{(1-\alpha)^2}{2}
< \frac{(1-\alpha)^{2(i_\Delta+1)}x_0^2+1}{2}
= f(\tilde{X}_{i_\Delta}^0),
\end{equation*}
we have $X_{i_\Delta+1} = \tilde{X}_{i_\Delta}^1$, in agreement with~\eqref{eq:ExampleApocalypseSmallestSizeP2GDR}.
Let us now assume that \eqref{eq:ExampleApocalypseSmallestSizeP2GDR} holds for some integer $i > i_\Delta$ and prove that it also holds for $i+1$. As $\hat{X}_i^0 = X_i$, $-\nabla f(X_i) = \diag(0,(1-\alpha)^{i-i_\Delta}) \in \tancone{\R_{\le 1}^{2 \times 2}}{X_i}$, $\s(X_i; f, \R_{\le 1}^{2 \times 2}) = (1-\alpha)^{i-i_\Delta}$, $X_i - \alpha \nabla f(X_i) = \diag(0,1-(1-\alpha)^{i+1-i_\Delta}) \in \R_{\le 1}^{2 \times 2}$, and $f(X_i) - f(\diag(0,1-(1-\alpha)^{i+1-i_\Delta})) \ge c \, \alpha \, \s(X_i; f, \R_{\le 1}^{2 \times 2})^2$,
we have $\tilde{X}_i^0 = \diag(0,1-(1-\alpha)^{i+1-i_\Delta})$. If $\rank_\Delta X_i = 0$, then $\ppgdr$ also considers $\hat{X}_i^1 = 0_{2 \times 2}$ and, from what precedes, $\tilde{X}_i^1 = \diag(0,\alpha)$. Since $f(\tilde{X}_i^0) < f(\tilde{X}_i^1)$, we have $X_{i+1} = \tilde{X}_i^0$, as wished. The other two claims follow.
\end{proof}

The iterates of $\ppgdr$ computed in Proposition~\ref{prop:ExampleApocalypseSmallestSizeP2GDR} are represented in Figure~\ref{fig:ExampleApocalypseSmallestSize}, which illustrates how the apocalypse is avoided. As explained, $\ppgd$ follows an apocalypse because, at any point $\diag(x,0)$ with $x \in (0,\infty)$, the projection of $-\nabla f$ onto the tangent cone to $\R_{\le 1}^{2 \times 2}$ is parallel to the $x$-axis, and can thus minimize only the first term of $f$. The descent direction $\diag(0,1)$, which enables the minimization of the second term of $f$, becomes accessible only at $\diag(0,0)$.

\begin{figure}[h]
\begin{center}
\begin{tikzpicture}
\def\x{1}
\def\a{0.6}
\def\D{0.2}
\pgfmathsetmacro\iD{ceil(ln(\D/\x)/ln(1-\a))}
\draw [->] (0,0) -- (2,0) node [right] {$x$};
\draw [->] (0,0) -- (0,3) node [above] {$y$};
\foreach \i in {0, 1, ..., 4}
{
\draw [dashed] (0,1) circle ({0.4*\i});
}
\foreach \i in {0, 1, ..., \iD}
{
\draw ({\x*(1-\a)^\i},0) node {$\boldsymbol\cdot$};
\draw [->] ({\x*(1-\a)^\i},0) -- ({\x*(1-\a)^(\i+1)},\a);
}
\foreach \i in {1, 2, ..., 3}
{
\draw (0,{1-(1-\a)^\i},0) node {$\boldsymbol\cdot$};
}
\end{tikzpicture}
\end{center}
\caption{Iterates $X_i$ produced by $\ppgdr$ for the problem of Section~\ref{subsubsec:P2GDfollowingApocalypseSize2*2} with $x_0 := 1$, $\alpha := \frac{3}{5}$, and $\Delta := \frac{1}{5}$ in the $xy$-plane of $\diag(x,y)$ matrices. The arrows represent $-\alpha\nabla f(X_i)$. The circles represent sublevel sets of $f$.}
\label{fig:ExampleApocalypseSmallestSize}
\end{figure}
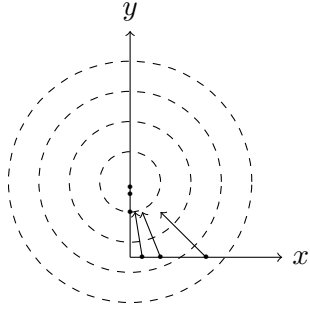

Although $\ppgdr$ avoids the apocalypse for every $\Delta \in (0, \infty)$, it should be noted that, if $\Delta \ge \alpha$, then its rank reduction mechanism makes it apply the $\ppgd$ map to $0_{2 \times 2}$ in at least one iteration from iteration $i_\Delta+1$, thereby constructing points that are not used, as shown in the proof of Proposition~\ref{prop:ExampleApocalypseSmallestSizeP2GDR}. For those iterations, $\ppgdr$ therefore produces the same iterates as $\ppgd$ at a higher computational cost.

We close this section by discussing how $\ppgdr$ behaves on this problem if it is used with the stopping criterion defined by \eqref{eq:StoppingCriterion}, i.e., if it is stopped when the stationarity measure $\s(\cdot; f, \R_{\le 1}^{2 \times 2})$ becomes smaller than or equal to some threshold $\varepsilon \in (0, \infty)$. If $\Delta := 0$, it returns the sequence $(X_i)_{i=0}^{i=i_\varepsilon}$ defined by~\eqref{eq:ExampleApocalypseSmallestSize}, where $i_\varepsilon := \max\bigg\{\bigg\lceil\frac{\ln(\frac{\varepsilon}{x_0})}{\ln(1-\alpha)}\bigg\rceil,0\bigg\}$. Thus, in view of~\eqref{eq:ExampleApocalypseSmallestSizeP2GDR}, for $\ppgdr$ to avoid stopping while it is heading towards the apocalyptic point, we must have $i_\Delta < i_\varepsilon$, i.e., $\Delta \ge (1-\alpha)^{i_\varepsilon-1}x_0$.

\subsection{The cone of symmetric positive-semidefinite matrices of bounded rank}
\label{subsec:ConePSDmatricesBoundedRank}
In this section, $\mathcal{E} := \R^{n \times n}$ and
\begin{equation*}
C
:= \mathrm{S}_{\le r}^+(n)
:= \{X \in \R_{\le r}^{n \times n} \mid X^\tp = X,\, X \succeq 0\}
\end{equation*}
for some positive integers $n$ and $r < n$, $\R^{n \times n}$ is endowed with the Frobenius inner product, and $\norm{\cdot}$ denotes the Frobenius norm. For every $q \in \N$, we let $\mathrm{S}(q) := \{X \in \R^{q \times q} \mid X^\tp = X\}$ denote the linear subspace of $\R^{q \times q}$ consisting of all real $q \times q$ symmetric matrices, $\mathrm{S}^+(q) := \{X \in \mathrm{S}(q) \mid X \succeq 0\}$ the closed convex cone of real $q \times q$ symmetric positive-semidefinite matrices, and $\mathrm{S}^-(q) := \{X \in \mathrm{S}(q) \mid X \preceq 0\}$ the closed convex cone of real $q \times q$ symmetric negative-semidefinite matrices. We also write $\R_*^{n \times r} := \R_r^{n \times r}$ and $\mathrm{S}_{< r}^+(n) := \mathrm{S}^+(n) \cap \R_{< r}^{n \times n}$.

In Section~\ref{subsubsec:StratificationPSDconeBoundedRank}, we review the stratification of $\mathrm{S}_{\le r}^+(n)$ by the rank, which satisfies conditions~1(a) and 1(b) of Assumption~\ref{assumption:Stratification}. We also recall basic facts about real symmetric positive-semidefinite matrices showing that condition~1(c) is satisfied and providing a formula to project onto $\mathrm{S}_{\le r}^+(n)$ and its strata. In Section~\ref{subsubsec:TangentConePSDconeBoundedRank}, we review an explicit description of the tangent cone to $\mathrm{S}_{\le r}^+(n)$ and derive a formula to project onto it (Proposition~\ref{prop:ProjectionOntoTangentConePSDconeBoundedRank}). Based on this description, we prove that $\mathrm{S}_{\le r}^+(n)$ satisfies the second statement of Theorem~\ref{thm:ExamplesStratifiedSetsSatisfyingMainAssumption} (Proposition~\ref{prop:GlobalSecondOrderUpperBoundDistanceToPSDconeBoundedRankFromTangentLine}) and condition~3 of Assumption~\ref{assumption:Stratification} (Proposition~\ref{prop:ContinuityTangentConeStratumPSDconeBoundedRank}). In Section~\ref{subsubsec:NormalConesPSDconeBoundedRank}, we deduce the regular normal cone, the normal cone, and the Clarke normal cone to $\mathrm{S}_{\le r}^+(n)$ and show that the sets of apocalyptic and serendipitous points of $\mathrm{S}_{\le r}^+(n)$ both equal $\mathrm{S}_{< r}^+(n)$ (Proposition~\ref{prop:PSDconeBoundedRankApocalypticSerendipitousPoints}). Finally, in Section~\ref{subsubsec:P2GDRpsdConeBoundedRank}, we present an alternative version of the $\ppgdr$ map on $\mathrm{S}_{\le r}^+(n)$ (Algorithm~\ref{algo:P2GDRmapPSDconeBoundedRank}) and show that the general theory developed in Section~\ref{sec:ProposedAlgorithmConvergenceAnalysis} also applies to this version (Proposition~\ref{prop:P2GDRmapPSDconeBoundedRankPolak}). We notably deduce Corollary~\ref{coro:P2GDRpsdConeBoundedRankPolakConvergence}.

\subsubsection{Stratification of the cone of positive-semidefinite matrices of bounded rank}
\label{subsubsec:StratificationPSDconeBoundedRank}
The rank stratifies $\mathrm{S}_{\le r}^+(n)$:
\begin{equation*}
\mathrm{S}_{\le r}^+(n) = \bigcup_{i=0}^r \mathrm{S}_i^+(n)
\end{equation*}
where, for every $i \in \{0, \dots, r\}$,
\begin{equation*}
\mathrm{S}_i^+(n) := \mathrm{S}^+(n) \cap \R_i^{n \times n}
\end{equation*}
is the smooth manifold of $n \times n$ rank-$i$ symmetric positive-semidefinite matrices \cite[Proposition~2.1]{HelmkeShayman1995}. Observe that $\mathrm{S}_{\le 0}^+(n) = \mathrm{S}_0^+(n) = \{0_{n \times n}\}$.
The stratification of $\mathrm{S}^+(n)$ by the rank follows from the fact that $\mathrm{S}^+(n)$ is in bijection with the orbit space $\R^{n \times n}/\mathrm{O}(n)$, as shown in \cite[\S 3.1]{ThanwerdasPennec2022} based on \cite{AlekseevskyKrieglLosikMichor}.

It follows that $\mathrm{S}_{\le r}^+(n)$ satisfies condition~1(a) of Assumption~\ref{assumption:Stratification}. By \cite[Proposition~2.1]{HelmkeShayman1995}, condition~1(b) is satisfied too.
To establish condition~1(c), we first review basic facts about the eigenvalues of a real symmetric matrix.

In what follows, the eigenvalues of $X \in \mathrm{S}(n)$, which are real \cite[Theorem~4.1.3]{HornJohnson}, are denoted by $\lambda_1(X) \ge \dots \ge \lambda_n(X)$, as in \cite[2.1.7]{GolubVanLoan}; moreover, $\lambda_1(X)$ and $\lambda_n(X)$ are respectively denoted by $\lambda_{\max}(X)$ and $\lambda_{\min}(X)$.
By the spectral theorem for real symmetric matrices \cite[Theorem~8.1.1]{GolubVanLoan}, for every $X \in \mathrm{S}(n)$, there exists $U \in \mathrm{O}(n)$ such that
\begin{equation}
\label{eq:EigendecompositionSymmetric}
X = U \diag(\lambda_1(X), \dots, \lambda_n(X)) U^\tp.
\end{equation}
Moreover, if $X \succeq 0$, then $\lambda_{\min}(X) \ge 0$ \cite[Theorem~4.1.8]{HornJohnson}, thus the eigendecomposition~\eqref{eq:EigendecompositionSymmetric} is an SVD, and the singular values of $X$ are its eigenvalues.

Proposition~\ref{prop:ProjectionOntoPSDconeBoundedRank} shows how to project onto $\mathrm{S}_{\le r}^+(n)$ and implies that $\mathrm{S}_{\le r}^+(n)$ satisfies condition~1(c) of Assumption~\ref{assumption:Stratification}.

\begin{proposition}[{projection onto $\mathrm{S}_{\le r}^+(n)$ \cite[Corollary~17]{Dax2014}}]
\label{prop:ProjectionOntoPSDconeBoundedRank}
For every $X \in \R^{n \times n}$, $\proj{\mathrm{S}_{\le r}^+(n)}{X}$ is the set of all possible outputs of Algorithm~\ref{algo:ProjectionPSDconeBoundedRank}; in particular, if $X_\mathrm{sym} := \frac{1}{2}(X+X^\tp)$ and
\begin{equation*}
i := \left\{\begin{array}{ll}
\max\{j \in \{1, \dots, r\} \mid \lambda_j(X_\mathrm{sym}) > 0\} & \text{if } \lambda_1(X_\mathrm{sym}) > 0,\\
0 & \text{otherwise},
\end{array}\right.
\end{equation*}
then
\begin{equation*}
\dist(X, \mathrm{S}_{\le r}^+(n))
= \sqrt{\norm{X}^2 - \sum_{j=1}^i \lambda_j^2(X_\mathrm{sym})}.
\end{equation*}
\end{proposition}

\begin{algorithm}[H]
\caption{Projection onto $\mathrm{S}_{\le r}^+(n)$}
\label{algo:ProjectionPSDconeBoundedRank}
\begin{algorithmic}[1]
\Require
$(n, r)$ where $n, r \in \N \setminus \{0\}$ and $r < n$.
\Input
$X \in \R^{n \times n}$.
\Output
$Y \in \proj{\mathrm{S}_{\le r}^+(n)}{X}$.

\State
$X_\mathrm{sym} \gets \frac{1}{2}(X+X^\tp)$;
\State
Choose $U \in \mathrm{O}(n)$ such that $X_\mathrm{sym} = U \diag(\lambda_1(X_\mathrm{sym}), \dots, \lambda_n(X_\mathrm{sym})) U^\tp$;
\If
{$\lambda_1(X_\mathrm{sym}) > 0$}
	\State
	$i \gets \max\{j \in \{1, \dots, r\} \mid \lambda_j(X_\mathrm{sym}) > 0\}$;
	\State
	$Y \gets U(:, 1\mathord{:}i) \diag(\lambda_1(X_\mathrm{sym}), \dots, \lambda_i(X_\mathrm{sym})) U(:, 1\mathord{:}i)^\tp$;
\Else
	\State
	$Y \gets 0_{n \times n}$;
\EndIf
\State
Return $Y$.
\end{algorithmic}
\end{algorithm}

\subsubsection{Tangent cone to the cone of positive-semidefinite matrices of bounded rank}
\label{subsubsec:TangentConePSDconeBoundedRank}
In Proposition~\ref{prop:ProjectionOntoTangentConePSDconeBoundedRank}, we review a formula describing $\tancone{\mathrm{S}_{\le r}^+(n)}{X}$ for every $X \in \mathrm{S}_{\le r}^+(n)$ based on orthonormal bases of $\im X$ and $(\im X)^\perp$, and we deduce a formula to project onto $\tancone{\mathrm{S}_{\le r}^+(n)}{X}$.

\begin{proposition}[tangent cone to $\mathrm{S}_{\le r}^+(n)$]
\label{prop:ProjectionOntoTangentConePSDconeBoundedRank}
Let $\ushort{r} \in \{0, \dots, r\}$, $X \in \mathrm{S}_{\ushort{r}}^+(n)$, $U \in \st(\ushort{r}, n)$, $U_\perp \in \st(n-\ushort{r}, n)$, $\im U = \im X$, and $\im U_\perp = (\im X)^\perp$.
Then,
\begin{equation*}
\tancone{\mathrm{S}_{\le r}^+(n)}{X} = \left\{[U \; U_\perp] \begin{bmatrix} A & B \\ B^\tp & E \end{bmatrix} [U \; U_\perp]^\tp \mid A \in \mathrm{S}(\ushort{r}),\, B \in \R^{\ushort{r} \times n-\ushort{r}},\, E \in \mathrm{S}_{\le r-\ushort{r}}^+(n-\ushort{r})\right\}.
\end{equation*}
Moreover, if $Z \in \R^{n \times n}$ is written as
\begin{equation*}
Z = [U \; U_\perp] \begin{bmatrix} A & B \\ D & E \end{bmatrix} [U \; U_\perp]^\tp
\end{equation*}
with $A = U^\tp Z U$, $B = U^\tp Z U_\perp$, $D = U_\perp^\tp Z U$, and $E = U_\perp^\tp Z U_\perp$, then
\begin{equation*}
\proj{\tancone{\mathrm{S}_{\le r}^+(n)}{X}}{Z} = [U \; U_\perp] \begin{bmatrix} \frac{1}{2}(A+A^\top) & \frac{1}{2}(B+D^\tp) \\[1mm] \frac{1}{2}(B^\tp+D) & \proj{\mathrm{S}_{\le r-\ushort{r}}^+(n-\ushort{r})}{E} \end{bmatrix} [U \; U_\perp]^\tp.
\end{equation*}
\end{proposition}

\begin{proof}
The description of $\tancone{\mathrm{S}_{\le r}^+(n)}{X}$ is given in \cite[Proposition~3.4]{LevinKileelBoumal2022SmoothLifts}.
All $\tilde{Z} \in \tancone{\mathrm{S}_{\le r}^+(n)}{X}$ can be written as
\begin{equation*}
\tilde{Z} = [U \; U_\perp] \begin{bmatrix} \tilde{A} & \tilde{B} \\ \tilde{B}^\tp & \tilde{E} \end{bmatrix} [U \; U_\perp]^\tp
\end{equation*}
with $\tilde{A} \in \mathrm{S}(\ushort{r})$, $\tilde{B} \in \R^{\ushort{r} \times n-\ushort{r}}$, and $\tilde{E} \in \mathrm{S}_{\le r-\ushort{r}}^+(n-\ushort{r})$, and
\begin{equation*}
\norm{Z-\tilde{Z}}^2 = \norm{A-\tilde{A}}^2 + \norm{B-\tilde{B}}^2 + \norm{D-\tilde{B}^\tp}^2 + \norm{E-\tilde{E}}^2
\end{equation*}
is minimized if and only if $\tilde{A} \in \proj{\mathrm{S}(\ushort{r})}{A} = \frac{1}{2}(A+A^\tp)$, $\tilde{B} = \frac{1}{2}(B+D^\tp)$, and $\tilde{E} \in \proj{\mathrm{S}_{\le r-\ushort{r}}^+(n-\ushort{r})}{E}$.
\end{proof}

Proposition~\ref{prop:GlobalSecondOrderUpperBoundDistanceToPSDconeBoundedRankFromTangentLine} shows that $\mathrm{S}_{\le r}^+(n)$ satisfies the second statement of Theorem~\ref{thm:ExamplesStratifiedSetsSatisfyingMainAssumption}.

\begin{proposition}
\label{prop:GlobalSecondOrderUpperBoundDistanceToPSDconeBoundedRankFromTangentLine}
For all $X \in \mathrm{S}_{\ushort{r}}^+(n)$ with $\ushort{r} \in \{1, \dots, r\}$,
\begin{equation*}
\frac{\sqrt{5}-1}{r+1} \frac{1}{2\lambda_{\ushort{r}}(X)}
\le \sup_{Z \in \tancone{\mathrm{S}_{\le r}^+(n)}{X} \setminus \{0_{n \times n}\}} \frac{\dist(X+Z, \mathrm{S}_{\le r}^+(n))}{\norm{Z}^2}
\le \frac{1}{2\lambda_{\ushort{r}}(X)}.
\end{equation*}
\end{proposition}

\begin{proof}
We prove only the upper bound; the lower bound can be obtained as the one in Proposition~\ref{prop:GlobalSecondOrderUpperBoundDistanceToRealDeterminantalVarietyFromTangentLine}. Let
\begin{equation*}
X = [U \; U_\perp] \diag(\Lambda, 0_{n-\ushort{r} \times n-\ushort{r}}) [U \; U_\perp]^\tp
\end{equation*}
be an eigendecomposition, and $Z \in \tancone{\mathrm{S}_{\le r}^+(n)}{X} \setminus \{0_{n \times n}\}$. By Proposition~\ref{prop:ProjectionOntoTangentConePSDconeBoundedRank}, there are $A \in \mathrm{S}(\ushort{r})$, $B \in \R^{\ushort{r} \times n-\ushort{r}}$, and $E \in \mathrm{S}_{\le r-\ushort{r}}^+(n-\ushort{r})$ such that
\begin{equation*}
Z =
[U \; U_\perp]
\begin{bmatrix}
A & B \\ B^\tp & E
\end{bmatrix}
[U \; U_\perp]^\tp.
\end{equation*}
Define the function
\begin{equation*}
\gamma : [0,\infty) \to \mathrm{S}_{\le r}^+(n) : t \mapsto \big(U+t(U_\perp B^\tp + {\textstyle\frac{1}{2}}UA)\Lambda^{-1}\big) \Lambda \big(U+t(U_\perp B^\tp + {\textstyle\frac{1}{2}}UA^\tp)\Lambda^{-1}\big)^\tp + tU_\perp E U_\perp^\tp;
\end{equation*}
$\gamma$ is well defined since the ranks of the two terms are respectively upper bounded by $\ushort{r}$ and $r-\ushort{r}$, and the sum of two positive-semidefinite matrices is positive-semidefinite.
For all $t \in [0,\infty)$,
\begin{equation*}
\gamma(t)
= X + t Z + \frac{t^2}{4}
[U \; U_\perp]
\begin{bmatrix}
A \\ 2B^\tp
\end{bmatrix}
\Lambda^{-1}
\begin{bmatrix}
A & 2B
\end{bmatrix}
[U \; U_\perp]^\tp
\end{equation*}
thus
\begin{equation*}
\dist(X+tZ, \mathrm{S}_{\le r}^+(n))
\le \norm{(X+tZ)-\gamma(t)}
= \frac{t^2}{4} \left\|\begin{bmatrix} A\Lambda^{-1}A & 2A\Lambda^{-1}B \\ 2B^\tp\Lambda^{-1}A & 4B^\tp\Lambda^{-1}B\end{bmatrix}\right\|.
\end{equation*}
Observe that
\begin{align*}
\left\|\begin{bmatrix} A\Lambda^{-1}A & 2A\Lambda^{-1}B \\ 2B^\tp\Lambda^{-1}A & 4B^\tp\Lambda^{-1}B\end{bmatrix}\right\|
&= \sqrt{\norm{A\Lambda^{-1}A}^2 + 4 \norm{A\Lambda^{-1}B}^2 + 4 \norm{B^\tp\Lambda^{-1}A}^2 + 16 \norm{B^\tp\Lambda^{-1}B}^2}\\
&= \sqrt{\norm{A\Lambda^{-1}A}^2 + 8 \norm{A\Lambda^{-1}B}^2 + 16 \norm{B^\tp\Lambda^{-1}B}^2}\\
&\le \norm{\Lambda^{-1}}_2 \left(\norm{A}^2 + 4 \norm{B}^2\right)\\
&\le \norm{\Lambda^{-1}}_2 \norm{Z}^2 \max_{\substack{x, y \in \R \\ x^2+2y^2 = 1}} x^2 + 4y^2\\
&= 2 \norm{\Lambda^{-1}}_2 \norm{Z}^2\\
&= \frac{2}{\lambda_{\ushort{r}}(X)} \norm{Z}^2.
\end{align*}
Therefore, for all $t \in [0,\infty)$,
\begin{equation*}
\dist(X+tZ, \mathrm{S}_{\le r}^+(n)) \le t^2 \frac{1}{2 \lambda_{\ushort{r}}(X)} \norm{Z}^2.
\end{equation*}
Choosing $t = 1$ yields the result.
\end{proof}

We now prove Proposition~\ref{prop:ContinuityTangentConeStratumPSDconeBoundedRank} which states that $\mathrm{S}_{\le r}^+(n)$ satisfies condition~3 of Assumption~\ref{assumption:Stratification}. To this end, we need some preliminary results.
Proposition~\ref{prop:ContinuityProjectionStiefel} allows us to deduce Lemma~\ref{lemma:ConvergenceTangentConeDecompositionConstantRank} from \cite[Lemma~4.1]{OlikierAbsil2022}.

\begin{proposition}
\label{prop:ContinuityProjectionStiefel}
For every $W \in \R_*^{n \times r}$, $\proj{\st(r,n)}{W} = U$ and $\im U = \im W$. Moreover, the function
\begin{equation*}
\proj{\st(r,n)}{\cdot} : \R_*^{n \times r} \to \st(r,n)
\end{equation*}
is continuous.
\end{proposition}

\begin{proof}
The first part follows from \cite[Theorem~10.2]{Bhattacharya2}. The second part then follows from \cite[Theorem~2.26]{FletcherMoors2015} since $\st(r, n)$ is compact.
\end{proof}

\begin{lemma}
\label{lemma:ConvergenceTangentConeDecompositionConstantRank}
Let $\ushort{r} \in \{0, \dots, n\}$, $(X_i)_{i \in \N}$ be a sequence in $\mathrm{S}_{\ushort{r}}^+(n)$ converging to $X \in \mathrm{S}_{\ushort{r}}^+(n)$, $U \in \st(\ushort{r}, n)$, $\im U = \im X$, $U_\perp \in \st(n-\ushort{r}, n)$, and $\im U_\perp = (\im X)^\perp$.
Then, there exist sequences $(U_i)_{i \in \N}$ in $\st(\ushort{r}, n)$ and $(U_{i\perp})_{i \in \N}$ in $\st(n-\ushort{r}, n)$ respectively converging to $U$ and $U_\perp$, and such that, for all $i \in \N$, $\im U_i = \im X_i$ and $\im U_{i\perp} = (\im X_i)^\perp$.
\end{lemma}

\begin{proof}
By \cite[Lemma~4.1]{OlikierAbsil2022}, there exist sequences $(\tilde{U}_i)_{i \in \N}$ in $\R_*^{n \times \ushort{r}}$ and $(\tilde{U}_{i\perp})_{i \in \N}$ in $\R_*^{n \times n-\ushort{r}}$ respectively converging to $U$ and $U_\perp$, and such that, for all $i \in \N$, $\im \tilde{U}_i = \im X_i$ and $\im \tilde{U}_{i\perp} = (\im X_i)^\perp$. By Proposition~\ref{prop:ContinuityProjectionStiefel}, we can take $U_i := \proj{\st(\ushort{r}, n)}{\tilde{U}_i}$ and $U_{i\perp} := \proj{\st(n-\ushort{r}, n)}{\tilde{U}_{i\perp}}$ for all $i \in \N$.
\end{proof}

Lemma~\ref{lemma:ConvergenceTangentConeDecompositionConstantRank} allows us to prove Proposition~\ref{prop:ContinuityTangentConeStratumPSDconeBoundedRank} which states that $\mathrm{S}_{\le r}^+(n)$ satisfies condition~3 of Assumption~\ref{assumption:Stratification}.

\begin{proposition}
\label{prop:ContinuityTangentConeStratumPSDconeBoundedRank}
For every $\ushort{r} \in \{0, \dots, r\}$, the correspondence $\tancone{\mathrm{S}_{\le r}^+(n)}{\cdot}$ is continuous at every $X \in \mathrm{S}_{\ushort{r}}^+(n)$ relative to $\mathrm{S}_{\le \ushort{r}}^+(n)$.
\end{proposition}

\begin{proof}
The result is clear if $\ushort{r} = 0$ since $\mathrm{S}_{\le 0}^+(n) = \mathrm{S}_0^+(n) = \{0_{n \times n}\}$. Let us therefore consider $\ushort{r} \in \{1, \dots, r\}$ and $X \in \mathrm{S}_{\ushort{r}}^+(n)$. We must prove that, for every sequence $(X_i)_{i \in \N}$ in $\mathrm{S}_{\le \ushort{r}}^+(n)$ converging to $X \in \mathrm{S}_{\ushort{r}}^+(n)$, it holds that
\begin{equation*}
\outlim_{i \to \infty} \tancone{\mathrm{S}_{\le r}^+(n)}{X_i}
\subseteq \tancone{\mathrm{S}_{\le r}^+(n)}{X}
\subseteq \inlim_{i \to \infty} \tancone{\mathrm{S}_{\le r}^+(n)}{X_i}.
\end{equation*}
We recall that the concepts of inner and outer limits of a sequence of sets have been reviewed in Section~\ref{subsec:InnerOuterLimitsContinuityCorrespondences}.
By Proposition~\ref{prop:ProjectionOntoPSDconeBoundedRank}, every sequence in $\mathrm{S}_{\le \ushort{r}}^+(n)$ converging to a point in $\mathrm{S}_{\ushort{r}}^+(n)$ contains finitely many elements in $\mathrm{S}_{< \ushort{r}}^+(n)$. Thus, it suffices to consider a sequence $(X_i)_{i \in \N}$ in $\mathrm{S}_{\ushort{r}}^+(n)$ converging to $X \in \mathrm{S}_{\ushort{r}}^+(n)$. Let $U \in \st(\ushort{r}, n)$, $\im U = \im X$, $U_\perp \in \st(n-\ushort{r}, n)$, and $\im U_\perp = (\im X)^\perp$. We apply Lemma~\ref{lemma:ConvergenceTangentConeDecompositionConstantRank} to $(X_i)_{i \in \N}$ and $X$.

Let us establish the first inclusion. Let $Z \in \outlim_{i \to \infty} \tancone{\mathrm{S}_{\le r}^+(n)}{X_i}$, i.e., $Z$ is an accumulation point of a sequence $(Z_i)_{i \in \N}$ such that, for all $i \in \N$, $Z_i \in \tancone{\mathrm{S}_{\le r}^+(n)}{X_i}$. We need to prove that $Z \in \tancone{\mathrm{S}_{\le r}^+(n)}{X}$.
By Proposition~\ref{prop:ProjectionOntoTangentConePSDconeBoundedRank}, for all $i \in \N$,
$
Z_i = [U_i \; U_{i\perp}] \left[\begin{smallmatrix}
A_i & B_i \\ B_i^\tp & E_i
\end{smallmatrix}\right] [U_i \; U_{i\perp}]^\tp
$
with $A_i = U_i^\tp Z_i U_i \in \mathrm{S}(\ushort{r})$, $B_i = U_i^\tp Z_i U_{i\perp} \in \R^{\ushort{r} \times n-\ushort{r}}$, and $E_i = U_{i\perp}^\tp Z_i U_{i\perp} \in \mathrm{S}_{\le r-\ushort{r}}^+(n-\ushort{r})$.
Let $(Z_{i_k})_{k \in \N}$ be a subsequence of $(Z_i)_{i \in \N}$ converging to $Z$.
Then, for all $k \in \N$,
$
Z_{i_k} = [U_{i_k} \; U_{i_k\perp}] \left[\begin{smallmatrix}
A_{i_k} & B_{i_k} \\ B_{i_k}^\tp & E_{i_k}
\end{smallmatrix}\right] [U_{i_k} \; U_{i_k\perp}]^\tp,
$
and, since the subsequences $(A_{i_k})_{k \in \N}$, $(B_{i_k})_{k \in \N}$, and $(E_{i_k})_{k \in \N}$ respectively converge to $A := U^\tp Z U \in \mathrm{S}(\ushort{r})$, $B := U^\tp Z U_\perp$, and $E := U_\perp^\tp Z U_\perp \in \mathrm{S}_{\le r-\ushort{r}}^+(n-\ushort{r})$, we have
$
Z = [U \; U_\perp] \left[\begin{smallmatrix}
A & B \\ B^\tp & E
\end{smallmatrix}\right] [U \; U_\perp]^\tp,
$
which shows that $Z \in \tancone{\mathrm{S}_{\le r}^+(n)}{X}$ by Proposition~\ref{prop:ProjectionOntoTangentConePSDconeBoundedRank}.

Let us establish the second inclusion. Let $Z \in \tancone{\mathrm{S}_{\le r}^+(n)}{X}$. We have to prove that there exists a sequence $(Z_i)_{i \in \N}$ converging to $Z$ and such that $Z_i \in \tancone{\mathrm{S}_{\le r}^+(n)}{X_i}$ for all $i \in \N$. By Proposition~\ref{prop:ProjectionOntoTangentConePSDconeBoundedRank}, there exist $A \in \mathrm{S}(\ushort{r})$, $B \in \R^{\ushort{r} \times n-\ushort{r}}$, and $E \in \mathrm{S}_{\le r-\ushort{r}}^+(n-\ushort{r})$ such that
$
Z = [U \; U_\perp] \left[\begin{smallmatrix}
A & B \\ B^\tp & E
\end{smallmatrix}\right] [U \; U_\perp]^\tp.
$
By Proposition~\ref{prop:ProjectionOntoTangentConePSDconeBoundedRank}, for all $i \in \N$,
$
Z_i := [U_i \; U_{i\perp}] \left[\begin{smallmatrix}
A & B \\ B^\tp & E
\end{smallmatrix}\right] [U_i \; U_{i\perp}]^\tp \in \tancone{\mathrm{S}_{\le r}^+(n)}{X_i}.
$
Since $(Z_i)_{i \in \N}$ converges to $Z$, the proof is complete.
\end{proof}

\subsubsection{Normal cones to the cone of positive-semidefinite matrices of bounded rank}
\label{subsubsec:NormalConesPSDconeBoundedRank}
In this section, we compute the polar of $\mathrm{S}_{\le r}^+(n)$ (Proposition~\ref{prop:PolarPSDconeBoundedRank}), the regular normal cone to $\mathrm{S}_{\le r}^+(n)$ (Proposition~\ref{prop:RegularNormalConePSDconeBoundedRank}), the normal cone to $\mathrm{S}_{\le r}^+(n)$ (Proposition~\ref{prop:NormalConePSDconeBoundedRank}), and the Clarke normal cone to $\mathrm{S}_{\le r}^+(n)$ (Corollary~\ref{coro:ClarkeNormalConePSDconeBoundedRank}). To this end, we use Proposition~\ref{prop:TangentNormalConesAmbientSpace} and the fact that
\begin{equation*}
\mathrm{S}(n)^\perp = \{X \in \R^{n \times n} \mid X^\tp = -X\}.
\end{equation*}
Finally, using Proposition~\ref{prop:CharacterizationApocalypticSerendipitousPoint}, we prove that the sets of apocalyptic points and of serendipitous points of $\mathrm{S}_{\le r}^+(n)$ both equal $\mathrm{S}_{< r}^+(n)$ (Proposition~\ref{prop:PSDconeBoundedRankApocalypticSerendipitousPoints}).

\begin{proposition}[polar of $\mathrm{S}_{\le r}^+(n)$]
\label{prop:PolarPSDconeBoundedRank}
For all $\ushort{r} \in \{1, \dots, n\}$,
\begin{equation*}
\mathrm{S}_{\le \ushort{r}}^+(n)^*
= \mathrm{S}^-(n) + \mathrm{S}(n)^\perp.
\end{equation*}
\end{proposition}

\begin{proof}
By Proposition~\ref{prop:PolarAmbientSpace}, it suffices to prove that $\mathrm{S}_{\le \ushort{r}}^+(n)^* \cap \mathrm{S}(n) = \mathrm{S}^-(n)$.
The inclusion $\supseteq$ holds since, for all $X, Y \in \mathrm{S}^+(n)$, $\ip{X}{Y} \ge 0$. Let us establish the inclusion $\subseteq$. Let $X \in \mathrm{S}_{\le \ushort{r}}^+(n)^* \cap \mathrm{S}(n)$. Then, there exists $U \in \mathrm{O}(n)$ such that $X = U \diag(\lambda_1(X), \dots, \lambda_n(X)) U^\tp$. Moreover, for all $i \in \{1, \dots, n\}$, $Y := U \diag(\delta_{1, i}, \dots, \delta_{n, i}) U^\tp \in \mathrm{S}_{\le \ushort{r}}^+(n)$ and thus $0 \ge \ip{X}{Y} = \lambda_i(X)$. Therefore, $X \preceq 0$.
\end{proof}

\begin{proposition}[regular normal cone to $\mathrm{S}_{\le r}^+(n)$]
\label{prop:RegularNormalConePSDconeBoundedRank}
For all $X \in \mathrm{S}_{\le r}^+(n)$,
\begin{equation*}
\regnorcone{\mathrm{S}_{\le r}^+(n)}{X}
= \mathrm{S}(n)^\perp + \left\{\begin{array}{ll}
\{Z \in \mathrm{S}^-(n) \mid XZ = 0_{n \times n}\} & \text{if } \rank X < r,\\
\{Z \in \mathrm{S}(n) \mid XZ = 0_{n \times n}\} & \text{if } \rank X = r.
\end{array}\right.
\end{equation*}
\end{proposition}

\begin{proof}
By \eqref{eq:RegularNormalCone} and Propositions~\ref{prop:ProjectionOntoTangentConePSDconeBoundedRank} and \ref{prop:PolarPSDconeBoundedRank}, we have
\begin{equation*}
\regnorcone{\mathrm{S}_{\le r}^+(n)}{0_{n \times n}}
= \tancone{\mathrm{S}_{\le r}^+(n)}{0_{n \times n}}^*
= \mathrm{S}_{\le r}^+(n)^*
= \mathrm{S}^-(n) + \mathrm{S}(n)^\perp.
\end{equation*}
Let $\ushort{r} \in \{1, \dots, r\}$, $X \in \mathrm{S}_{\ushort{r}}^+(n)$, $U \in \st(\ushort{r}, n)$, $U_\perp \in \st(n-\ushort{r}, n)$, $\im U = \im X$, and $\im U_\perp = (\im X)^\perp$. By Proposition~\ref{prop:TangentNormalConesAmbientSpace}, it suffices to prove that
\begin{equation*}
\regnorcone{\mathrm{S}_{\le r}^+(n)}{X} \cap \mathrm{S}(n) = \left\{\begin{array}{ll}
\{Z \in \mathrm{S}^-(n) \mid XZ = 0_{n \times n}\} & \text{if } \ushort{r} < r,\\
\{Z \in \mathrm{S}(n) \mid XZ = 0_{n \times n}\} & \text{if } \ushort{r} = r.
\end{array}\right.
\end{equation*}
Let $Z \in \mathrm{S}(n)$ be written as
\begin{equation*}
Z = [U \; U_\perp] \begin{bmatrix} \tilde{A} & \tilde{B} \\ \tilde{B}^\tp & \tilde{D} \end{bmatrix} [U \; U_\perp]^\tp
\end{equation*}
with $\tilde{A} = U^\tp Z U$, $\tilde{B} = U^\tp Z U_\perp$, and $\tilde{D} = U_\perp^\tp Z U_\perp$. Then, $Z \in \regnorcone{\mathrm{S}_{\le r}^+(n)}{X} \cap \mathrm{S}(n)$ if and only if $\ip{Z}{Y} \le 0$ for all $Y \in \tancone{\mathrm{S}_{\le r}^+(n)}{X}$. By Proposition~\ref{prop:ProjectionOntoTangentConePSDconeBoundedRank}, all $Y \in \tancone{\mathrm{S}_{\le r}^+(n)}{X}$ can be written as
\begin{equation*}
Y = [U \; U_\perp] \begin{bmatrix} A & B \\ B^\tp & D \end{bmatrix} [U \; U_\perp]^\tp,
\end{equation*}
with $A \in \mathrm{S}(\ushort{r})$, $B \in \R^{\ushort{r} \times n-\ushort{r}}$, and $D \in \mathrm{S}_{\le r-\ushort{r}}^+(n-\ushort{r})$, and
\begin{equation*}
\ip{Z}{Y} = \tr A \tilde{A} + 2 \tr B \tilde{B}^\tp + \tr D \tilde{D}.
\end{equation*}
Thus, $\ip{Z}{Y} \le 0$ for all $Y \in \tancone{\mathrm{S}_{\le r}^+(n)}{X}$ if and only if $\tilde{A} = 0_{\ushort{r} \times \ushort{r}}$, $\tilde{B} = 0_{\ushort{r} \times n-\ushort{r}}$, and $\tilde{D} \in \mathrm{S}_{\le r-\ushort{r}}^+(n-\ushort{r})^* \cap \mathrm{S}(n-\ushort{r})$. Therefore, $Z \in U_\perp \mathrm{S}^-(n-\ushort{r}) U_\perp^\tp$ if $\ushort{r} < r$ (by Proposition~\ref{prop:PolarPSDconeBoundedRank}) and $Z \in U_\perp \mathrm{S}(n-r) U_\perp^\tp$ if $\ushort{r} = r$ (because $\{0_{n-r \times n-r}\}^* = \R^{n-r \times n-r}$). The result follows.
\end{proof}

The normal cone to $\mathrm{S}_{\le r}^+(n)$ in $\mathrm{S}(n)$ is given in \cite[Theorem~3.12]{Tam2017}. In Proposition~\ref{prop:NormalConePSDconeBoundedRank}, we deduce the normal to $\mathrm{S}_{\le r}^+(n)$ in $\R^{n \times n}$ thanks to Proposition~\ref{prop:TangentNormalConesAmbientSpace}.

\begin{proposition}[normal cone to $\mathrm{S}_{\le r}^+(n)$]
\label{prop:NormalConePSDconeBoundedRank}
For all $X \in \mathrm{S}_{\le r}^+(n)$,
\begin{equation*}
\norcone{\mathrm{S}_{\le r}^+(n)}{X}
= \left\{Z \in \mathrm{S}^-(n) \cup \mathrm{S}_{\le n-r}(n) \mid XZ = 0_{n \times n}\right\} + \mathrm{S}(n)^\perp.
\end{equation*}
In particular, $\mathrm{S}_{\le r}^+(n)$ is not Clarke regular on $\mathrm{S}_{< r}^+(n)$.
\end{proposition}

\begin{proof}
This follows from Proposition~\ref{prop:TangentNormalConesAmbientSpace} and \cite[Theorem~3.12]{Tam2017}.
\end{proof}

\begin{corollary}[Clarke normal cone to $\mathrm{S}_{\le r}^+(n)$]
\label{coro:ClarkeNormalConePSDconeBoundedRank}
For all $X \in \mathrm{S}_{\le r}^+(n)$,
\begin{equation*}
\connorcone{\mathrm{S}_{\le r}^+(n)}{X}
= \left\{Z \in \mathrm{S}(n) \mid XZ = 0_{n \times n}\right\} + \mathrm{S}(n)^\perp.
\end{equation*}
\end{corollary}

By Proposition~\ref{prop:NormalConePSDconeBoundedRank}, $\mathrm{S}_{\le r}^+(n)$ is not Clarke regular on $\mathrm{S}_{< r}^+(n)$. Proposition~\ref{prop:PSDconeBoundedRankApocalypticSerendipitousPoints} states that every point of $\mathrm{S}_{< r}^+(n)$ is apocalyptic, which is a stronger result by \cite[Corollary~2.15]{LevinKileelBoumal2022}.

\begin{proposition}
\label{prop:PSDconeBoundedRankApocalypticSerendipitousPoints}
The set of apocalyptic points of $\mathrm{S}_{\le r}^+(n)$ and the set of serendipitous points of $\mathrm{S}_{\le r}^+(n)$ both equal $\mathrm{S}_{< r}^+(n)$.
\end{proposition}

\begin{proof}
We use Proposition~\ref{prop:CharacterizationApocalypticSerendipitousPoint}.
Let $X \in \mathrm{S}_r^+(n)$ and $(X_i)_{i \in \N}$ be a sequence in $\mathrm{S}_{\le r}^+(n)$ converging to $X$. By Proposition~\ref{prop:ProjectionOntoPSDconeBoundedRank}, $(X_i)_{i \in \N}$ contains finitely many elements in $\mathrm{S}_{< r}^+(n)$. Therefore, by Proposition~\ref{prop:ContinuityTangentConeStratumPSDconeBoundedRank}, $\outlim_{i \to \infty} \tancone{\mathrm{S}_{\le r}^+(n)}{X_i} = \tancone{\mathrm{S}_{\le r}^+(n)}{X}$, and thus $\big(\outlim_{i \to \infty} \tancone{\mathrm{S}_{\le r}^+(n)}{X_i}\big)^* = \regnorcone{\mathrm{S}_{\le r}^+(n)}{X}$. Thus, $X$ is neither apocalyptic nor serendipitous. 

Let $X \in \mathrm{S}_{\ushort{r}}^+(n)$ with $\ushort{r} \in \{1, \dots, r-1\}$. Let $\Lambda := \diag(\lambda_1(X), \dots, \lambda_{\ushort{r}}(X))$ and $U \in \st(\ushort{r}, n)$ be such that $X = U \Lambda U^\tp$. Let $\bar{U}_\perp \in \st(r-\ushort{r}, n)$ and $U_\perp \in \st(n-r, n)$ be such that $[U \; \bar{U}_\perp \; U_\perp] \in \mathrm{O}(n)$. For all $i \in \N$, let $X_i := [U \; \bar{U}_\perp] \diag(\Lambda, \frac{\lambda_{\ushort{r}}(X)}{i+1} I_{r-\ushort{r}}) [U \; \bar{U}_\perp]^\tp$. Thus, for all $i \in \N$, 
\begin{equation*}
\tancone{\mathrm{S}_{\le r}^+(n)}{X_i} = \left\{[U \; \bar{U}_\perp \; U_\perp] \begin{bmatrix} A & B & F \\ B^\tp & D & G \\ F^\tp & G^\tp & 0_{n-r \times n-r} \end{bmatrix} [U \; \bar{U}_\perp \; U_\perp]^\tp ~\bigg|~\begin{array}{l} A \in \mathrm{S}(\ushort{r}),\, B \in \R^{\ushort{r} \times r-\ushort{r}}, \\ D \in \mathrm{S}(r-\ushort{r}), \\ F \in \R^{\ushort{r} \times n-r},\, G \in \R^{r-\ushort{r} \times n-r} \end{array}\right\}.
\end{equation*}
Therefore,
\begin{equation*}
\outlim_{i \to \infty} \tancone{\mathrm{S}_{\le r}^+(n)}{X_i} = \left\{[U \; \bar{U}_\perp \; U_\perp] \begin{bmatrix} A & B & F \\ B^\tp & D & G \\ F^\tp & G^\tp & 0_{n-r \times n-r} \end{bmatrix} [U \; \bar{U}_\perp \; U_\perp]^\tp ~\bigg|~\begin{array}{l} A \in \mathrm{S}(\ushort{r}),\, B \in \R^{\ushort{r} \times r-\ushort{r}}, \\ D \in \mathrm{S}(r-\ushort{r}), \\ F \in \R^{\ushort{r} \times n-r},\, G \in \R^{r-\ushort{r} \times n-r} \end{array}\right\}
\end{equation*}
and, by Proposition~\ref{prop:PolarAmbientSpace},
\begin{equation*}
\Big(\outlim_{i \to \infty} \tancone{\mathrm{S}_{\le r}^+(n)}{X_i}\Big)^* = \mathrm{S}(n)^\perp + U_\perp \R^{n-r \times n-r} U_\perp^\tp.
\end{equation*}
Furthermore, by Proposition~\ref{prop:PolarPSDconeBoundedRank},
\begin{align*}
\regnorcone{\mathrm{S}_{\le r}^+(n)}{X} = \mathrm{S}(n)^\perp + [\bar{U}_\perp \; U_\perp] \mathrm{S}^-(n-\ushort{r}) [\bar{U}_\perp \; U_\perp]^\tp.
\end{align*}
Thus, since neither of $\regnorcone{\mathrm{S}_{\le r}^+(n)}{X}$ and $\Big(\outlim_{i \to \infty} \tancone{\mathrm{S}_{\le r}^+(n)}{X_i}\Big)^*$ is a subset of the other, $X$ is apocalyptic and serendipitous.
The argument is the same if $X = 0_{n \times n}$.
\end{proof}

\subsubsection{A variant of the $\ppgdr$ map on the cone of positive-semidefinite matrices of bounded rank}
\label{subsubsec:P2GDRpsdConeBoundedRank}
The variant of the $\ppgdr$ map on $\R_{\le r}^{m \times n}$ given in Algorithm~\ref{algo:P2GDRmapRealDeterminantalVariety} can also be defined on $\mathrm{S}_{\le r}^+(n)$, yielding Algorithm~\ref{algo:P2GDRmapPSDconeBoundedRank}.

\begin{algorithm}[H]
\caption{Variant of the $\ppgdr$ map on $\mathrm{S}_{\le r}^+(n)$}
\label{algo:P2GDRmapPSDconeBoundedRank}
\begin{algorithmic}[1]
\Require
$(f, r, \ushort{\alpha}, \bar{\alpha}, \beta, c, \Delta)$ where $f : \R^{n \times n} \to \R$ is differentiable with $\nabla f$ locally Lipschitz continuous, $r \le n$ is a positive integer, $0 < \ushort{\alpha} \le \bar{\alpha} < \infty$, $\beta, c \in (0,1)$, and $\Delta \in (0,\infty)$.
\Input
$X \in \mathrm{S}_{\le r}^+(n)$ such that $\s(X; f, \mathrm{S}_{\le r}^+(n)) > 0$.
\Output
$Y \in \text{Algorithm~\ref{algo:P2GDRmapPSDconeBoundedRank}}(X; f, r, \ushort{\alpha}, \bar{\alpha}, \beta, c, \Delta)$.

\For
{$j \in \{0, \dots, \rank X - \rank_\Delta X\}$}
\State
Choose $\hat{X}^j \in \proj{\mathrm{S}_{\rank X - j}^+(n)}{X}$;
\State
Choose $\tilde{X}^j \in \hyperref[algo:P2GDmap]{\ppgd}(\hat{X}^j; \R^{n \times n}, \mathrm{S}_{\le r}^+(n), f, \ushort{\alpha}, \bar{\alpha}, \beta, c)$;
\EndFor
\State
Return $Y \in \argmin_{\{\tilde{X}^j \mid j \in \{0, \dots, \rank X - \rank_\Delta X\}\}} f$.
\end{algorithmic}
\end{algorithm}

Proposition~\ref{prop:P2GDRmapPSDconeBoundedRankPolak} states that this variant satisfies the same decrease guarantee as the original version.

\begin{proposition}
\label{prop:P2GDRmapPSDconeBoundedRankPolak}
Proposition~\ref{prop:P2GDRmapPolak} holds for Algorithm~\ref{algo:P2GDRmapPSDconeBoundedRank}.
\end{proposition}

\begin{proof}
The proof is similar to the one of Proposition~\ref{prop:P2GDRmapRealDeterminantalVarietyPolak}.
\end{proof}

The main result of this section is the following.

\begin{corollary}
\label{coro:P2GDRpsdConeBoundedRankPolakConvergence}
On $\mathrm{S}_{\le r}^+(n)$, Theorem~\ref{thm:P2GDRPolakConvergence} holds if Algorithm~\ref{algo:P2GDRmap} is replaced by Algorithm~\ref{algo:P2GDRmapPSDconeBoundedRank} in line~\ref{algo:P2GDR:P2GDRmap} of Algorithm~\ref{algo:P2GDR}.
\end{corollary}

\begin{proof}
The result follows from Proposition~\ref{prop:P2GDRmapPSDconeBoundedRankPolak} and Theorem~\ref{thm:P2GDRPolakConvergence} since $\mathrm{S}_{\le r}^+(n)$ satisfies Assumption~\ref{assumption:Stratification}.
\end{proof}

By Proposition~\ref{prop:PSDconeBoundedRankApocalypticSerendipitousPoints}, unlike on $\R_{\le r}^{m \times n}$, it is an open question whether $\ppgdr$ on $\mathrm{S}_{\le r}^+(n)$ (Algorithm~\ref{algo:P2GDR} using either Algorithm~\ref{algo:P2GDRmap} or Algorithm~\ref{algo:P2GDRmapPSDconeBoundedRank} in line~\ref{algo:P2GDR:P2GDRmap}) can follow a serendipity, i.e., produce a convergent sequence whose limit is stationary (by Corollary~\ref{coro:P2GDRpsdConeBoundedRankPolakConvergence}) but along which the stationarity measure $\s(\cdot; f, \mathrm{S}_{\le r}^+(n))$ does not go to zero.

\section{Complementary results}
\label{sec:ComplementaryResults}
In this section, we prove complementary results to those presented in Sections~\ref{sec:ProposedAlgorithmConvergenceAnalysis} and \ref{sec:ExamplesStratifiedSetsSatisfyingMainAssumption}. In Section~\ref{subsec:GeometricParabolicDerivability}, we show that Assumption~\ref{assumption:GlobalSecondOrderUpperBoundDistanceFromTangentLine} is related to the concept of parabolic derivability (Proposition~\ref{prop:LocalSecondOrderUpperBoundDistanceFromTangentLineParabolicDerivability}). In Section~\ref{subsec:FiniteStratificationsConditionFrontier}, we prove that conditions~1(a) and 1(b) of Assumption~\ref{assumption:Stratification} imply that $\{S_0, \dots, S_p\}$ is a stratification of $C$ satisfying the condition of the frontier (Proposition~\ref{prop:FiniteStratificationsConditionFrontierVsAssumption}). In Section~\ref{subsec:ConvergenceAnalysisRFDR}, under Assumption~\ref{assumption:StratificationRFDR}, we define $\rfdr$ (Algorithm~\ref{algo:RFDR}) and prove that it produces a sequence whose accumulation points are stationary for~\eqref{eq:OptiProblem} (Theorem~\ref{thm:RFDRPolakConvergence}). In Section~\ref{subsec:SparseVectors}, we prove that $\R_{\le s}^n$ satisfies Assumption~\ref{assumption:StratificationRFDR}. Finally, in Section~\ref{subsec:ConvergenceAnalysisRRAM}, we show that the convergence analysis of the Riemannian rank-adaptive method given in \cite[Algorithm~3]{ZhouEtAl2016} does not apply to all cost functions considered in \cite{ZhouEtAl2016}.

\subsection{Geometric and parabolic derivability of a set and distance from a tangent line}
\label{subsec:GeometricParabolicDerivability}
In this section, given a point $x$ in a subset $S$ of $\mathcal{E}$, we investigate the links between:
\begin{enumerate}
\item the derivability of $v \in \tancone{S}{x}$ and the existence of an upper bound on $\frac{\dist(x+tv, S)}{t}$ holding for all $t \in (0, \infty)$ sufficiently small (Proposition~\ref{prop:LocalFirstOrderUpperBoundDistanceFromTangentLineGeometricDerivability});
\item the parabolic derivability of $S$ at $x \in S$ for $v \in \tancone{S}{x}$ and the existence of an upper bound on $\frac{\dist(x+tv, S)}{t^2}$ holding for all $t \in (0, \infty)$ (Proposition~\ref{prop:LocalSecondOrderUpperBoundDistanceFromTangentLineParabolicDerivability}).
\end{enumerate}
We recall that the concepts of geometric and parabolic derivability are reviewed in Section~\ref{subsec:TangentNormalConesGeometricDerivability}.

For all $x \in S$, all $v \in \tancone{S}{x}$, and all $t \in (0, \infty)$,
\begin{equation*}
\dist(x+tv, S) \le \norm{(x+tv)-x} = t\norm{v}.
\end{equation*}
Proposition~\ref{prop:LocalFirstOrderUpperBoundDistanceFromTangentLineGeometricDerivability} shows that, if $S$ is geometrically derivable at $x$, then the factor in front of $t$ can be made arbitrarily small if $t$ is sufficiently small.

\begin{proposition}
\label{prop:LocalFirstOrderUpperBoundDistanceFromTangentLineGeometricDerivability}
If $S$ is geometrically derivable at $x \in S$, then, for every $v \in \tancone{S}{x}$ and every $\varepsilon \in (0, \infty)$, there exists $\delta \in (0, \infty)$ such that, for all $t \in [0, \delta]$,
\begin{equation*}
\dist(x+tv, S) \le \varepsilon t.
\end{equation*}
\end{proposition}

\begin{proof}
Let $v \in \tancone{S}{x} \setminus \{0\}$. By assumption, there exists $\gamma : [0, \tau] \to \mathcal{E}$ with $\tau \in (0, \infty)$, $\gamma([0, \tau]) \subseteq S$, $\gamma(0) = x$, and $\gamma'(0) = v$. Thus, for all $t \in (0, \tau]$,
\begin{equation*}
\frac{\dist(x+tv, S)}{t} \le \frac{\norm{\gamma(t)-(x+tv)}}{t}.
\end{equation*}
Therefore, since
\begin{equation*}
0
= \lim_{t \searrow 0} \frac{\norm{\gamma(t)-(\gamma(0)+t\gamma'(0))}}{t}
= \lim_{t \searrow 0} \frac{\norm{\gamma(t)-(x+tv)}}{t},
\end{equation*}
it holds that
\begin{equation*}
\lim_{t \searrow 0} \frac{\dist(x+tv, S)}{t} = 0.
\end{equation*}
Thus, for every $\varepsilon \in (0, \infty)$, there exists $\delta \in (0, \infty)$ such that, for all $t \in (0, \delta]$,
\begin{equation*}
\dist(x+tv, S) \le \varepsilon t,
\end{equation*}
which completes the proof.
\end{proof}

Proposition~\ref{prop:LocalSecondOrderUpperBoundDistanceFromTangentLineParabolicDerivability} shows that, if $S$ is parabolically derivable at $x \in S$ for $v \in \tancone{S}{x}$, then, for all $t \in (0, \infty)$, $\frac{\dist(x+tv, S)}{t^2}$ is bounded from above.

\begin{proposition}
\label{prop:LocalSecondOrderUpperBoundDistanceFromTangentLineParabolicDerivability}
If $S$ is parabolically derivable at $x \in S$ for $v \in \tancone{S}{x}$, then
\begin{equation*}
\sup_{t \in (0, \infty)} \frac{\dist(x+tv, S)}{t^2} < \infty.
\end{equation*}
\end{proposition}

\begin{proof}
Let $w \in \sectancone{S}{x}{v}$. By assumption, there exists $\gamma : [0, \tau] \to \mathcal{E}$ with $\tau \in (0, \infty)$, $\gamma([0, \tau]) \subseteq S$, $\gamma(0) = x$, $\gamma'(0) = v$, and $\gamma''(0) = w$. Thus, for all $t \in (0, \tau]$,
\begin{equation*}
\frac{\dist(x+tv+\frac{t^2}{2}w, S)}{\frac{t^2}{2}} \le \frac{\norm{\gamma(t)-(x+tv+\frac{t^2}{2}w)}}{\frac{t^2}{2}}.
\end{equation*}
Therefore, since
\begin{equation*}
0
= \lim_{t \searrow 0} \frac{\norm{\gamma(t)-(\gamma(0)+t\gamma'(0)+\frac{t^2}{2}\gamma''(0))}}{\frac{t^2}{2}}
= \lim_{t \searrow 0} \frac{\norm{\gamma(t)-(x+tv+\frac{t^2}{2}w)}}{\frac{t^2}{2}},
\end{equation*}
it holds that
\begin{equation*}
\lim_{t \searrow 0} \frac{\dist(x+tv+\frac{t^2}{2}w, S)}{\frac{t^2}{2}} = 0.
\end{equation*}
By \cite[Proposition~1.3.17]{Willem}, for all $t \in (0, \tau]$,
\begin{equation*}
\frac{|\dist(x+tv+\frac{t^2}{2}w, S)-\dist(x+tv, S)|}{\frac{t^2}{2}} \le \norm{w}.
\end{equation*}
Thus,
\begin{equation*}
L_0 := \limsup_{t \searrow 0} \frac{\dist(x+tv, S)}{\frac{t^2}{2}} \le \norm{w}.
\end{equation*}
Let $\varepsilon \in (0, \infty)$ and $L_1 := \frac{L_0+\varepsilon}{2}$. There exists $\delta \in (0, \infty)$ such that
\begin{equation*}
\left|\sup_{t \in (0, \delta]} \frac{\dist(x+tv, S)}{\frac{t^2}{2}} - L_0\right| \le \varepsilon.
\end{equation*}
Therefore, for all $t \in [0, \delta]$, $\dist(x+tv, S) \le L_1 t^2$.
Let $L := \max\{L_1, \frac{\norm{v}}{\delta}\}$. On the one hand, for all $t \in [0, \frac{\norm{v}}{L}]$, $\dist(x+tv, S) \le L_1 t^2 \le L t^2$. On the other hand, for all $t \in (\frac{\norm{v}}{L}, \infty)$, $\dist(x+tv, S) \le t \norm{v} < L t^2$. Hence, for all $t \in [0, \infty)$, $\dist(x+tv, S) \le L t^2$.
\end{proof}

\subsection{Finite stratifications satisfying the condition of the frontier}
\label{subsec:FiniteStratificationsConditionFrontier}
In this section, we prove that conditions~1(a) and 1(b) of Assumption~\ref{assumption:Stratification} imply that $\{S_0, \dots, S_p\}$ is a stratification of $C$ satisfying the condition of the frontier (Proposition~\ref{prop:FiniteStratificationsConditionFrontierVsAssumption}).

Let $Z$ be a topological space. By \cite[définition~2 \& proposition~5]{BourbakiTopologie}, a subset of $Z$ is locally closed if and only if it is the intersection of an open set and a closed set.

\begin{proposition}
\label{prop:DisjointLocallyClosedSets}
If $A_1$ and $A_2$ are two nonempty locally closed subsets of $Z$, then $A_1 \cap A_2 = \emptyset$ implies $\overline{A_1} \ne \overline{A_2}$.
\end{proposition}

\begin{proof}
For every $i \in \{1, 2\}$, since $A_i$ is nonempty and locally closed, there exist a nonempty open set $O_i$ and a nonempty closed set $C_i$ such that $A_i = O_i \cap C_i$. Then, either $A_1 \cap O_2 = \emptyset$ or $A_1 \cap O_2 \ne \emptyset$.
If $A_1 \cap O_2 = \emptyset$, then $A_2 \subseteq \overline{A_2} \setminus \overline{A_1}$. Indeed, if $x \in A_2$, then, $x \in \overline{A_2}$ and $x \in O_2$. Thus, $O_2$ is a neighborhood of $x$ that does not meet $A_1$. Therefore, $x \not\in \overline{A_1}$.
If $A_1 \cap O_2 \ne \emptyset$, then $A_1 \cap O_2 \subseteq \overline{A_1} \setminus \overline{A_2}$. Indeed, let $x \in A_1 \cap O_2$. Then, $x \in \overline{A_1}$. Furthermore, since $A_1 \cap A_2 = \emptyset$, $x \not\in C_2$. Since $Z \setminus C_2$ is open, there exists a neighborhood of $x$ contained in $Z \setminus C_2$ and thus in $Z \setminus A_2$. Therefore, $x \not\in \overline{A_2}$.
\end{proof}

\begin{proposition}
\label{prop:MatherConditionFrontier}
Let $\mathcal{S}$ be a partition of $Z$ the blocks of which are locally closed.
If, for every $S \in \mathcal{S}$, $\overline{S} \setminus S$ is a union of elements of $\mathcal{S}$, then:
\begin{enumerate}
\item the relation $<$ defined as follows is a strict partial order on $\mathcal{S}$:
\begin{equation*}
S < S' \text{ iff } S \ne S' \text{ and } S \subseteq \overline{S'} \setminus S';
\end{equation*}
\item if there exists $S \in \mathcal{S}$ such that $S < S'$ for all $S' \in \mathcal{S}$, then $S$ is closed.
\end{enumerate}
\end{proposition}

\begin{proof}
\begin{enumerate}
\item We have to prove that $<$ is irreflexive and transitive. The irreflexivity is clear. Let us prove the transitivity. Let $S, S', S'' \in \mathcal{S}$ be such that $S < S'$ and $S' < S''$. Then, $S \subseteq \overline{S''}$. Moreover, $S \ne S''$; if not, then $\overline{S} = \overline{S'}$ which is impossible in view of Proposition~\ref{prop:DisjointLocallyClosedSets} since $S \ne S'$. Thus, $S \subseteq \overline{S''} \setminus S''$. Therefore, $S < S''$.

\item Assume, for the sake of contradiction, that $\overline{S} \setminus S \ne \emptyset$. Then, since $\overline{S} \setminus S$ is a union of elements of $\mathcal{S}$, there exists $S' \in \mathcal{S} \setminus \{S\}$ such that $S' \subseteq \overline{S} \setminus S$, i.e., $S' < S$, a contradiction.
\qedhere
\end{enumerate}
\end{proof}

\begin{proposition}
\label{prop:FiniteStratificationsConditionFrontierVsAssumption}
Let $\mathcal{E}$ be a Euclidean vector space and $C$ be a nonempty closed subset of $\mathcal{E}$. The following statements are equivalent:
\begin{enumerate}
\item $C$ admits a finite stratification satisfying the condition of the frontier \cite[\S 5]{Mather} and for which the partial order defined in Proposition~\ref{prop:MatherConditionFrontier} is total;
\item conditions 1(a) and 1(b) of Assumption~\ref{assumption:Stratification} are satisfied.
\end{enumerate}
\end{proposition}

\begin{proof}
Clearly, the second statement implies the first. Let us prove the converse. In view of the definition given in \cite[\S 5]{Mather}, if $C$ admits a finite stratification, then $C$ can be partitioned into finitely many smooth submanifolds of $\mathcal{E}$ contained in $C$; thus, condition~1(a) holds. Since any submanifold of $\mathcal{E}$ is locally closed \cite[\S 1]{Mather}, if the stratification satisfies the condition of the frontier as formulated in \cite[\S 5]{Mather}, then Proposition~\ref{prop:MatherConditionFrontier} defines a strict partial order $<$. If, moreover, $<$ is total, then the stratification can be written as $\{S_0, \dots, S_p\}$ for some nonnegative integer $p$ with, if $p \ge 1$, $S_i < S_{i+1}$ for all $i \in \{0, \dots, p-1\}$. Let us prove that condition~1(b) holds. By Proposition~\ref{prop:MatherConditionFrontier}, $\overline{S_0} = S_0$. If $p \ge 1$, let $i \in \{1, \dots, p\}$. Then, $\bigcup_{j < i} S_j \subseteq \overline{S_i} \setminus S_i$. Let us prove that this inclusion is an equality. Assume, for the sake of contradiction, that $\overline{S_i} \setminus S_i$ is not a subset of $\bigcup_{j < i} S_j$. Then, by the condition of the frontier, there exists $k \in \{i+1, \dots, p\}$ such that $S_k \subseteq \overline{S_i} \setminus S_i$, i.e., $k < i$, a contradiction.
\end{proof}

\subsection{Convergence analysis of $\rfdr$ under Assumption~\ref{assumption:StratificationRFDR}}
\label{subsec:ConvergenceAnalysisRFDR}
In this section, under Assumption~\ref{assumption:StratificationRFDR}, we define $\rfdr$ (Algorithm~\ref{algo:RFDR}) and prove that it produces a sequence whose accumulation points are stationary for~\eqref{eq:OptiProblem} (Theorem~\ref{thm:RFDRPolakConvergence}). The organization of the section is described hereafter and summarized in Table~\ref{tab:AssumptionsAlgorithmsMainResultsRFDR}.
In Section~\ref{subsubsec:RFDmap}, under condition~3 of Assumption~\ref{assumption:StratificationRFDR}, we prove that the $\rfd$ map (Algorithm~\ref{algo:RFDmap}) produces a point satisfying an Armijo condition (Corollary~\ref{coro:RFDmapArmijoCondition}).
In Section~\ref{subsubsec:RFDRmap}, under Assumption~\ref{assumption:StratificationRFDR}, we introduce the $\rfdr$ map (Algorithm~\ref{algo:RFDRmap}), which uses the $\rfd$ map as a subroutine, and, based on Corollary~\ref{coro:RFDmapArmijoCondition}, prove Proposition~\ref{prop:RFDRmapPolak}.
Finally, in Section~\ref{subsubsec:RFDR}, we introduce the $\rfdr$ algorithm and prove Theorem~\ref{thm:RFDRPolakConvergence} based on Proposition~\ref{prop:RFDRmapPolak}. Using the concept of serendipitous point (see Section~\ref{subsec:StationarityMeasure}), we also deduce Corollary~\ref{coro:RFDRPolakConvergence} from Theorem~\ref{thm:RFDRPolakConvergence}.

\begin{table}[h]
\begin{center}
\begin{tabular}{llll}
\hline
\emph{Section} & \emph{Assumption} & \emph{Algorithm} & \emph{Main result}\\
\hline
Section~\ref{subsubsec:RFDmap} & condition~3 of Assumption~\ref{assumption:StratificationRFDR} & $\rfd$ map (Algorithm~\ref{algo:RFDmap}) & Corollary~\ref{coro:RFDmapArmijoCondition}\\
\hline
Section~\ref{subsubsec:RFDRmap} & \multirow{2}{*}{Assumption~\ref{assumption:StratificationRFDR}} & $\rfdr$ map (Algorithm~\ref{algo:RFDRmap}) & Proposition~\ref{prop:RFDRmapPolak}\\
\cline{1-1}\cline{3-4}
Section~\ref{subsubsec:RFDR} & & $\rfdr$ (Algorithm~\ref{algo:RFDR}) & Theorem~\ref{thm:RFDRPolakConvergence}\\
\hline
\end{tabular}
\end{center}
\caption{Assumptions, algorithms, and main results of Section~\ref{subsec:ConvergenceAnalysisRFDR}.}
\label{tab:AssumptionsAlgorithmsMainResultsRFDR}
\end{table}

\subsubsection{The $\rfd$ map}
\label{subsubsec:RFDmap}
In this section, under condition~3 of Assumption~\ref{assumption:StratificationRFDR}, we prove that the $\rfd$ map (Algorithm~\ref{algo:RFDmap}) is well defined and produces a point satisfying an Armijo condition (Corollary~\ref{coro:RFDmapArmijoCondition}).
For convenience, we recall that condition~3 of Assumption~\ref{assumption:StratificationRFDR} states that $C$ admits a \emph{restricted tangent cone}, i.e., a correspondence $C \setmapsto \mathcal{E} : x \mapsto \restancone{C}{x}$ such that:
\begin{enumerate}
\item for every $x \in C$, $\restancone{C}{x}$ is a closed cone contained in $\tancone{C}{x}$;
\item for all $x \in C$ and all $z \in \restancone{C}{x}$, $x+z \in C$;
\item there exists $\mu \in (0, 1]$ such that, for all $x \in C$ and all $z \in \mathcal{E}$, $\norm{\proj{\restancone{C}{x}}{z}} \ge \mu \norm{\proj{\tancone{C}{x}}{z}}$.
\end{enumerate}

If $C = \R_{\le r}^{m \times n}$, the $\rfd$ map corresponds to the iteration map of \cite[Algorithm~4]{SchneiderUschmajew2015} except that the initial step size for the backtracking procedure is chosen in a given bounded interval.

\begin{algorithm}[H]
\caption{$\rfd$ map}
\label{algo:RFDmap}
\begin{algorithmic}[1]
\Require
$(\mathcal{E}, C, f, \ushort{\alpha}, \bar{\alpha}, \beta, c)$ where $\mathcal{E}$ is a Euclidean vector space, $C$ is a nonempty closed subset of $\mathcal{E}$ satisfying condition~3 of Assumption~\ref{assumption:StratificationRFDR}, $f : \mathcal{E} \to \R$ is differentiable with $\nabla f$ locally Lipschitz continuous, $0 < \ushort{\alpha} \le \bar{\alpha} < \infty$, and $\beta, c \in (0,1)$.
\Input
$x \in C$ such that $\s(x; f, C) > 0$.
\Output
a point in $\rfd(x; \mathcal{E}, C, f, \ushort{\alpha}, \bar{\alpha}, \beta, c)$.

\State
Choose $g \in \proj{\restancone{C}{x}}{-\nabla f(x)}$ and $\alpha \in [\ushort{\alpha}, \bar{\alpha}]$;
\label{algo:RFDmap:FirstLine}
\While
{$f(x+\alpha g) > f(x) - c \, \alpha \norm{g}^2$}
\State
$\alpha \gets \alpha \beta$;
\EndWhile
\State
Return $x+\alpha g$.
\end{algorithmic}
\end{algorithm}

\begin{proposition}
\label{prop:RFDmapUpperBoundCost}
Let $x \in C$ and $\bar{\alpha} \in (0,\infty)$. Assume that $C$ satisfies condition~3 of Assumption~\ref{assumption:StratificationRFDR}.
Let $\mathcal{B} \subsetneq \mathcal{E}$ be a closed ball such that, for all $g \in \proj{\restancone{C}{x}}{-\nabla f(x)}$ and all $\alpha \in [0, \bar{\alpha}]$, $x+\alpha g \in \mathcal{B}$; an example of such a ball is $\ball[x, \bar{\alpha}\s(x; f, C)]$.
Then, for all $g \in \proj{\restancone{C}{x}}{-\nabla f(x)}$ and all $\alpha \in [0, \bar{\alpha}]$,
\begin{equation}
\label{eq:RFDmapUpperBoundCost}
f(x+\alpha g) \le f(x) + \norm{g}^2 \alpha \left(-1+\frac{\lip_\mathcal{B}(\nabla f)}{2}\alpha\right).
\end{equation}
\end{proposition}

\begin{proof}
The example $\ball[x, \bar{\alpha}\s(x; f, C)]$ is correct because, for all $g \in \proj{\restancone{C}{x}}{-\nabla f(x)}$ and all $\alpha \in [0, \bar{\alpha}]$,
\begin{equation*}
\norm{(x+\alpha g)-x}
= \alpha \norm{g}
\le \bar{\alpha} \s(x; f, C).
\end{equation*}
Let $g \in \proj{\restancone{C}{x}}{-\nabla f(x)}$. The proof of~\eqref{eq:RFDmapUpperBoundCost} is based on~\eqref{eq:InequalityLipschitzContinuousGradient} and the equality $\ip{\nabla f(x)}{g} = -\norm{g}^2$ which holds by Proposition~\ref{prop:NormProjectionOntoClosedCone} since $\restancone{C}{x}$ is a closed cone. For all $\alpha \in [0, \bar{\alpha}]$,
\begin{align*}
f(x+\alpha g)-f(x)
&\le \ip{\nabla f(x)}{(x+\alpha g)-x} + \frac{\lip_\mathcal{B}(\nabla f)}{2} \norm{(x+\alpha g)-x}^2\\
&= -\alpha \norm{g}^2 + \frac{\lip_\mathcal{B}(\nabla f)}{2} \alpha^2 \norm{g}^2.
\qedhere
\end{align*}
\end{proof}

Two remarks on Proposition~\ref{prop:RFDmapUpperBoundCost} can be made. First, the existence of a ball $\mathcal{B}$ crucially relies on the upper bound $\bar{\alpha}$ required by Algorithm~\ref{algo:RFDmap}. Second, in contrast with \eqref{eq:P2GDmapUpperBoundCost}, the upper bound \eqref{eq:RFDmapUpperBoundCost} does not depend on the function $u$ defined in Assumption~\ref{assumption:GlobalSecondOrderUpperBoundDistanceFromTangentLine}. This fundamental difference is the reason why $\rfdr$ accumulates at stationary points of~\eqref{eq:OptiProblem} while requiring at most one projection onto $S_{p-1}$ per iteration, whereas $\ppgdr$ can require a projection onto $S_i$ for every $i \in \{0, \dots, p-1\}$ per iteration.

Corollary~\ref{coro:RFDmapArmijoCondition} states that the while loop in Algorithm~\ref{algo:RFDmap} terminates and produces a point satisfying an Armijo condition. It plays an instrumental role in the proof of Proposition~\ref{prop:RFDRmapPolak}. Observe that the $\mu^2$ factor, which comes from condition~3(c) of Assumption~\ref{assumption:StratificationRFDR}, does not appear in the Armijo condition given in Corollary~\ref{coro:P2GDmapArmijoCondition}.

\begin{corollary}
\label{coro:RFDmapArmijoCondition}
The while loop in Algorithm~\ref{algo:RFDmap} terminates and every $\tilde{x} \in \hyperref[algo:RFDmap]{\rfd}(x; \mathcal{E}, C, f, \ushort{\alpha}, \bar{\alpha}, \beta, c)$ satisfies the Armijo condition
\begin{equation*}
f(\tilde{x}) \le f(x) - \mu^2 c \, \alpha \s(x; f, C)^2
\end{equation*}
for some $\alpha \in \left[\min\left\{\ushort{\alpha}, 2\beta\frac{1-c}{\lip_\mathcal{B}(\nabla f)}\right\}, \bar{\alpha}\right]$, where $\mathcal{B}$ is any closed ball as in Proposition~\ref{prop:RFDmapUpperBoundCost}.
\end{corollary}

\begin{proof}
For all $\alpha \in (0,\infty)$,
\begin{equation*}
f(x) + \norm{g}^2 \alpha \left(-1+\frac{\lip_\mathcal{B}(\nabla f)}{2}\alpha\right)
\le f(x) - c \norm{g}^2 \alpha
\quad \text{iff} \quad \alpha \le 2\frac{1-c}{\lip_\mathcal{B}(\nabla f)}.
\end{equation*}
Since the left-hand side of the first inequality is an upper bound on $f(x+\alpha g)$ for all $\alpha \in (0, \bar{\alpha}]$, the Armijo condition is necessarily satisfied if $\alpha \in (0, \min\{\bar{\alpha}, 2\frac{1-c}{\lip_\mathcal{B}(\nabla f)}\}]$.
Therefore, either the initial step size chosen in $[\ushort{\alpha}, \bar{\alpha}]$ satisfies the Armijo condition or the while loop ends with $\alpha$ such that $\frac{\alpha}{\beta} > 2\frac{1-c}{\lip_\mathcal{B}(\nabla f)}$. The result then follows from condition~3(c) of Assumption~\ref{assumption:StratificationRFDR}.
\end{proof}

\subsubsection{The $\rfdr$ map}
\label{subsubsec:RFDRmap}
In this section, under Assumption~\ref{assumption:StratificationRFDR}, we introduce the $\rfdr$ map (Algorithm~\ref{algo:RFDRmap}) and prove Proposition~\ref{prop:RFDRmapPolak} from which we deduce Theorem~\ref{thm:RFDRPolakConvergence} in Section~\ref{subsubsec:RFDR}.
For convenience, we recall that Assumption~\ref{assumption:StratificationRFDR} states that $C$ satisfies the following conditions:
\begin{enumerate}
\item there exist a positive integer $p$ and nonempty smooth submanifolds $S_0, \dots, S_p$ of $\mathcal{E}$ contained in $C$ such that:
\begin{enumerate}
\item for all $i, j \in \{0, \dots, p\}$, $i \ne j$ implies $S_i \cap S_j =\emptyset$;
\item $\overline{S_p} = C$ and, for all $i \in \{0, \dots, p\}$, $\overline{S_i} = \bigcup_{j=0}^i S_j$;
\item if $p \ge 2$, then, for all $x \in S_p$, $\dist(x, S_{p-1}) < \dist(x, S_{p-2})$;
\end{enumerate}
\item $\inf_{x \in C \setminus S_p, z \in \mathcal{E} \setminus \{0\}} \frac{\norm{\proj{\tancone{C}{x}}{z}}}{\norm{z}} > 0$;
\item $C$ admits a \emph{restricted tangent cone}, i.e., a correspondence $C \setmapsto \mathcal{E} : x \mapsto \restancone{C}{x}$ such that:
\begin{enumerate}
\item for every $x \in C$, $\restancone{C}{x}$ is a closed cone contained in $\tancone{C}{x}$;
\item for all $x \in C$ and all $z \in \restancone{C}{x}$, $x+z \in C$;
\item there exists $\mu \in (0, 1]$ such that, for all $x \in C$ and all $z \in \mathcal{E}$, $\norm{\proj{\restancone{C}{x}}{z}} \ge \mu \norm{\proj{\tancone{C}{x}}{z}}$.
\end{enumerate}
\end{enumerate}
The $\rfdr$ map involves the projection of a point of $S_p$ onto $S_{p-1}$ which exists by condition~1(c) above.

The $\rfdr$ map is defined as Algorithm~\ref{algo:RFDRmap}. Given $x \in C$ as input, it proceeds as follows: it applies the $\rfd$ map (Algorithm~\ref{algo:RFDmap}) to $x$, thereby producing a point $\tilde{x}$, (ii) if $x \in S_p$ and $d(x, S_{p-1}) \le \Delta$ for some threshold $\Delta \in (0, \infty)$, it applies the $\rfd$ map to a projection $\hat{x}$ of $x$ onto $S_{p-1}$, then producing a point $\tilde{x}^\mathrm{R}$, and (iii) it outputs a point among $\tilde{x}$ and $\tilde{x}^\mathrm{R}$ that maximally decreases $f$.

\begin{algorithm}[H]
\caption{$\rfdr$ map}
\label{algo:RFDRmap}
\begin{algorithmic}[1]
\Require
$(\mathcal{E}, C, f, \ushort{\alpha}, \bar{\alpha}, \beta, c, \Delta)$ where $\mathcal{E}$ is a Euclidean vector space, $C$ is a nonempty closed subset of $\mathcal{E}$ satisfying Assumption~\ref{assumption:StratificationRFDR}, $f : \mathcal{E} \to \R$ is differentiable with $\nabla f$ locally Lipschitz continuous, $0 < \ushort{\alpha} \le \bar{\alpha} < \infty$, $\beta, c \in (0,1)$, and $\Delta \in (0,\infty)$.
\Input
$x \in C$ such that $\s(x; f, C) > 0$.
\Output
a point in $\rfdr(x; \mathcal{E}, C, f, \ushort{\alpha}, \bar{\alpha}, \beta, c, \Delta)$.

\State
Choose $\tilde{x} \in \hyperref[algo:RFDmap]{\rfd}(x; \mathcal{E}, C, f, \ushort{\alpha}, \bar{\alpha}, \beta, c)$;
\If
{$d(x, S_{p-1}) \in (0,\Delta]$}
\State
Choose $\hat{x} \in \proj{S_{p-1}}{X}$;
\State
Choose $\tilde{x}^\mathrm{R} \in \hyperref[algo:RFDmap]{\rfd}(\hat{x}; \mathcal{E}, C, f, \ushort{\alpha}, \bar{\alpha}, \beta, c)$;
\State
Return $y \in \argmin_{\{\tilde{x}, \tilde{x}^\mathrm{R}\}} f$.
\Else
\State
Return $\tilde{x}$.
\EndIf
\end{algorithmic}
\end{algorithm}

Proposition~\ref{prop:RFDRmapPolak} is to the $\rfdr$ map as Proposition~\ref{prop:P2GDRmapPolak} is to the $\ppgdr$ map.

\begin{proposition}
\label{prop:RFDRmapPolak}
For every $\ushort{x} \in C$ such that $\s(\ushort{x}; f, C) > 0$, there exist $\varepsilon(\ushort{x}), \delta(\ushort{x}) \in (0,\infty)$ such that, for all $x \in \ball[\ushort{x}, \varepsilon(\ushort{x})] \cap C$ and all $y \in \hyperref[algo:RFDRmap]{\rfdr}(x; \mathcal{E}, C, f, \ushort{\alpha}, \bar{\alpha}, \beta, c, \Delta)$,
\begin{equation}
\label{eq:PolakLocallyUniformSufficientDecreaseRFDR}
f(y) - f(x) \le - \delta(\ushort{x}).
\end{equation}
\end{proposition}

\begin{proof}
Let $\ushort{x} \in C$ be such that $\s(\ushort{x}; f, C) > 0$. This proof constructs $\varepsilon(\ushort{x})$ and $\delta(\ushort{x})$ based on the Armijo condition given in Corollary~\ref{coro:RFDmapArmijoCondition}. This requires to derive local lower and upper bounds on $\s(\cdot; f, C)$ around $\ushort{x}$. It first considers the case where $\ushort{x} \in S_p$, in which the construction essentially relies on the continuity of $\s(\cdot; f, C)$ on $\ball(\ushort{x}, d(\ushort{x}, S_{p-1})) \cap C$. Then, it focuses on the case where $\ushort{x} \in C \setminus S_p$, in which the bounds on $\s(\cdot; f, C)$ follow from the bounds \eqref{eq:RFDRPolakInequalityNormGradient} on $\nabla f$ thanks to condition~2 of Assumption~\ref{assumption:StratificationRFDR}. If $x \in C \setminus S_p$, then condition~2 of Assumption~\ref{assumption:StratificationRFDR} readily gives a lower bound on $\s(x; f, C)$. This is not the case if $x \in S_p$, however. This is where the projection mechanism comes into play. It considers a projection $\hat{x}$ of $x$ onto $S_{p-1}$, and condition~2 of Assumption~\ref{assumption:StratificationRFDR} gives a lower bound on $\s(\hat{x}; f, C)$. The inequality \eqref{eq:PolakLocallyUniformSufficientDecreaseRFDR} is then obtained from \eqref{eq:RFDRPolakCostContinuous} which follows from the continuity of $f$ at $\ushort{x}$.

Let us first consider the case where $\ushort{x} \in S_p$. On $\ball(\ushort{x}, d(\ushort{x}, S_{p-1})) \cap C = \ball(\ushort{x}, d(\ushort{x}, S_{p-1})) \cap S_p$, $\s(\cdot; f, C)$ coincides with the norm of the Riemannian gradient of the restriction of $f$ to the smooth manifold $S_p$, which is continuous. In particular, there exists $\rho(\ushort{x}) \in (0, d(\ushort{x}, S_{p-1}))$ such that $\s(\ball[\ushort{x}, \rho(\ushort{x})] \cap S_p; f, C) \subseteq [\frac{1}{2}\s(\ushort{x}; f, C), \frac{3}{2}\s(\ushort{x}; f, C)]$. Let $\bar{\rho}(\ushort{x}) := \rho(\ushort{x}) + \frac{3}{2} \bar{\alpha} \s(\ushort{x}; f, C)$, $\varepsilon(\ushort{x}) := \rho(\ushort{x})$, and
\begin{equation*}
\delta(\ushort{x}) := \frac{1}{4} c \mu^2 \s(\ushort{x}; f, C)^2 \min\left\{\ushort{\alpha}, 2\beta\frac{1-c}{\lip_{\ball[\ushort{x},\bar{\rho}(\ushort{x})]}(\nabla f)}\right\}.
\end{equation*}
Let $x \in \ball[\ushort{x}, \varepsilon(\ushort{x})] \cap C$. Then, $\ball[x, \bar{\alpha} \s(x; f, C)] \subseteq \ball[\ushort{x}, \bar{\rho}(\ushort{x})]$. Indeed, for all $z \in \ball[x, \bar{\alpha} \s(x; f, C)]$,
\begin{equation*}
\norm{z-\ushort{x}}
\le \norm{z-x} + \norm{x-\ushort{x}}
\le \bar{\alpha} \s(x; f, C) + \rho(\ushort{x})
\le \bar{\rho}(\ushort{x}).
\end{equation*}
Therefore, Corollary~\ref{coro:RFDmapArmijoCondition} applies and, for all $\tilde{x} \in \hyperref[algo:RFDmap]{\rfd}(x; \mathcal{E}, C, f, \ushort{\alpha}, \bar{\alpha}, \beta, c)$, 
\begin{equation*}
f(\tilde{x})
\le f(x) - c \mu^2 \s(x; f, C)^2 \min\left\{\ushort{\alpha}, 2\beta\frac{1-c}{\lip_{\ball[\ushort{x},\bar{\rho}(\ushort{x})]}(\nabla f)}\right\}
\le f(x) - \delta(\ushort{x}),
\end{equation*}
and thus, for all $y \in \hyperref[algo:RFDRmap]{\rfdr}(x; \mathcal{E}, C, f, \ushort{\alpha}, \bar{\alpha}, \beta, c, \Delta)$,
\begin{equation*}
f(y)
\le f(\tilde{x})
\le f(x) - \delta(\ushort{x}).
\end{equation*}
This completes the proof for the case where $\ushort{x} \in S_p$.

Let us now consider the case where $\ushort{x} \in C \setminus S_p$. It holds that $\norm{\nabla f(\ushort{x})} \ge \s(\ushort{x}; f, C)$. Let $\bar{\rho}(\ushort{x}) := \Delta + \frac{3}{2} \bar{\alpha} \norm{\nabla f(\ushort{x})}$ and
\begin{equation}
\label{eq:RFDRPolakDelta}
\delta(\ushort{x}) := \frac{1}{12} c \mu^2 \left(\inf_{\substack{\check{x} \in C \setminus S_p \\ z \in \mathcal{E} \setminus \{0\}}} \frac{\norm{\proj{\tancone{C}{\check{x}}}{z}}}{\norm{z}}\right)^2 \norm{\nabla f(\ushort{x})}^2 \min\left\{\ushort{\alpha}, 2\beta\frac{1-c}{\lip_{\ball[\ushort{x}, \bar{\rho}(\ushort{x})]}(\nabla f)}\right\}.
\end{equation}
Since $f$ is continuous at $\ushort{x}$, there exists $\rho_f(\ushort{x}) \in (0, \infty)$ such that $f(\ball[\ushort{x}, \rho_f(\ushort{x})]) \subseteq [f(\ushort{x})-\delta(\ushort{x}), f(\ushort{x})+\delta(\ushort{x})]$. Let $\varepsilon(\ushort{x}) := \frac{1}{2} \min\Big\{\Delta, \rho_f(\ushort{x}), \frac{\norm{\nabla f(\ushort{x})}}{2\lip_{\ball[\ushort{x},\Delta]}(\nabla f)}\Big\}$. Then, for all $z \in \ball[\ushort{x}, 2\varepsilon(\ushort{x})]$, since
\begin{equation*}
|\norm{\nabla f(z)}-\norm{\nabla f(\ushort{x})}|
\le \norm{\nabla f(z)-\nabla f(\ushort{x})}
\le \lip_{\ball[\ushort{x},\Delta]}(\nabla f) \norm{z-\ushort{x}}
\le \lip_{\ball[\ushort{x},\Delta]}(\nabla f) 2\varepsilon(\ushort{x})
\le \frac{\norm{\nabla f(\ushort{x})}}{2},
\end{equation*}
it holds that
\begin{equation}
\label{eq:RFDRPolakInequalityNormGradient}
\frac{1}{2} \norm{\nabla f(\ushort{x})} \le \norm{\nabla f(z)} \le \frac{3}{2} \norm{\nabla f(\ushort{x})}.
\end{equation}
Let $x \in \ball[\ushort{x}, \varepsilon(\ushort{x})] \cap C$.
Let us first consider the case where $x \in S_p$. Then,
\begin{equation*}
0 < d(x, S_{p-1}) \le \norm{x-\ushort{x}} \le \varepsilon(\ushort{x}) \le \Delta.
\end{equation*}
Thus, given $x$ as input, the $\rfdr$ map considers $\hat{x} \in \proj{S_{p-1}}{x} \subseteq \ball[\ushort{x}, 2\varepsilon(\ushort{x})] \cap S_{p-1}$ and $\tilde{x}^\mathrm{R} \in \hyperref[algo:RFDmap]{\rfd}(\hat{x}; \mathcal{E}, C, f, \ushort{\alpha}, \bar{\alpha}, \beta, c)$, where the inclusion holds because
\begin{equation*}
\norm{\hat{x}-\ushort{x}}
\le \norm{\hat{x}-x} + \norm{x-\ushort{x}}
\le d(x, S_{p-1}) + \varepsilon(\ushort{x})
\le 2 \varepsilon(\ushort{x}).
\end{equation*}
As $x, \hat{x} \in \ball[\ushort{x}, \rho_f(\ushort{x})]$, we have
\begin{equation}
\label{eq:RFDRPolakCostContinuous}
f(\hat{x}) \le f(x) + 2\delta(\ushort{x}).
\end{equation}
Since $\ball[\hat{x}, \bar{\alpha} \s(\hat{x}, f, C)] \subseteq \ball[\ushort{x}, \bar{\rho}(\ushort{x})]$,  Corollary~\ref{coro:RFDmapArmijoCondition} applies to $\tilde{x}^\mathrm{R}$ with $\ball[\ushort{x}, \bar{\rho}(\ushort{x})]$. The inclusion holds because, for all $z \in \ball[\hat{x}, \bar{\alpha} \s(\hat{x}, f, C)]$,
\begin{equation*}
\norm{z-\ushort{x}}
\le \norm{z-\hat{x}} + \norm{\hat{x}-\ushort{x}}
\le \bar{\alpha} \s(\hat{x}, f, C) + 2 \varepsilon(\ushort{x})
\le \bar{\alpha} \norm{\nabla f(\hat{x})} + \Delta
\le \bar{\rho}(\ushort{x}),
\end{equation*}
where the last inequality follows from \eqref{eq:RFDRPolakInequalityNormGradient}. Therefore, for all $y \in \hyperref[algo:RFDRmap]{\rfdr}(x; \mathcal{E}, C, f, \ushort{\alpha}, \bar{\alpha}, \beta, c, \Delta)$,
\begin{align*}
f(y)
&\le f(\tilde{x}^\mathrm{R})\\
&\le f(\hat{x}) - c \mu^2 \s(\hat{x}; f, C)^2 \min\left\{\ushort{\alpha}, 2\beta\frac{1-c}{\lip_{\ball[\ushort{x}, \bar{\rho}(\ushort{x})]}(\nabla f)}\right\}\\
&\le f(\hat{x}) - c \mu^2 \left(\inf_{\substack{\check{x} \in C \setminus S_p \\ z \in \mathcal{E} \setminus \{0\}}} \frac{\norm{\proj{\tancone{C}{\check{x}}}{z}}}{\norm{z}}\right)^2 \norm{\nabla f(\hat{x})}^2 \min\left\{\ushort{\alpha}, 2\beta\frac{1-c}{\lip_{\ball[\ushort{x}, \bar{\rho}(\ushort{x})]}(\nabla f)}\right\}\\
&\le f(\hat{x}) - 3 \delta(\ushort{x})\\
&\le f(x) - \delta(\ushort{x}),
\end{align*}
where the second inequality follows from Corollary~\ref{coro:RFDmapArmijoCondition}, the third from condition~2 of Assumption~\ref{assumption:StratificationRFDR}, the fourth from \eqref{eq:RFDRPolakInequalityNormGradient} and \eqref{eq:RFDRPolakDelta}, and the fifth from \eqref{eq:RFDRPolakCostContinuous}.
Let us now consider the case where $x \in C \setminus S_p$. Let $\tilde{x} \in \hyperref[algo:RFDmap]{\rfd}(x; \mathcal{E}, C, f, \ushort{\alpha}, \bar{\alpha}, \beta, c)$. Since $\ball[x, \bar{\alpha} \s(x, f, C)] \subseteq \ball[\ushort{x}, \bar{\rho}(\ushort{x})]$, Corollary~\ref{coro:RFDmapArmijoCondition} applies to $\tilde{x}$ with $\ball[\ushort{x}, \bar{\rho}(\ushort{x})]$. The inclusion holds because, for all $z \in \ball[x, \bar{\alpha} \s(x; f, C)]$,
\begin{equation*}
\norm{z-\ushort{x}}
\le \norm{z-x} + \norm{x-\ushort{x}}
\le \bar{\alpha} \s(x, f, C) + \varepsilon(\ushort{x})
< \bar{\alpha} \norm{\nabla f(x)} + \Delta
\le \bar{\rho}(\ushort{x}),
\end{equation*}
where the last inequality follows from \eqref{eq:RFDRPolakInequalityNormGradient}. Therefore, for all $y \in \hyperref[algo:RFDRmap]{\rfdr}(x; \mathcal{E}, C, f, \ushort{\alpha}, \bar{\alpha}, \beta, c, \Delta)$,
\begin{align*}
f(y)
&\le f(\tilde{x})\\
&\le f(x) - c \mu^2 \s(x; f, C)^2 \min\left\{\ushort{\alpha}, 2\beta\frac{1-c}{\lip_{\ball[\ushort{x}, \bar{\rho}(\ushort{x})]}(\nabla f)}\right\}\\
&\le f(x) - c \mu^2 \left(\inf_{\substack{\check{x} \in C \setminus S_p \\ z \in \mathcal{E} \setminus \{0\}}} \frac{\norm{\proj{\tancone{C}{\check{x}}}{z}}}{\norm{z}}\right)^2 \norm{\nabla f(x)}^2 \min\left\{\ushort{\alpha}, 2\beta\frac{1-c}{\lip_{\ball[\ushort{x}, \bar{\rho}(\ushort{x})]}(\nabla f)}\right\}\\
&\le f(x) - 3 \delta(\ushort{x})\\
&\le f(x) - \delta(\ushort{x}),
\end{align*}
where the second inequality follows from Corollary~\ref{coro:RFDmapArmijoCondition}, the third from condition~2 of Assumption~\ref{assumption:StratificationRFDR}, and the fourth from \eqref{eq:RFDRPolakInequalityNormGradient} and \eqref{eq:RFDRPolakDelta}.
\end{proof}

\subsubsection{The $\rfdr$ algorithm}
\label{subsubsec:RFDR}
The $\rfdr$ algorithm is defined as Algorithm~\ref{algo:RFDR}. It produces a sequence along which $f$ is strictly decreasing.

\begin{algorithm}[H]
\caption{$\rfdr$}
\label{algo:RFDR}
\begin{algorithmic}[1]
\Require
$(\mathcal{E}, C, f, \ushort{\alpha}, \bar{\alpha}, \beta, c, \Delta)$ where $\mathcal{E}$ is a Euclidean vector space, $C$ is a nonempty closed subset of $\mathcal{E}$ satisfying Assumption~\ref{assumption:StratificationRFDR}, $f : \mathcal{E} \to \R$ is differentiable with $\nabla f$ locally Lipschitz continuous, $0 < \ushort{\alpha} \le \bar{\alpha} < \infty$, $\beta, c \in (0,1)$, and $\Delta \in (0,\infty)$.
\Input
$x_0 \in C$.
\Output
a sequence in $C$.

\State
$i \gets 0$;
\While
{$\s(x_i; f, C) > 0$}
\State
Choose $x_{i+1} \in \hyperref[algo:RFDRmap]{\rfdr}(x_i; \mathcal{E}, C, f, \ushort{\alpha}, \bar{\alpha}, \beta, c, \Delta)$;
\State
$i \gets i+1$;
\EndWhile
\end{algorithmic}
\end{algorithm}

Theorem~\ref{thm:RFDRPolakConvergence} states that $\rfdr$ accumulates at stationary points of~\eqref{eq:OptiProblem} and is thus apocalypse-free. However, it does not state that an accumulation point necessarily exists.

\begin{theorem}
\label{thm:RFDRPolakConvergence}
Consider a sequence constructed by $\rfdr$ (Algorithm~\ref{algo:RFDR}). If this sequence is finite, then its last element is stationary for~\eqref{eq:OptiProblem}, i.e., is a zero of the stationarity measure $\s(\cdot; f, C)$ defined in~\eqref{eq:StationarityMeasure}. If it is infinite, then all of its accumulation points are stationary for~\eqref{eq:OptiProblem}.
\end{theorem}

\begin{proof}
The proof is the same as the one of Theorem~\ref{thm:P2GDRPolakConvergence} except that Proposition~\ref{prop:RFDRmapPolak} replaces Proposition~\ref{prop:P2GDRmapPolak}.
\end{proof}

Corollary~\ref{coro:RFDRPolakConvergence} considers a sequence $(x_i)_{i \in \N}$ produced by $\rfdr$. It guarantees that, if $C$ has no serendipitous point, which is notably the case of $\R_{\le r}^{m \times n}$ (Proposition~\ref{prop:RealDeterminantalVarietyNoSerendipitousPoint}), and the sublevel set $\{x \in C \mid f(x) \le f(x_0)\}$ is bounded, then $\lim_{i \to \infty} \s(x_i; f, C) = 0$, and all accumulation points, of which there exists at least one, have the same image by $f$.

\begin{corollary}
\label{coro:RFDRPolakConvergence}
Let $(x_i)_{i \in \N}$ be a sequence produced by $\rfdr$ (Algorithm~\ref{algo:RFDR}).
The sequence has at least one accumulation point if and only if $\liminf_{i \to \infty} \norm{x_i} < \infty$. If $C$ has no serendipitous point, then, for every convergent subsequence $(x_{i_k})_{k \in \N}$, $\lim_{k \to \infty} \s(x_{i_k}; f, C) = 0$.
If, moreover, $(x_i)_{i \in \N}$ is bounded, which is the case notably if the sublevel set $\{x \in C \mid f(x) \le f(x_0)\}$ is bounded, then $\lim_{i \to \infty} \s(x_i; f, C) = 0$, and all accumulation points have the same image by $f$.
\end{corollary}

\begin{proof}
The proof is the same as the one of Corollary~\ref{coro:P2GDRPolakConvergence}.
\end{proof}

Corollary~\ref{coro:RFDRPolakConvergence} shows that the stopping criterion defined by \eqref{eq:StoppingCriterion} is always eventually satisfied by $\rfdr$ if $C$ has no serendipitous point and the generated sequence has an accumulation point. Indeed, if $C$ has no serendipitous point and $\rfdr$ produces a sequence $(x_i)_{i \in \N}$ that has an accumulation point, i.e., that does not diverge to infinity, then, for every $\varepsilon \in (0, \infty)$, the set $\{i \in \N \mid \s(x_i; f, C) \le \varepsilon\}$ is nonempty and thus its minimum $i_\varepsilon$ exists.

\subsection{The set of sparse vectors satisfies Assumption~\ref{assumption:StratificationRFDR}}
\label{subsec:SparseVectors}
In this section, we prove that $\R_{\le s}^n$ satisfies Assumption~\ref{assumption:StratificationRFDR}.
Problem~\eqref{eq:OptiProblem} with $C = \R_{\le s}^n$ appears in sparse signal approximation which has several applications in signal processing such as compressed sensing \cite{BlumensathDavies2008, BlumensathDavies2009, BlumensathDavies2010, BlanchardTannerWei}.

In Section~\ref{subsubsec:StratificationSparseVectors}, we show how to project onto $\sparse{n}{s}$ (Proposition~\ref{prop:ProjectionSparseVectors}) and prove that $\sparse{n}{s}$ admits a stratification satisfying condition~1 of Assumption~\ref{assumption:Stratification} (Proposition~\ref{prop:StratificationSparseVectorsCondition1MainAssumption}). In Section~\ref{subsubsec:TangentConeSparseVectors}, we review the tangent cone to $\sparse{n}{s}$ (Proposition~\ref{prop:TangentConeSparseVectors}) and prove that $\sparse{n}{s}$ satisfies Assumption~\ref{assumption:StratificationRFDR} (Proposition~\ref{prop:SparseVectorsStrongAssumption}). In Section~\ref{subsubsec:NormalConeSparseVectors}, we deduce the regular normal cone, the normal cone, and the Clarke normal cone to $\sparse{n}{s}$, and prove that the set of apocalyptic points of $\sparse{n}{s}$ is $\StrictSparsity{n}{s}$ and that $\sparse{n}{s}$ has no serendipitous point (Proposition~\ref{prop:SparseVectorsApocalypticSerendipitousPoints}). Finally, in Section~\ref{subsubsec:SparseVectorsP2GDapocalypseExample}, we present an example of $\ppgd$ following an apocalypse on $\sparse{n}{s}$.

\subsubsection{Stratification of the set of sparse vectors}
\label{subsubsec:StratificationSparseVectors}
In this section, we prove that $\sparse{n}{s}$ admits a stratification satisfying Assumption~\ref{assumption:StratificationRFDR}.
The number of nonzero components stratifies $\sparse{n}{s}$:
\begin{equation*}
\sparse{n}{s} = \bigcup_{i=0}^s \FixedSparsity{n}{i}
\end{equation*}
where, for every $i \in \{0, \dots, s\}$,
\begin{equation*}
\FixedSparsity{n}{i} = \{x \in \R^n \mid |\supp(x)| = i\}
\end{equation*}
is the subset of $\R^n$ containing the points having exactly $i$ nonzero component(s).

Proposition~\ref{prop:ProjectionSparseVectors} shows how to project onto $\sparse{n}{s}$ and is used in the proof that $\sparse{n}{s}$ satisfies condition~1(c) of Assumption~\ref{assumption:Stratification}.

\begin{proposition}[{projection onto the set of sparse vectors \cite[Proposition~3.6]{BauschkeLukePhanWang2014}}]
\label{prop:ProjectionSparseVectors}
For every $x \in \R^n \setminus \sparse{n}{s}$, $\proj{\sparse{n}{s}}{x}$ is the set of all possible outputs of Algorithm~\ref{algo:ProjectionSparseVectors}, and $\dist(x, \sparse{n}{s})$ is the sum of the smallest $|\supp(x)|-s$ absolute values of components of $x$.
\end{proposition}

\begin{algorithm}[H]
\caption{Projection onto the set of sparse vectors}
\label{algo:ProjectionSparseVectors}
\begin{algorithmic}[1]
\Require
$(n, s)$ where $n$ and $s$ are positive integers such that $s < n$.
\Input
$x \in \R^n$.
\Output
$y \in \proj{\sparse{n}{s}}{x}$.

\State
$y \gets x$;
\While
{$|\supp(y)| > s$}
	\State
	Choose $i \in \argmin_{j \in \supp(y)} |y_j|$;
	\State
	$y_i \gets 0$;
\EndWhile
\State
Return $y$.
\end{algorithmic}
\end{algorithm}

Based on Proposition~\ref{prop:ProjectionSparseVectors}, we now prove that $\sparse{n}{s}$ satisfies condition~1 of Assumption~\ref{assumption:Stratification}.

\begin{proposition}[stratification of the set of sparse vectors]
\label{prop:StratificationSparseVectorsCondition1MainAssumption}
The stratification $\{\FixedSparsity{n}{0}, \dots, \FixedSparsity{n}{s}\}$ of $\sparse{n}{s}$ satisfies condition~1 of Assumption~\ref{assumption:Stratification}.
\end{proposition}

\begin{proof}
Using the submanifold property \cite[Proposition~3.3.2]{AbsilMahonySepulchre}, we first prove that, for every $i \in \{0, \dots, s\}$, $\FixedSparsity{n}{i}$ is an $i$-dimensional embedded submanifold of $\R^n$.
For $\FixedSparsity{n}{0} = \{0\}^n$, we take $U := \R^n$ and $\varphi : \R^n \to \R^n : x \mapsto x$, and we have
\begin{equation*}
\{x \in U \mid \varphi(x) \in \{0\}^n\}
= \{0\}^n
= \FixedSparsity{n}{0} \cap U.
\end{equation*}
Let $\ushort{x} \in \FixedSparsity{n}{i}$ with $i \in \{1, \dots, s\}$. For $U := \ball(\ushort{x}, \dist(\ushort{x}, \FixedSparsity{n}{i-1}))$ and
\begin{equation*}
\varphi : U \to \R^n : x \mapsto ((x_j)_{j \in \supp(\ushort{x})}, (x_j)_{j \in \{1, \dots, n\} \setminus \supp(\ushort{x})}),
\end{equation*}
we have
\begin{equation*}
\{x \in U \mid \varphi(x) \in \R^i \times \{0\}^{n-i}\}
= \{x \in U \mid x_j = 0 \, \forall j \in \{1, \dots, n\} \setminus \supp(\ushort{x})\}
= \FixedSparsity{n}{i} \cap U.
\end{equation*}
Thus, condition~1(a) is satisfied. By Proposition~\ref{prop:ProjectionSparseVectors}, condition~1(c) is satisfied too. To establish condition~1(b), it suffices to prove that, for all $i, j \in \{0, \dots, s\}$, $\FixedSparsity{n}{j} \cap \overline{\FixedSparsity{n}{i}} = \emptyset$ if $j > i$ and $\FixedSparsity{n}{j} \subseteq \overline{\FixedSparsity{n}{i}}$ if $j \le i$.
Let $i, j \in \{0, \dots, s\}$. If $j > i$, then, by Proposition~\ref{prop:ProjectionSparseVectors}, for all $x \in \FixedSparsity{n}{j}$, $\ball(x, \dist(x, \FixedSparsity{n}{j-1})) \cap \FixedSparsity{n}{i} = \emptyset$ and thus $x \not\in \overline{\FixedSparsity{n}{i}}$. If $j \le i$, then, for all $x \in \FixedSparsity{n}{j}$ and all $\varepsilon \in (0, \infty)$, $\ball[x, \varepsilon] \cap \FixedSparsity{n}{i} \ne \emptyset$. This is clear if $j = i$. Let us prove it in the case where $j < i$. Let $x \in \FixedSparsity{n}{j}$ and $\varepsilon \in (0, \infty)$. Let $I(x) \subseteq \{1, \dots, n\} \setminus \supp(x)$ such that $|I(x)| = i-j$. Define $y \in \R^n$ by $y_k := \frac{\varepsilon}{\sqrt{n}}$ if $k \in I(x)$ and $y_k := 0$ otherwise. Then, $x+y \in \ball[x, \varepsilon] \cap \FixedSparsity{n}{i}$.
\end{proof}

\subsubsection{Tangent cone to the set of sparse vectors}
\label{subsubsec:TangentConeSparseVectors}
Proposition~\ref{prop:TangentConeSparseVectors} gives an explicit description of the tangent cone to $\sparse{n}{s}$ and show how to project onto it.

\begin{proposition}[tangent cone to the set of sparse vectors]
\label{prop:TangentConeSparseVectors}
For every $x \in \sparse{n}{s}$,
\begin{equation*}
\tancone{\sparse{n}{s}}{x} = \left\{v \in \R^n \mid |\supp(x) \cup \supp(v)| \le s\right\},
\end{equation*}
$\sparse{n}{s}$ is geometrically derivable at $x$, and, for every $v \in \R^n$, $\proj{\tancone{\sparse{n}{s}}{x}}{v}$ is the set of all possible outputs of Algorithm~\ref{algo:ProjectionTangentConeSparseVectors}.
\end{proposition}

\begin{proof}
The description of $\tancone{\sparse{n}{s}}{x}$ is given in \cite[Theorem~3.15]{BauschkeLukePhanWang2014} and the projection onto $\tancone{\sparse{n}{s}}{x}$ follows. Thus, we only prove that $\sparse{n}{s}$ is geometrically derivable.
Let $x \in \sparse{n}{s}$ and $v \in \R^n$ such that $|\supp(x) \cup \supp(v)| \le s$. Then, for all $t \in (0, \infty)$, since $\supp(tv) = \supp(v)$ and, by~\eqref{eq:SubadditivitySupport}, $\supp(x+tv) \subseteq \supp(x) \cup \supp(tv)$, it holds that $|\supp(x+tv)| \le s$, i.e., $x+tv \in \sparse{n}{s}$, and hence $\dist(x+tv, \sparse{n}{s})/t = 0$. Thus, $\lim_{t \searrow 0} \dist(x+tv, \sparse{n}{s})/t = 0$ and it follows that $v$ is geometrically derivable.
\end{proof}

\begin{algorithm}[H]
\caption{Projection onto the tangent cone to the set of sparse vectors}
\label{algo:ProjectionTangentConeSparseVectors}
\begin{algorithmic}[1]
\Require
$(n, s, x)$ where $n$ and $s$ are positive integers such that $s < n$, and $x \in \sparse{n}{s}$.
\Input
$v \in \R^n$.
\Output
$w \in \proj{\tancone{\sparse{n}{s}}{x}}{v}$.

\State
$w \gets v$;
\While
{$|\supp(x) \cup \supp(w)| > s$}
	\State
	Choose $i \in \argmin_{j \in \supp(w) \setminus \supp(x)} |w_j|$;
	\State
	$w_i \gets 0$;
\EndWhile
\State
Return $w$.
\end{algorithmic}
\end{algorithm}

Proposition~\ref{prop:SparseVectorsStrongAssumption} shows that $\sparse{n}{s}$ satisfies Assumption~\ref{assumption:StratificationRFDR}.

\begin{proposition}
\label{prop:SparseVectorsStrongAssumption}
The set $\sparse{n}{s}$ satisfies Assumption~\ref{assumption:StratificationRFDR}:
\begin{enumerate}
\item for all $x \in \sparse{n}{s}$ and all $v \in \tancone{\sparse{n}{s}}{x}$, $x+v \in \sparse{n}{s}$;
\item for all $x \in \StrictSparsity{n}{s}$ and all $v \in \R^n$, $\norm{\proj{\tancone{\sparse{n}{s}}{x}}{v}} \ge \norm{v}_\infty \ge \frac{1}{\sqrt{n}} \norm{v}$.
\end{enumerate}
\end{proposition}

\begin{proof}
The first property follows from~\eqref{eq:SubadditivitySupport} and Proposition~\ref{prop:TangentConeSparseVectors}. Let $x \in \StrictSparsity{n}{s}$ and $v \in \R^n \setminus \tancone{\sparse{n}{s}}{x}$. Then, Algorithm~\ref{algo:ProjectionTangentConeSparseVectors} produces $w \in \proj{\tancone{\sparse{n}{s}}{x}}{v}$ such that $\max_{i \in \supp(w) \setminus \supp(x)} |w_i| = \max_{i \in \supp(v) \setminus \supp(x)} |v_i|$ and, for all $i \in \{1, \dots, n\} \setminus (\supp(v) \setminus \supp(x))$, $w_i = v_i$. Hence, $\norm{w}_\infty = \norm{v}_\infty$, and the result follows.
\end{proof}

Proposition~\ref{prop:ContinuityTangentConeStratumSparseVectors} states that $\sparse{n}{s}$ satisfies condition~3 of Assumption~\ref{assumption:Stratification}.

\begin{proposition}
\label{prop:ContinuityTangentConeStratumSparseVectors}
For every $i \in \{0, \dots, s\}$, the correspondence $\tancone{\sparse{n}{s}}{\cdot}$ is continuous at every $x \in \FixedSparsity{n}{i}$ relative to $\FixedSparsity{n}{i}$.
\end{proposition}

\begin{proof}
The result is clear if $i = 0$ since $\FixedSparsity{n}{0} = \{0\}^n$. Let us therefore consider $i \in \{1, \dots, s\}$. We have to prove that, for every sequence $(x^j)_{j \in \N}$ in $\FixedSparsity{n}{i}$ converging to $x \in \FixedSparsity{n}{i}$, it holds that
\begin{equation*}
\outlim_{j \to \infty} \tancone{\sparse{n}{s}}{x^j}
\subseteq \tancone{\sparse{n}{s}}{x}
\subseteq \inlim_{j \to \infty} \tancone{\sparse{n}{s}}{x^j}.
\end{equation*}
Let $x \in \FixedSparsity{n}{i}$. Then, for all $y \in \ball(x, \dist(x, \FixedSparsity{n}{i-1})) \cap \FixedSparsity{n}{i}$, $\supp(y) = \supp(x)$ and thus, by Proposition~\ref{prop:TangentConeSparseVectors}, $\tancone{\sparse{n}{s}}{y} = \tancone{\sparse{n}{s}}{x}$. Thus, the result follows from the fact that a sequence $(x^j)_{j \in \N}$ in $\FixedSparsity{n}{i}$ converging to $x$ contains finitely many elements in $\FixedSparsity{n}{i} \setminus \ball(x, \dist(x, \FixedSparsity{n}{i-1}))$.
\end{proof}

\subsubsection{Normal cone to the set of sparse vectors}
\label{subsubsec:NormalConeSparseVectors}
In Proposition~\ref{prop:RegularNormalConeSparseVectors}, we deduce the regular normal cone to $\sparse{n}{s}$ from the description of the tangent cone to $\sparse{n}{s}$ given in Proposition~\ref{prop:TangentConeSparseVectors}.

\begin{proposition}[regular normal cone to the set of sparse vectors]
\label{prop:RegularNormalConeSparseVectors}
For all $x \in \sparse{n}{s}$,
\begin{equation*}
\regnorcone{\sparse{n}{s}}{x} = \left\{\begin{array}{ll}
\{w \in \R^n \mid \supp(w) \subseteq \{1, \dots, n\} \setminus \supp(x)\} & \text{if } x \in \FixedSparsity{n}{s},\\
\{0\}^n & \text{if } x \in \StrictSparsity{n}{s}.
\end{array}\right.
\end{equation*}
\end{proposition}

\begin{proof}
The proof is based on Proposition~\ref{prop:TangentConeSparseVectors}. Let $x \in \sparse{n}{s}$. By \eqref{eq:RegularNormalCone} and because $\supp(-x) = \supp(x)$, we have
\begin{equation*}
\regnorcone{\sparse{n}{s}}{x}
= \left\{w \in \R^n \mid \ip{w}{v} \le 0 \; \forall v \in \tancone{\sparse{n}{s}}{x}\right\}
= \left\{w \in \R^n \mid \ip{w}{v} = 0 \; \forall v \in \tancone{\sparse{n}{s}}{x}\right\}.
\end{equation*}
Assume that $x \in \StrictSparsity{n}{s}$. Then, for all $i \in \{1, \dots, n\}$, $v := (\delta_{i, j})_{j \in \{1, \dots, n\}} \in \tancone{\sparse{n}{s}}{x}$ and, for all $w \in \R^n$, $\ip{w}{v} = w_i$. Thus, $\regnorcone{\sparse{n}{s}}{x} = \{0\}^n$.
Assume now that $x \in \FixedSparsity{n}{s}$. Then, for all $v \in \tancone{\sparse{n}{s}}{x}$ and all $i \in \{1, \dots, n\} \setminus \supp(x)$, $v_i = 0$. Thus, for all $w \in \R^n$ and all $v \in \tancone{\sparse{n}{s}}{x}$, $\ip{w}{v} = \sum_{i \in \supp(x)} w_i v_i$. Since, for all $i \in \supp(x)$, $v := (\delta_{i, j})_{j \in \{1, \dots, n\}} \in \tancone{\sparse{n}{s}}{x}$ and, for all $w \in \R^n$, $\ip{w}{v} = w_i$, we have $\regnorcone{\sparse{n}{s}}{x} \subseteq \{w \in \R^n \mid \supp(w) \subseteq \{1, \dots, n\} \setminus \supp(x)\}$. The converse inclusion also holds, and the result follows.
\end{proof}

\begin{proposition}[{normal cone to the set of sparse vectors \cite[Theorem~3.9]{BauschkeLukePhanWang2014}}]
\label{prop:NormalConeSparseVectors}
For all $x \in \sparse{n}{s}$,
\begin{equation*}
\norcone{\sparse{n}{s}}{x} = \{w \in \sparse{n}{n-s} \mid \supp(w) \subseteq \{1, \dots, n\} \setminus \supp(x)\}.
\end{equation*}
In particular, $\sparse{n}{s}$ is not Clarke regular on $\StrictSparsity{n}{s}$.
\end{proposition}

\begin{proof}
We provide an alternative proof to the one of \cite[Theorem~3.9]{BauschkeLukePhanWang2014}. This argument is based on the definition \eqref{eq:NormalCone} of the normal cone and is used again in the proof of Proposition~\ref{prop:SparseVectorsApocalypticSerendipitousPoints}.

By \cite[Example~6.8]{RockafellarWets}, the result follows from Proposition~\ref{prop:RegularNormalConeSparseVectors} if $x \in \FixedSparsity{n}{s}$. Let $x \in \StrictSparsity{n}{s}$.
We first establish the inclusion $\subseteq$. Let $(x^k)_{k \in \N}$ be a sequence in $\sparse{n}{s}$ converging to $x$. We have to prove that
\begin{equation*}
\outlim_{k \to \infty} \regnorcone{\sparse{n}{s}}{x^k} \subseteq \{w \in \sparse{n}{n-s} \mid \supp(w) \subseteq \{1, \dots, n\} \setminus \supp(x)\}.
\end{equation*}
If $(x^k)_{k \in \N}$ contains finitely many elements in $\FixedSparsity{n}{s}$, then $\outlim_{k \to \infty} \regnorcone{\sparse{n}{s}}{x^k} = \{0\}^n$. Indeed, for all $k \in \N$ sufficiently large, $x^k \in \StrictSparsity{n}{s}$ and thus, by Proposition~\ref{prop:RegularNormalConeSparseVectors}, $\regnorcone{\sparse{n}{s}}{x^k} = \{0\}^n$.
Assume that $(x^k)_{k \in \N}$ contains infinitely many elements in $\FixedSparsity{n}{s}$ and let $w \in \outlim_{k \to \infty} \regnorcone{\sparse{n}{s}}{x^k}$. Then, $w$ is an accumulation point of a sequence $(w^k)_{k \in \N}$ such that, for all $k \in \N$, $w^k \in \regnorcone{\sparse{n}{s}}{x^k}$. Moreover, there exists a strictly increasing sequence $(k_l)_{l \in \N}$ in $\N$ such that $(w^{k_l})_{l \in \N}$ converges to $w$ and, for all $l \in \N$, $x^{k_l} \in \FixedSparsity{n}{s}$, $\supp(x) \subseteq \supp(x^{k_l})$, and $\supp(w) \subseteq \supp(w^{k_l})$. Thus, for all $l \in \N$, since $w^{k_l} \in \regnorcone{\sparse{n}{s}}{x^{k_l}}$, it holds that $\supp(w^{k_l}) \subseteq \{1, \dots, n\} \setminus \supp(x^{k_l})$. Therefore, $\supp(w) \subseteq \supp(w^{k_0}) \subseteq \{1, \dots, n\} \setminus \supp(x^{k_0}) \subseteq \{1, \dots, n\} \setminus \supp(x)$ and $|\supp(w)| \le n-s$.

We now establish the inclusion $\supseteq$. Let $w \in \sparse{n}{n-s}$ such that $\supp(w) \subseteq \{1, \dots, n\} \setminus \supp(x)$. Let $I(x) \subseteq \{1, \dots, n\} \setminus (\supp(w) \cup \supp(x))$ such that $|I(x)| = s-|\supp(x)|$; this is possible since $|\{1, \dots, n\} \setminus (\supp(w) \cup \supp(x))| = n-|\supp(w)|-|\supp(x)| \ge s-|\supp(x)|$. For all $k \in \N$ and all $i \in \{1, \dots, n\}$, let
\begin{equation*}
x_i^k := \left\{\begin{array}{ll}
\frac{1}{k+1} & \text{if } i \in I(x),\\
x_i & \text{otherwise}.
\end{array}\right.
\end{equation*}
Then, for all $k \in \N$, $\supp(x^k) = \supp(x) \cup I(x)$, thus $\supp(w) \subseteq \{1, \dots, n\} \setminus \supp(x^k)$, and therefore $w \in \regnorcone{\sparse{n}{s}}{x^k}$. It follows that $w \in \outlim_{k \to \infty} \regnorcone{\sparse{n}{s}}{x^k}$.
\end{proof}

\begin{corollary}[Clarke normal cone to the set of sparse vectors ]
\label{coro:ClarkeNormalConeSparseVectors}
For all $x \in \sparse{n}{s}$,
\begin{equation*}
\connorcone{\sparse{n}{s}}{x} = \{w \in \R^n \mid \supp(w) \subseteq \{1, \dots, n\} \setminus \supp(x)\}.
\end{equation*}
\end{corollary}

By Proposition~\ref{prop:NormalConeSparseVectors}, $\sparse{n}{s}$ is not Clarke regular on $\StrictSparsity{n}{s}$. Proposition~\ref{prop:SparseVectorsApocalypticSerendipitousPoints} states that every point of $\StrictSparsity{n}{s}$ is apocalyptic, which is a stronger result by \cite[Corollary~2.15]{LevinKileelBoumal2022}.

\begin{proposition}
\label{prop:SparseVectorsApocalypticSerendipitousPoints}
The set of apocalyptic points of $\sparse{n}{s}$ is $\StrictSparsity{n}{s}$, and $\sparse{n}{s}$ has no serendipitous point.
\end{proposition}

\begin{proof}
We use Proposition~\ref{prop:CharacterizationApocalypticSerendipitousPoint}.
Let $x \in \FixedSparsity{n}{s}$ and $(x^k)_{k \in \N}$ be a sequence in $\sparse{n}{s}$ converging to $x$. Since $(x^k)_{k \in \N}$ contains finitely many elements not in $\ball(x, \dist(x, \FixedSparsity{n}{s-1})) \subseteq \FixedSparsity{n}{s}$, we can assume that $(x^k)_{k \in \N}$ is in $\FixedSparsity{n}{s}$. Therefore, by Proposition~\ref{prop:ContinuityTangentConeStratumSparseVectors}, $\outlim_{k \to \infty} \tancone{\sparse{n}{s}}{x^k} = \tancone{\sparse{n}{s}}{x}$ and thus $\big(\outlim_{k \to \infty} \tancone{\sparse{n}{s}}{x^k}\big)^* = \regnorcone{\sparse{n}{s}}{x}$. It follows that $x$ is neither apocalyptic nor serendipitous.

Let $x \in \FixedSparsity{n}{i}$ with $i \in \{0, \dots, s-1\}$.
By Proposition~\ref{prop:RegularNormalConeSparseVectors}, $\regnorcone{\sparse{n}{s}}{x} = \{0\}^n$ and thus $x$ is not serendipitous.
Let us prove that $x$ is apocalyptic. Let $I(x) \subseteq \{1, \dots, n\} \setminus \supp(x)$ such that $|I(x)| = s-|\supp(x)|$. For all $k \in \N$, define
\begin{equation*}
x_j^k := \left\{\begin{array}{ll}
\frac{1}{k+1} & \text{if } j \in I(x),\\
x_j & \text{if } j \in \{1, \dots, n\} \setminus I(x).
\end{array}\right.
\end{equation*}
Then, $(x^k)_{k \in \N}$ is in $\FixedSparsity{n}{s}$ and converges to $x$. Define $\varepsilon := 1$ if $i = 0$ and $\varepsilon := \dist(x, S_{i-1})$ otherwise. There exists $K \in \N$ such that, for every integer $k \ge K$, $x^k \in \ball(x, \varepsilon)$, thus $\supp(x^k) = \supp(x^0)$ and
\begin{equation*}
\tancone{\sparse{n}{s}}{x^k} = \tancone{\sparse{n}{s}}{x^0}.
\end{equation*}
Thus,
\begin{equation*}
\setlim_{k \to \infty} \tancone{\sparse{n}{s}}{x^k} = \tancone{\sparse{n}{s}}{x^0}.
\end{equation*}
Therefore,
\begin{equation*}
\Big(\setlim_{k \to \infty} \tancone{\sparse{n}{s}}{x^k}\Big)^*
= \regnorcone{\sparse{n}{s}}{x^0}
\ne \emptyset
= \regnorcone{\sparse{n}{s}}{x}.
\end{equation*}
It follows that $x$ is apocalyptic.
\end{proof}

\subsubsection{$\ppgd$ following an apocalypse on $\sparse{n}{s}$}
\label{subsubsec:SparseVectorsP2GDapocalypseExample}
Proposition~\ref{prop:SparseVectorsP2GDapocalypseExample} presents an example of $\ppgd$ following an apocalypse on $\sparse{n}{s}$.

\begin{proposition}
\label{prop:SparseVectorsP2GDapocalypseExample}
~
\begin{enumerate}
\item Define:
\begin{itemize}
\item $x^* \in \FixedSparsity{n}{s}$ by $x_i^* := 1$ if $i \in \{1, \dots, s\}$ and $x_i^* := 0$ if $i \in \{s+1, \dots, n\}$;
\item $x^0 := (\delta_{i, n})_{i \in \{1, \dots, n\}}$;
\item for all $j \in \{1, \dots, s\}$, $x^j := x^*-(\delta_{i, j})_{i \in \{1, \dots, n\}}$;
\item for all $k \in \N$ and all $j \in \{1, \dots, s\}$, $x^{k, j} := (1-2^{-k}) x^j + 2^{-k} x^0$.
\end{itemize}
Then, for every $j \in \{1, \dots, s\}$, $(x^{k, j})_{k \in \N \setminus \{0\}}$ is in $\FixedSparsity{n}{s}$ and converges to $x^j$.
\item For $f : \R^n \to \R : x \mapsto \frac{1}{4} \norm{x-x^*}^2$, $x^0$, $\bar{\alpha} := \ushort{\alpha} := 1$, any $\beta \in (0, 1)$, and any $c \in (0, \frac{3}{4}]$, the set of sequences that $\ppgd$ can produce is $\{(x^{k, j})_{k \in \N} \mid j \in \{1, \dots, s\}\}$. Moreover, for all $k \in \N$ and all $j \in \{1, \dots, s\}$, $\s(x^{k, j}; f, \sparse{n}{s}) = 2^{-k-1} \sqrt{s}$ and thus, since $\lim_{k \to \infty} \s(x^{k, j}; f, \sparse{n}{s}) = 0$ and $\s(x^j; f, \sparse{n}{s}) = \frac{1}{2}$, it follows that $(x^j, (x^{k, j})_{k \in \N}, f)$ is an apocalypse.
\end{enumerate}
\end{proposition}

\begin{proof}
The first part is clear. We therefore prove the second part.
For all $x \in \R^n$, $\nabla f(x) = \frac{1}{2}(x-x^*)$. Thus, by Proposition~\ref{prop:TangentConeSparseVectors},
\begin{equation*}
\proj{\tancone{\sparse{n}{s}}{x^0}}{-\nabla f(x^0)} = \left\{\frac{1}{2}(x^j-x^0) \mid j \in \{1, \dots, s\}\right\}
\end{equation*}
and, for all $k \in \N \setminus \{0\}$ and all $j \in \{1, \dots, s\}$,
\begin{equation*}
\proj{\tancone{\sparse{n}{s}}{x^{k, j}}}{-\nabla f(x^{k, j})} = 2^{-k-1}(x^j-x^0).
\end{equation*}
Furthermore, simple computations show that, for all $k \in \N$ and all $j \in \{1, \dots, s\}$,
\begin{align*}
x^{k+1, j} = x^{k, j} + 2^{-k-1}(x^j-x^0),&&
f(x^{k+1, j}) \le f(x^{k, j}) - c \s(x^{k, j}; f, \sparse{n}{s})^2.
\end{align*}
It follows that, at the first iteration, the only choice to make is to choose one of the $s$ elements of $\proj{\tancone{\sparse{n}{s}}{x^0}}{-\nabla f(x^0)}$ as a search direction. This choice defines $j \in \{1, \dots, s\}$. For all subsequent iterations, the $\ppgd$ map produces a singleton.
\end{proof}

\subsection{On the convergence analysis of a Riemannian rank-adaptive method}
\label{subsec:ConvergenceAnalysisRRAM}
Proposition~\ref{prop:NoRankRelatedRetractionLocallyRadiallyLipschitzContinuouslyDifferentiable} shows that the convergence analysis of the Riemannian rank-adaptive method given in \cite[Algorithm~3]{ZhouEtAl2016} does not apply to all cost functions considered in \cite{ZhouEtAl2016}.
As the lower bounds in Propositions~\ref{prop:GlobalSecondOrderUpperBoundDistanceToRealDeterminantalVarietyFromTangentLine} and \ref{prop:GlobalSecondOrderUpperBoundDistanceToPSDconeBoundedRankFromTangentLine}, this result is related to the curvature of the fixed-rank manifold~\eqref{eq:RealFixedRankManifold}.

\begin{proposition}
\label{prop:NoRankRelatedRetractionLocallyRadiallyLipschitzContinuouslyDifferentiable}
There exists no \emph{rank-related retraction} \cite[Definition~2]{ZhouEtAl2016} such that \cite[Assumption~6]{ZhouEtAl2016} holds for every analytic cost function $f : \R^{m \times n} \to \R$.
\end{proposition}

\begin{proof}
Let $\tilde{R} : \R^{2 \times 2} \times \R^{2 \times 2} \to \R^{2 \times 2}$ be a rank-related retraction \cite[Definition~2]{ZhouEtAl2016}, where we have identified the tangent bundle of $\R^{2 \times 2}$ with $\R^{2 \times 2} \times \R^{2 \times 2}$. Let us prove that, for the determinantal variety $\R_{\le 1}^{2 \times 2}$ and the cost function $f : \R^{2 \times 2} \to \R : X \mapsto X_{2,2}$, the lifted function $f \circ \tilde{R}$ does not satisfy \cite[Assumption~6]{ZhouEtAl2016}.
For $X_* := 0_{2 \times 2}$, let $\delta_{X_*}$ and $\mathcal{U}_*$ be respectively the positive real number and the neighborhood of $X_*$ in $\R_{\le 1}^{2 \times 2}$ that correspond to $X_*$ in \cite[Definition~2]{ZhouEtAl2016}. Let $\beta_\mathrm{RL} \in (0,\infty)$, $\delta_\mathrm{RL} \in (0,\delta_{X_*})$, and $\mathcal{U} \subseteq \mathcal{U}_*$ be a neighborhood of $X_*$ in $\R_{\le 1}^{2 \times 2}$.
Let $\xi_1, \xi_2 \in (0,1)$ be such that $\xi_1^2 + 2\xi_2^2 = 1$. Let $\sigma \in (0,\frac{\xi_2^2}{3\beta_\mathrm{RL}})$ be such that $X := \diag(\sigma,0) \in \mathcal{U}$. Let $\tilde{R}_X : \R^{2 \times 2} \to \R^{2 \times 2} : \xi \mapsto \tilde{R}(X,\xi)$. Then, $\xi := \left[\begin{smallmatrix} -\xi_1 & \xi_2 \\ \xi_2 & 0 \end{smallmatrix}\right] \in \tancone{\R_{\le 1}^{2 \times 2}}{X}$, $\norm{\xi} = 1$, the \emph{update-rank} \cite[Definition~1]{ZhouEtAl2016} of $\xi$ is $1$, and the following properties hold: $\tilde{R}_X(0_{2 \times 2}) = X$, the function $[0,\delta_{X_*}) \to \R^{2 \times 2} : t \mapsto \tilde{R}_X(t\xi)$ is continuously differentiable, its image is contained in $\R_{\le 1}^{2 \times 2}$, and $\frac{\dd}{\dd t} \tilde{R}_X(t\xi)|_{t=0} = \xi$.
Since the function $[0,\delta_{X_*}) \to \R^{2 \times 2} : t \mapsto \tilde{R}_X(t\xi)$ is continuous and $(\tilde{R}_X(t\xi))_{1,1} = \sigma$, there exists $\tilde{\delta} \in (0,\delta_\mathrm{RL}]$ such that, for all $t \in [0,\tilde{\delta}]$, $(\tilde{R}_X(t\xi))_{1,1} \in [\frac{1}{2}\sigma, \frac{3}{2}\sigma]$. Therefore, for all $t \in [0,\tilde{\delta}]$, since $\det \tilde{R}_X(t\xi) = 0$, $\tilde{R}_X(t\xi) = \left[\begin{smallmatrix} x(t) & y(t) \\ z(t) & \frac{y(t)z(t)}{x(t)} \end{smallmatrix}\right]$, where $x : [0,\tilde{\delta}] \to \R$, $y : [0,\tilde{\delta}] \to \R$, and $z : [0,\tilde{\delta}] \to \R$ are continuously differentiable, and such that $x(0) = \sigma$, $y(0) = z(0) = 0$, $\dot{x}(0) = -\xi_1$, $\dot{y}(0) = \dot{z}(0) = \xi_2$, and $x([0,\tilde{\delta}]) \subseteq [\frac{1}{2}\sigma, \frac{3}{2}\sigma]$.
Let $\hat{f} : [0,\tilde{\delta}] \to \R : t \mapsto f(\tilde{R}_X(t\xi))$. Then, $\hat{f} = \frac{yz}{x}$ and $\dot{\hat{f}} = \frac{y\dot{z}}{x} + \frac{\dot{y}z}{x} - \frac{\dot{x}yz}{x^2}$. By continuity, there exists $\delta \in (0,\tilde{\delta}]$ such that $\dot{y}([0,\delta]), \dot{z}([0,\delta]) \subseteq [\frac{1}{2}\xi_2, \frac{3}{2}\xi_2]$, and $-\dot{x}([0,\delta]) \subseteq [\frac{1}{2}\xi_1, \frac{3}{2}\xi_1]$. Thus, since $\dot{\hat{f}}(0) = 0$, for all $t \in (0,\delta]$,
\begin{equation*}
\frac{\big|\dot{\hat{f}}(t)-\dot{\hat{f}}(0)\big|}{t}
= \frac{\dot{\hat{f}}(t)}{t}
\ge \frac{\xi_2^2}{3\sigma} + \frac{\xi_1\xi_2^2}{18\sigma^2}t
> \frac{\xi_2^2}{3\sigma}
> \beta_\mathrm{RL},
\end{equation*}
which completes the proof.
\end{proof}

\section{Conclusion}
\label{sec:Conclusion}
We close this work with five concluding remarks.
\begin{enumerate}
\item As in \cite{LevinKileelBoumal2022}, the analysis conducted in Sections~\ref{sec:ProposedAlgorithmConvergenceAnalysis} and \ref{subsec:ConvergenceAnalysisRFDR} remains true if $f$ is only defined on an open subset of $\mathcal{E}$ containing $C$.

\item To the best of our knowledge, Assumption~\ref{assumption:Stratification} is new in the numerical optimization literature. We have drawn some links between this assumption and known concepts of variational analysis and stratification theory in Section~\ref{sec:ComplementaryResults}. It is desirable to pursue this investigation, notably by focusing on the possible links between Assumption~\ref{assumption:Stratification} and concepts of algebraic geometry or stratification theory such as the Whitney conditions \cite{Whitney1965}. If successful, this investigation may offer necessary or sufficient conditions for Assumption~\ref{assumption:Stratification} to be satisfied or answer open questions such as the following.
In Assumption~\ref{assumption:Stratification}, does condition~1 imply condition~3?
Are the second, fourth, and fifth statements of Theorem~\ref{thm:ExamplesStratifiedSetsSatisfyingMainAssumption} general facts for sets $C$ satisfying Assumption~\ref{assumption:Stratification}?

\item The range of values of $j$ in the for-loop of the $\ppgdr$ map (Algorithm~\ref{algo:P2GDRmap}) can be as large as $\{0, \dots, p\}$, and there are situations where this occurs each time the for-loop is reached (e.g., in the case of a bounded sublevel set, when $\Delta$ is chosen so large that $\dist(x, S_0) \le \Delta$ for all $x$ in the sublevel set).
One can thus wonder whether it is possible to restrict (conditionally or not) the range of values of $j$ while still accumulating at stationary points. This is an open question. It seems unlikely that $\ppgdr$ with a restricted for-loop can be analyzed along the lines of Section~\ref{sec:ProposedAlgorithmConvergenceAnalysis}. Indeed, as pointed out in Section~\ref{subsec:AssumptionsFeasibleSet}, if $C$ is $\R_{\le s}^n \cap \R_+^n$, $\R_{\le r}^{m \times n}$, or $\mathrm{S}_{\le r}^+(n)$, then the function $u$ defined in Assumption~\ref{assumption:GlobalSecondOrderUpperBoundDistanceFromTangentLine} is not locally bounded at any point of $C \setminus S_p$ and the function $\s(\cdot; f, C)$ defined in~\eqref{eq:StationarityMeasure} is not necessarily lower semicontinuous on $C \setminus S_p$. Two remarks on the case where $C = \R_{\le r}^{m \times n}$ should be added. First, \eqref{eq:NormProjectionTangentConeDeterminantalVariety} can compensate for the discontinuity of $\s(\cdot; f, \R_{\le r}^{m \times n})$, as explained after the proof of Proposition~\ref{prop:P2GDRmapRealDeterminantalVarietyPolak}. Second, should the answer to the open question be negative, a counterexample other than the one of~\cite[\S 2.2]{LevinKileelBoumal2022} would be required in view of~\cite[Remark~2.11]{LevinKileelBoumal2022}.

\item The preceding remark should be tempered by the following observation. In many practical situations, when $\Delta$ is chosen reasonably small, the distance between almost every iterate and its lower stratum is larger than $\Delta$. For such an iterate, the range of values of $j$ in the for-loop of the $\ppgdr$ map reduces to $\{0\}$, and the $\ppgdr$ map generates the same point as the $\ppgd$ map. In this scenario, the only computational overhead in $\ppgdr$ is the computation of the distance to the lower stratum. If $C$ is one of the three sets studied in Theorem~\ref{thm:ExamplesStratifiedSetsSatisfyingMainAssumption}, then, in view of line~\ref{algo:P2GDmap:LineSearch} of Algorithm~\ref{algo:P2GDmap}, it is reasonable to assume that every iterate has been obtained by a truncated SVD if $C$ is $\R_{\le r}^{m \times n}$ or $\mathrm{S}_{\le r}^+(n)$ and by Algorithm~\ref{algo:ProjectionNonnegativeSparseVectors} if $C = \sparse{n}{s} \cap \R_+^n$, in which case the distance to the lower stratum is immediately available, making the overhead insignificant. In summary, if $C$ is one of the three sets studied in Theorem~\ref{thm:ExamplesStratifiedSetsSatisfyingMainAssumption}, then $\ppgdr$ offers stronger convergence properties than $\ppgd$, and while incurring an insignificant overhead in many practical situations.

However, if $\ppgd$ follows an apocalypse, then there is at least one iterate for which the range of values of $j$ in the for-loop of the $\ppgdr$ map does not reduce to $\{0\}$. It seems that this is the price to pay for this first-order algorithm to accumulate at stationary points of~\eqref{eq:OptiProblem}.

\item The comparison of the six algorithms listed in Table~\ref{tab:IndexAlgorithmsRealDeterminantalVariety} conducted in Section~\ref{subsubsec:ComparisonSixOptimizationAlgorithmsRealDeterminantalVariety} for $C = \R_{\le r}^{m \times n}$ can be summarized as follows.
$\pgd$ requires a large scale truncated SVD at each iteration where $\nabla f$ does not have low-rank, and it is not known whether it can converge to a Mordukhovich stationary point of~\eqref{eq:OptiProblem} that is not stationary.
$\ppgd$ and $\rfd$ are not apocalypse-free.
$\ppgdr$ and $\rfdr$ are compared in the last paragraph of Section~\ref{subsec:RFDRreview}.
Thus, it remains to compare $\ppgdr$ and \cite[Algorithm~1]{LevinKileelBoumal2022}.
First, $\ppgdr$ requires only first-order information about $f$ while \cite[Algorithm~1]{LevinKileelBoumal2022} requires second-order information.
Second, Table~\ref{tab:ComparisonComputationalCostPerIterationAlgorithmsRealDeterminantalVariety} shows that every iteration of \cite[Algorithm~1]{LevinKileelBoumal2022} requires one large scale truncated SVD while, in the worst case, an iteration of $\ppgdr$ requires $r$ large scale truncated SVDs. However, as pointed out in the fourth remark, in many practical situations, a typical iteration of $\ppgdr$ requires no large scale (truncated) SVD.
Third, as explained in Section~\ref{subsubsec:ComparisonSixOptimizationAlgorithmsRealDeterminantalVariety}, no upper bound on the number of iterations needed to satisfy the stopping criterion defined by \eqref{eq:StoppingCriterion}, i.e., to bring $\s(\cdot; f, \R_{\le r}^{m \times n})$ below some threshold $\varepsilon \in (0, \infty)$, is available for $\ppgdr$ or \cite[Algorithm~1]{LevinKileelBoumal2022}, although the latter enjoys the guarantees given in \cite[Theorems~3.4 and 3.16]{LevinKileelBoumal2022}.
Fourth, on the numerical experiment described in Section~\ref{subsubsec:LKB22instance}, $\ppgdr$ converges much faster than \cite[Algorithm~1]{LevinKileelBoumal2022}.
\end{enumerate}

\section*{Acknowledgements}
This paper benefited from useful discussions with Laurent Jacques and Simon Vary.

\appendix
\section{The gradient and Hessian of a real-valued function on a Hilbert space}
\label{sec:GradientHessianRealValuedFunctionOnHilbertSpace}
Surprisingly, we did not find in the literature a proper introduction of the concept of Hessian for a real-valued function defined on some nonempty open subset of a real Hilbert space. As the Hessian and, in particular, its eigenvalues are needed in second-order optimization, notably in Section~\ref{subsubsec:PracticalImplementationLKB22Algo1RealDeterminantalVariety}, this section provides such an introduction. Although the finite-dimensional case suffices for this work, we make no assumption on the dimension.

Multilinear mappings and, in particular, multilinear forms are defined, e.g., in \cite[(A.6.1)]{Dieudonne}. Let $X_1$, \dots, $X_n$, and $Y$ be real normed spaces. A necessary and sufficient condition for a multilinear mapping of $X_1 \times \dots \times X_n$ into $Y$ to be continuous is given in \cite[(5.5.1)]{Dieudonne}. The real normed space of all continuous multilinear mappings of $X_1 \times \dots \times X_n$ into $Y$ is denoted by $\mathcal{L}(X_1, \dots, X_n; Y)$ \cite[\S 5.7]{Dieudonne} and simply by $\mathcal{L}_n(X, Y)$ if $X_i = X$ for all $i \in \{1, \dots, n\}$; it is complete if $Y$ is complete \cite[(5.7.3)]{Dieudonne}. If $X$ is a real normed space, then the Banach space $\mathcal{L}(X, \R)$ is called the \emph{dual space} of $X$ and denoted by $X^*$ \cite[\S 1.1]{Brezis}; an element of $X^*$ is called a continuous linear form \cite[\S 5.8]{Dieudonne} or a continuous linear functional \cite[\S 1.1]{Brezis} on $X$. The Riesz--Fréchet representation theorem (see, e.g., \cite[(6.3.2)]{Dieudonne} or \cite[Theorem~5.5]{Brezis}) enables to identify $X^*$ with $X$. Moreover, a consequence of this theorem (see, e.g., \cite[Theorem~7.19]{Ovchinnikov}) enables to identify the Banach space $\mathcal{L}_2(X, \R)$ of all continuous bilinear forms on $X$ with $\mathcal{L}(X) := \mathcal{L}(X, X)$. The elements of $\mathcal{L}(X)$ are called \emph{continuous linear operators} on $X$.

Let $X$ be a real normed space, $U$ be a nonempty open subset of $X$, and $f$ be a real-valued function defined on $U$. For each $p \in \N \setminus \{0\}$, if $f$ is $p$ times differentiable in $U$, then the $p$th derivative of $f$ is $f^{(p)} : U \to \mathcal{L}_p(X, \R)$ and, by \cite[(8.12.4)]{Dieudonne}, for each $x \in U$, the $p$-linear form $f^{(p)}(x)$ is symmetric. 

Assume now that $X$ is a Hilbert space and that $f$ is two times differentiable.
By the Riesz--Fréchet representation theorem, for each $x \in U$, there exists a unique $g_x \in X$ such that, for all $z \in X$, $f'(x)(z) = \ip{g_x}{z}$; $g_x$ is called the \emph{gradient} of $f$ at $x$ and denoted by $\nabla f(x)$. The gradient of $f$ is $\nabla f : U \to X : x \mapsto \nabla f(x)$.
By \cite[Theorem~7.19]{Ovchinnikov}, for each $x \in U$, there exists a unique $H_x \in \mathcal{L}(X)$ such that, for all $u, v \in X$, $f''(x)(u, v) = \ip{H_x(u)}{v}$; $H_x$ is called the \emph{Hessian} of $f$ at $x$ and denoted by $\nabla^2 f(x)$. The Hessian of $f$ is $\nabla^2 f : U \to \mathcal{L}(X) : x \mapsto \nabla^2 f(x)$. For each $x \in U$, since the bilinear form $f''(x)$ is symmetric, the linear operator $\nabla^2 f(x)$ is \emph{self-adjoint}, i.e., $\ip{\nabla^2 f(x)(u)}{v} = \ip{u}{\nabla^2 f(x)(v)}$ for all $u, v \in X$ (see \cite[\S 11.5]{Dieudonne} or \cite[\S 6.4]{Brezis} for the definition).

\begin{proposition}
Let $f : U \to \R$ be differentiable. For each $x \in U$, $f' : U \to X^*$ is differentiable at $x$ if and only if $\nabla f : U \to X$ is differentiable at $x$, in which case $(\nabla f)'(x) = \nabla^2 f(x)$.
\end{proposition}

\begin{proof}
Let $L \in \mathcal{L}(X)$ be associated with $b \in \mathcal{L}_2(X, \R)$. For all $u \in X$ such that $x+u \in U$,
\begin{align*}
\norm{\nabla f(x+u)-\nabla f(x)-L(u)}_X
&= \sup_{v \in B_X[0, 1]} \left|\ip{\nabla f(x+u)-\nabla f(x)-L(u)}{v}\right|\\
&= \sup_{v \in B_X[0, 1]} \left|\ip{\nabla f(x+u)}{v}-\ip{\nabla f(x)}{v}-\ip{L(u)}{v}\right|\\
&= \sup_{v \in B_X[0, 1]} \left|f'(x+u)(v)-f'(x)(v)-b(u, v)\right|\\
&= \sup_{v \in B_X[0, 1]} \left|(f'(x+u)-f'(x)-b(u, \cdot))(v)\right|\\
&= \norm{f'(x+u)-f'(x)-b(u, \cdot)}_{\mathcal{L}(X)},
\end{align*}
where the first equality follows from the Cauchy--Schwarz inequality.
If $f'$ is differentiable at $x$, then, by taking $b := f''(x)$ and $\nabla^2 f(x) := L$, we have $(\nabla f)'(x) = \nabla^2 f(x)$. Conversely, if $\nabla f$ is differentiable at $x$, then, by choosing $L := \nabla^2 f(x) := (\nabla f)'(x)$, we have $f''(x) = b$. In both cases, since $f''(x)$ is symmetric, $\nabla^2 f(x)$ is self-adjoint. 
\end{proof}

If $X$ is finite-dimensional, then the eigenvalues of $\nabla^2 f(x) \in \mathcal{L}(X)$ are the eigenvalues of the $\dim X \times \dim X$ matrix representing $\nabla^2 f(x)$ with respect to any basis of $X$; this matrix is symmetric since $\nabla^2 f(x)$ is self-adjoint, and its eigenvalues are therefore real.

\bibliographystyle{plainurl}
\bibliography{golikier_bib}
\end{document}